\documentclass[11pt,reqno]{amsart}
\usepackage[margin=1in,letterpaper]{geometry}
\usepackage{graphicx}
\usepackage{amssymb}
\usepackage{amsthm}
\usepackage{mathrsfs}
\usepackage{stmaryrd}
\usepackage{accents} 
\usepackage{enumitem} 
\usepackage{subcaption}
\usepackage{microtype}
\usepackage{colonequals}
\usepackage{textcomp} % prevent gensymb below from outputting useless warnings
\usepackage{gensymb}
\usepackage{resmes}

\usepackage[bookmarksopen,bookmarksdepth=2]{hyperref} 
\allowdisplaybreaks

\makeatletter\let\over\@@over\makeatother

\numberwithin{equation}{section}
\theoremstyle{plain} 
\newtheorem{theorem}{Theorem}[section] 
 
\newtheorem{corollary}[theorem]{Corollary}
\newtheorem{lemma}[theorem]{Lemma} 
\theoremstyle{remark}
\newtheorem{remark}[theorem]{Remark}
\theoremstyle{definition}
\newtheorem{definition}[theorem]{Definition}

\newcommand{\be}{\begin{equation}}
\newcommand{\ee}{\end{equation}}%
\newcommand{\bse}{\begin{subequations}}
\newcommand{\ese}{\end{subequations}}

\newcommand{\supp}[1]{\operatorname{supp}{#1}}
 
\newcommand{\jump}[1]{\left\llbracket{#1}\right\rrbracket}

\newcommand{\id}{\operatorname{id}}

\newcommand{\inter}{\operatorname{int}}

\newcommand{\R}{\mathbb{R}} 
\newcommand{\placeholder}{\,\cdot\,}
         
\newcommand{\by}{\times}

\newcommand{\abs}[2][]{#1\lvert #2 #1\rvert}

\newcommand{\ina}{\textup{~in~}}
\newcommand{\ona}{\textup{~on~}}

\newcommand{\bdd}{\mathrm{b}}       
\newcommand{\loc}{{\mathrm{loc}} }

\newcommand{\Stokesu}{{u_{\textup{S}}}}

\newcommand{\cm}{{\mathscr C}}

\newcommand\fluidD{\mathscr{D}}
\newcommand\fluidS{\mathscr{S}}
\newcommand{\fluidB}{\mathscr{B}}
\newcommand{\fluidT}{\mathscr{T}}
\newcommand{\extD}{\mathscr{R}}

\newcommand\Lip{\operatorname{Lip}}

\newcommand{\gc}{\mathrm{gc}}

\newcommand{\cme}{\cm_{\textup{elev}}}
\newcommand{\cmd}{\cm_{\textup{depr}}}

\begin{document}

\title[Extreme internal waves]{Extreme internal waves: gravity currents and overturning fronts}

\date{\today}

\author[R. M. Chen]{Robin Ming Chen}
\address{Department of Mathematics, University of Pittsburgh, Pittsburgh, PA 15260} 
\email{mingchen@pitt.edu}  

\author[S. Walsh]{Samuel Walsh}
\address{Department of Mathematics, University of Missouri, Columbia, MO 65211} 
\email{walshsa@missouri.edu} 

\author[M. H. Wheeler]{Miles H. Wheeler}
\address{Department of Mathematical Sciences, University of Bath, Bath BA2 7AY, United Kingdom}
\email{mw2319@bath.ac.uk}

\begin{abstract}
Hydrodynamic bores are front-type traveling wave solutions to the two-layer free boundary Euler equations in two dimensions. The velocity field in each layer is assumed to be incompressible and irrotational, and it limits to distinct laminar flows upstream and downstream. Rigid horizontal boundaries confine the fluids from above and below. A constant gravitational force acts on the waves, but surface tension is neglected.  It was recently shown by the authors~\cite{chen2023global} that there exist two large-amplitude families of hydrodynamic bores: a curve of depression bores $\cmd$ and a curve of elevation bores $\cme$. 
%Both bifurcate from the trivial solution where the interface is flat. 

We now prove that in the limit along $\cme$, the solutions must overturn: the interface separating the layers develops a vertical tangent. This type of behavior was first observed over $45$ years ago in numerical computations of internal gravity waves and gravity water waves with vorticity. Despite considerable progress over the past decade in constructing families of water waves that potentially overturn, a proof that overturning definitively occurs has been stubbornly elusive.  We further show that in the limit along $\cmd$, either overturning occurs or the solutions converge to a gravity current: the free boundary contacts the upper wall and the relative velocity in the upper fluid is stagnant. 
%These configurations have been the subject of many studies in the applied literature, but no exact solutions have previously been obtained. 
We also determine the contact angle between the interface and the rigid barrier for the limiting gravity current, giving the first rigorous confirmation of a conjecture of von Kármán.

The resolutions of these questions in the specific case of hydrodynamic bores is accomplished through the use of novel geometric analysis techniques, including bounds on the decay of the velocity field near a hypothetical double stagnation point. These ideas may have broader applications to bifurcation theoretic studies of large-amplitude waves.

\end{abstract}

% Please keep this here! Moving it above the abstract throws an error. --Miles
\maketitle

\setcounter{tocdepth}{1}
\tableofcontents

\section{Introduction}

The cresting and collapsing of water waves is a dramatic if familiar sight in nature. Mathematically, though, water waves are typically modeled by the incompressible free boundary Euler equations, and, in the presence of vorticity, this system allows for \emph{steady overhanging waves}: classical solutions that do not break but simply translate at a fixed velocity remaining overturned for all time. Numerical evidence for overhanging gravity water waves goes back to the late 1970s. Holyer~\cite{holyer1979large} computed overhanging internal waves propagating along an interface dividing two immiscible fluid layers. This was accomplished by making a high-order Fourier series expansion, then solving the resulting recurrence relation for the coefficients via computer algebra. Later, starting at a laminar flow where the interface is perfectly flat and using numerical continuation methods, Meiron and Saffman~\cite{meiron1983overhanging}, Pullin and Grimshaw \cite{pullin1988finite}, and Turner and Vanden-Broeck~\cite{turner1988broadening} likewise obtained overhanging {internal waves}, while Teles da Silva and Peregrine~\cite{dasilva1988steep} computed overhanging surface water waves with constant vorticity.  Investigations in this vein have continued to the present; see, for example, \cite{dias2003internal,maklakov2018almost,maklakov2020note,hur2020exact,guan2021local,hur2022overhanging}. Reproducing these findings rigorously is among the largest open problems in water waves. The appeal of overhanging waves is not primarily due to physical relevance, as they would necessarily be highly unstable, but rather that the existence of such strikingly nonlinear and counterintuitive solutions vividly illustrates the richness of the steady water wave problem itself. 

In a breakthrough paper, Constantin, Vărvăruca, and Strauss~\cite{constantin2016global} constructed families of periodic waves with constant vorticity in a framework that admits overturning. This was improved to general vorticity by Wahlén and Weber~\cite{wahlen2024large}. Haziot~\cite{haziot2021stratified} proved an analogous result for waves with linear density stratification, while Haziot and Wheeler~\cite{haziot2023large} treated the case of solitary waves with constant vorticity. Finally, solitary waves carrying point vortices or hollow vortices were constructed by Chen, Varholm, Walsh, and Wheeler~\cite{chen2025vortex}. Numerical results~\cite{dyachenko2019folds,dyachenko2019stokes,chen2025vortex} suggests that many of these solutions do indeed overturn, but a rigorous proof has not yet been achieved. The central issue is that, while global bifurcation theory is a powerful tool for constructing large continua of solutions, it must be supplemented with highly nontrivial qualitative theory in order to fully determine the limiting behavior. It is extremely difficult to say whether a particular global branch of solutions will exhibit overturning, say, or first develop a surface singularity such as a corner. 

With that in mind, another set of papers have adopted the opposite approach: instead of beginning at a laminar flow and using global bifurcation, they construct overhanging waves directly. One tack has been to take an explicit gravity-less overhanging wave and introduce perturbative gravity~\cite{akers2014gravity,cordoba2016existence,hur2020exact,hur2022overhanging,carvalho2023gravity}. The task is then to fully understand the linearized steady water wave problem at an overhanging wave, which is certainly harder than working near laminar flows, but avoids any global continuation. Quite recently, Dávila, del Pino, Musso, and Wheeler~\cite{davila2024overhanging} gave a gluing method construction of overhanging solitary waves with small gravity. Like those predicted in~\cite{dasilva1988steep}, the fluid domain consists of a nearly circular constant vorticity region that is perched on top of a body of nearly laminar flow. An exceedingly delicate fixed point argument is used to match the velocity fields together in a thin neck region. Collectively, these results show that overhanging waves exist, but they leave unresolved some interesting questions. First, they all work within the weak gravity regime, whereas the numerics suggest overhanging waves are possible even with $O(1)$ gravity. 
% Miles: Maybe not here, but could be interesting to report some of the dimensionless parameter values for which overhanging is observed.
Moreover, these constructions cleverly sidestep the need to discern overturning along a solution branch. But overcoming this problem directly is arguably just as important as establishing the existence of overhanging waves as it represents a fundamental obstacle to global bifurcation theoretic methods more broadly.

In this work, we study two familes of large-amplitude front-type solutions to the internal wave problem. These are called (smooth) \emph{hydrodynamic bores}, and were obtained via global bifurcation theory in~\cite{chen2023global}. One curve $\cme$ consists of elevation bores, for which the height of the internal interface is minimized upstream, and the other curve $\cmd$ is of depression bores, for which the height is maximized upstream. Earlier numerics by Dias and Vanden-Broeck~\cite{dias2003internal} indicate that overturning invariably occurs as one follows  $\cme$.  We are now able to give a rigorous proof. The original paper~\cite{chen2023global} left open two possibilities: overturning or the development of a ``double stagnation'' degeneracy wherein the relative velocity field in both fluids tends to $0$ at some common point on the interface.  Using ideas from the geometric analysis of free boundary elliptic equations, most notably an estimate on the decay of the velocity near such surface stagnation points, the degeneracy alternative can be excluded, leaving only overturning. While there is a natural connection between free boundary regularity techniques and the qualitative theory of water waves, historically the two fields have been largely isolated. In part, this is because in the neighborhood of a surface stagnation point, the gravity water wave problem is beyond the reach of much of the classical literature on Bernoulli free boundary problems, though there are several notable recent exceptions~\cite{shargorodsky2008bernoulli,varvaruca2011geometric,kriventsov2025minmax}. 
% Miles: Should cite https://arxiv.org/abs/2508.12159

For the depression bore family $\cmd$, the numerics in~\cite{dias2003internal} instead suggest that the interface approaches the upper wall. This coincides with the entire upper fluid region becoming stagnant --- that is, it appears stationary in a frame moving with the front --- corresponding to a so-called gravity current configuration. These have been the subject of a large number of works in the applied fluid mechanics literature, yet no exact solutions have been constructed previously. Gravity currents featured in von Kármán's~\cite{vonkarman1940engineer} influential lecture on nonlinear phenomena with salience to engineering. He asserted that, if the lighter fluid has a stagnation point where the interface meets the wall, then the contact angle must be precisely $60\degree$. In that connection, we prove that following $\cmd$, either overturning occurs or else there is a limiting gravity current solution. Via geometric methods, we give a rigorous proof of von Kármán's conjectured contact angle for these waves, as well as more general solutions to the gravity current problem.

\subsection{The internal wave problem}

\begin{figure}
	\centering
	\includegraphics{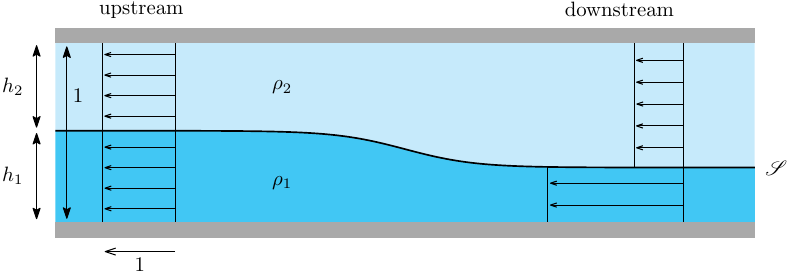}
	\caption{A steady hydrodynamic bore viewed in a frame moving with the front. The channel extends from $y=-h_1$ to $y = h_2$, with the heavier fluid layer (shaded in darker blue) $\fluidD_2$ lying below the lighter fluid layer $\fluidD_1$. Both upstream and downstream, the limiting velocity field is purely horizontal. Note that the upstream velocity is the same in both layers, here normalized to $(-1,0)$, whereas there will be a jump in the downstream velocity.}
	\label{bore problem setup figure}
\end{figure}

In density stratified bodies of water, it is typical to have large regions of nearly constant density separated by pycnoclines or thermoclines, which are thinner regions where the density varies rapidly. A common model is to replaces the pycnoclines with sharp material interfaces that now divide immiscible constant density fluid layers. \emph{Internal waves} are traveling waves that propagate along these interfaces. 

Consider a stably stratified two-layer configuration in which a lighter fluid with constant density $\rho_2 > 0$ lies atop a heavier fluid with constant density $\rho_1 > \rho_2$. In the ocean, for instance, the densities of the two layers are often quite close, so it is desirable to also study the so-called \emph{Boussinesq limit} where $\rho_2 \to \rho_1$. Working in a reference frame moving with the wave, we assume that the interface $\fluidS$ between the two layers, as well as the fluid velocity fields, are independent of time. Both for simplicity and because it is the setting with the most applied interest, suppose that the fluid velocity tends to some constant value $(-c,0)$ in the upstream limit $x \to -\infty$, where here $c>0$ is interpreted as the wave speed. Letting $h_1,h_2 > 0$ be the upstream thicknesses of the two layers and $g > 0$ be the constant acceleration due to gravity, the dimensionless Froude number
\begin{equation}
  \label{definition Froude number}
  F \colonequals \frac c{\sqrt{g(h_1+h_2)}} 
\end{equation}
measures the relative importance of inertial and gravitational effects. Switching to dimensionless units with $h_1+h_2$ as the length scale and $c$ as the velocity scale, we are left with the three dimensionless parameters $F$, $\rho_2/\rho_1$, and $h_1$. 

\emph{Bores} are front-type solutions to this system, meaning that the height of each fluid layer in the downstream limit $x \to +\infty$ is distinct from the upstream limit. Through a so-called conjugate flow analysis, it can be shown that there is a unique value of the Froude number for which bores are possible:
\begin{equation}
  \label{definition front Froude number}
  F^2 = \frac{\sqrt{\rho_1}-\sqrt{\rho_2}}{\sqrt{\rho_1}+\sqrt{\rho_2}};
\end{equation}
see, for example, \cite[Appendix A]{Laget1997interfacial} or \cite[Lemma 5.5]{chen2023global}.  For the time being, we consider the case where the interface $\fluidS$ between the layers can be parameterized globally as the graph of a single-valued function $\eta = \eta(x)$. We adopt coordinates so that $\eta \to 0$ as $x \to -\infty$. The bottom boundary of the channel is therefore $y=-h_1$, and the upper boundary is $y=h_2=1-h_1$.  We denote the upper fluid domain in these variables by $\fluidD_2$, the lower fluid by $\fluidD_1$, and set $\fluidD \colonequals \fluidD_1 \cup \fluidD_2$. See Figure~\ref{bore problem setup figure} for an illustration.

\begin{subequations}\label{eqn:stream}
  Requiring the flow to be irrotational and incompressible in each layer, the Euler equations can be expressed in terms of a \emph{(pseudo) stream function} $\psi$ satisfying
  \begin{equation} \label{eqn:psi harmonic}
    \Delta \psi = 0 \quad \text{in } \fluidD.
  \end{equation}
  The pseudo stream function is constant along the interface as it is a material surface, and we normalize this constant to be zero:
  \begin{align}
    \label{eqn:stream:kinint}
    \psi = 0 \qquad \text{on } \fluidS.
  \end{align}
  In particular, $\psi$ is continuous across the interface. Our assumptions upstream can be written as
  \begin{align}
    \label{eqn:stream:asym}
    \nabla\psi_i \to (0,-\sqrt{\rho_i}),
    \quad 
    \eta \to 0
    \qquad \text{ as } x \to -\infty.
  \end{align}
  The rigid boundaries $y=-h_1$ and $y=h_2$ are also material surfaces, and hence level curves of the stream function. Calculating these values using \eqref{eqn:stream:kinint} and \eqref{eqn:stream:asym}, we find
  \begin{align}
    \begin{alignedat}{2}
      \psi &=  h_1 \sqrt{\rho_1} &\quad& \text{ on } y=-h_1,  \\
      \psi &= -  h_2 \sqrt{\rho_2} &\quad& \text{ on } y=h_2. 
    \end{alignedat}
    \label{psi value on top}
  \end{align}
  Finally, the dynamic boundary condition on $y=\eta(x)$ asserts the continuity of the pressure. Evaluating the pressure using Bernoulli's law leads to the fully nonlinear boundary condition
  \begin{equation}
    \label{eqn:stream:dynamic}
     \abs{\nabla\psi_2}^2 -  \abs{\nabla\psi_1}^2
    + \frac {2(\rho_2-\rho_1)}{F^2} y =  {\rho_2-\rho_1}  \qquad  \ona \fluidS.
  \end{equation}
For a general solution of the internal wave problem, the constant on the right-hand side above is an unknown. Here, however, it is determined completely by the upstream state~\eqref{eqn:stream:asym}. Note that while we have assumed the densities are distinct, in the {Boussinesq limit} $\rho_2 \to \rho_1$, taking $F$ to be the unique Froude number \eqref{definition front Froude number} admitting fronts, one finds that the dynamic condition~\eqref{eqn:stream:dynamic} likewise becomes
  \begin{equation}
  \label{eqn:stream:dynamic-boussinesq}
  \tag{\ref*{eqn:stream:dynamic}$'$}
  	 |\nabla \psi_2|^2 -  |\nabla \psi_1|^2 - 8 \rho_1 y = 0 \qquad \textrm{on } \fluidS.
  \end{equation}
\end{subequations}
Therefore, when we refer to the Boussinesq limit of the internal wave problem~\eqref{eqn:stream}, we mean the system that results from setting $\rho_1 = \rho_2$ and replacing the dynamic condition~\eqref{eqn:stream:dynamic} by \eqref{eqn:stream:dynamic-boussinesq}.

For any $\alpha \in (0,1)$, there exists classical solutions $(\psi,\eta,h_1)$ to the internal front problem~\eqref{eqn:stream} enjoying the Hölder regularity
\begin{equation}
  \label{definition classical regularity internal wave}
  \psi \in C_\bdd^{2+\alpha}(\overline{\fluidD_1}) \cap C_\bdd^{2+\alpha}(\overline{\fluidD_2}) \cap C_\bdd^0(\overline{\fluidD}), \qquad \eta \in C_\bdd^{2+\alpha}(\mathbb{R}).
\end{equation}
Small-amplitude bores (that is, with $\| \eta \|_{C^{2+\alpha}} \ll 1$) were originally obtained by Amick and Turner~\cite{amick1989small}, and then later through different methods by Makarenko~\cite{makarenko1992bore,makarenko1999conjugate}, Mielke~\cite{mielke1995homoclinic}, and Chen, Walsh, and Wheeler~\cite{chen2022center}. We also mentioned that viscous bores of small amplitude (in a single fluid) were quite recently studied by Stevenson and Tice~\cite{stevenson2025gravity}. Families of bores extending to large amplitude were first constructed in~\cite{chen2023global}, the main result of which is recalled in Section~\ref{preliminaries existence section}. Our purpose in the present paper is to understand the limiting form of the bores along these families. With that in mind, we will later consider a broader notion of solution to~\eqref{eqn:stream}, namely \emph{domain variational solutions}, that allows for more singular behavior. We postpone their definition until Section~\ref{model problem section}.

\subsection{Gravity currents and cavity flows}
\label{intro gravity current section}

\begin{figure}
	\centering
	\includegraphics{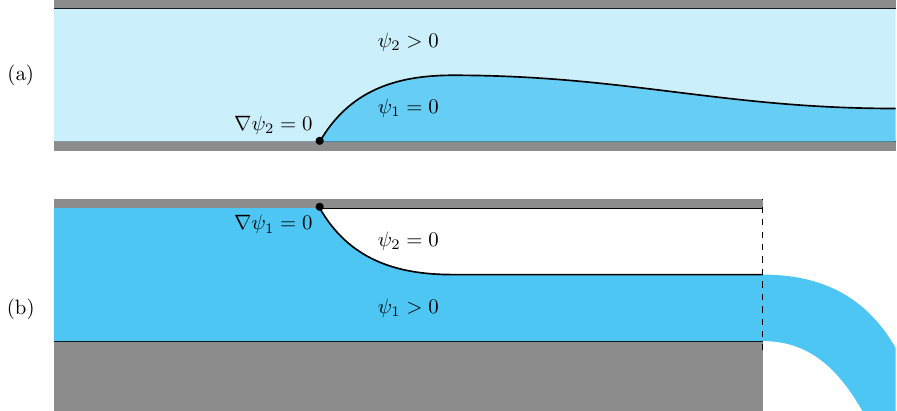}
	\caption{(a) The gravity current model proposed by von Kármán~\cite{vonkarman1940engineer}. A heavier fluid (shaded darker blue) intrudes on a lighter fluid (shaded light blue). The velocity in the heavy fluid is assumed to be constant. In the moving frame, the heavy fluid appears to be at rest, while the lighter fluid flows over it. The free boundary meets the bed at a stagnation point (in both phases). (b) Benjamin's~\cite{benjamin1968gravity} cavity front model. There is a single phase, which is emptying out of the channel far downstream.} 
	\label{gravity current figure}
\end{figure}

When a heavier fluid intrudes into a lighter fluid occupying a rigid channel, it can produce a traveling front called a \emph{gravity current}.  This phenomenon occurs in a remarkably large number of applications, including such varied examples as billowing clouds following volcanic eruptions \cite{moore1969nuees}, river plumes \cite{garvine1974dynamics}, and oil spillage in the ocean \cite{hoult1972oil}; see the survey \cite{simpson1982gravity}.  In 1940, von K\'arm\'an \cite{vonkarman1940engineer} formally analyzed a model for gravity currents consisting of two superposed layers of immiscible perfect fluids.  Evoking the classical conjecture of Stokes \cite{stokes1880theory} for periodic water waves beneath vacuum, he reasoned that if there was a stagnation point where the interface meets the wall, then the contact angle there must be precisely $60\degree$. Almost three decades later, Benjamin \cite{benjamin1968gravity} reconsidered the physical validity of von K\'arm\'an's model and gave a second derivation based on conjugate flow analysis and formal power series. He also noted that the same mathematical problem describes \emph{cavity flow}, wherein a lighter density fluid at uniform velocity intrudes into a heavier fluid from above. It is now common to use the term gravity current to refer to both scenarios. See Figure~\ref{gravity current figure} for an illustration.

Benjamin's cavity flow model arises in the formal limit of the internal bore problem~\eqref{eqn:stream} when the interface $\fluidS$ approaches the upper rigid boundary; von Kármán's gravity current on the other hand would arise if the interface met the lower rigid boundary.  We will present the formulation for the cavity flow case first. We also relax the assumption that $\fluidS$ is a graph, and only ask it to be a (locally) rectifiable globally injective curve. Suppose that $\fluidS$ meets the upper boundary at a unique point, which in particular means the upper layer has an upstream depth $h_2 = 0$. This choice is in keeping with the decision to fix the upstream flow in \eqref{eqn:stream:asym}. We may therefore choose the horizontal axes so that contact occurs at the origin:
\begin{subequations}
\label{gravity current problem}
\begin{equation}
  (0,0) \in \partial\fluidS.
\end{equation}
The flow in the lower layer is still taken to be irrotational, while the flow in the upper layer is at constant velocity, hence
\begin{align}
	\label{intro dc laplace}
	\Delta \psi_1 = 0 \quad \textrm{in } \fluidD_1, \qquad \nabla \psi_2 = 0 \quad \textrm{in } \fluidD_2.
\end{align} 
The kinematic conditions on the interface remains the same 
  \begin{align}
    \label{intro gc kinematic S}
    \psi = 0 \qquad \text{on } \fluidS,
  \end{align}
  and setting $h_2 = 0$ and $h_1 = 1$ in \eqref{psi value on top}, we find that on the rigid boundaries 
    \begin{align}
    \begin{alignedat}{2}
      \psi &=   \sqrt{\rho_1} &\quad& \text{ on } y=-1,  \\
      \psi &=  0 &\quad& \text{ on } y= 0. 
    \end{alignedat}
    \label{intro gc kinematic bed and lid}
  \end{align}
  Note that these three equations together force $\psi_2 \equiv 0$. Because the upper layer does not extend upstream, we only impose asymptotic conditions on the flow in the lower layer:
  \[
  	\nabla \psi_1 \longrightarrow (0,-\sqrt{\rho_1}) \qquad \textrm{as } x \to -\infty.
  \]
  Finally, the dynamic or Bernoulli condition on the interface becomes
  \begin{equation}
    \label{intro gc dynamic non-boussinesq}
     \abs{\nabla\psi_1}^2     = \frac {2 (\rho_2-\rho_1)}{F^2} \left( y + Q \right)  \qquad  \ona \fluidS,
  \end{equation}
  in the case $\rho_1 \neq \rho_2$, where $Q$ is the Bernoulli constant. The value of $Q$ can be determined from the asymptotic behavior downstream. For gravity currents that are the limit of solutions to the internal front problem~\eqref{eqn:stream}, from~\eqref{eqn:stream:dynamic} we would have $Q = -F^2/2 < 0$. In the Boussinesq limit $\rho_1 = \rho_2$, the dynamic condition~\eqref{eqn:stream:dynamic-boussinesq} instead takes the form 
  \begin{equation}
    \label{intro gc dynamic boussinesq}
    \tag{\ref*{intro gc dynamic non-boussinesq}$'$}
     |\nabla \psi_1|^2 = -8 \rho_1 y  \qquad \textrm{on } \fluidS.
  \end{equation}
\end{subequations}
Note that when $\fluidS$ is merely locally rectifiable, the boundary conditions~\eqref{intro gc dynamic boussinesq} and \eqref{intro gc dynamic non-boussinesq} must be satisfied in some appropriately weak sense. As mentioned above, in this paper we work with domain variational solutions; for the gravity current problem, specifically, see Definition~\ref{vK variational solution definition}.

A key distinction between \eqref{intro gc dynamic non-boussinesq} and \eqref{intro gc dynamic boussinesq} is that for the former, the contact point at the origin \emph{cannot} be a stagnation point for the lower fluid velocity, whereas in the latter Boussinesq limit, it \emph{must be}. One can also consider more general cavity flows or gravity currents where the asymptotic conditions are such that a stagnation point is possible even in the non-Boussinesq case. These, however, are not potential limits of the family of internal fronts we study.
Nonetheless, in Section~\ref{gravity current section}, the formulation of the gravity current problem we use does allow for this scenario and many of our results extend to this more general setting.

In the case of a gravity current where the interface $\fluidS$ meets the lower rigid boundary, the formulation is essentially the same with the roles of the fluid layers reversed. That is, the upstream heights are $h_1 = 0$ and $h_2 = 1$, and the flow in the lower fluid is stagnant. The resulting system is thus
\begin{subequations}
\label{elevation gravity current problem}
\begin{equation}
  \label{elevation gravity current momentum and kinematic}
  \left\{
    \begin{aligned}
      \nabla \psi_1 & = 0 & \qquad & \ina \fluidD_1 \\
      \Delta \psi_2 & = 0 & \qquad & \ina  \fluidD_2 \\
      \psi & = 0 & \qquad & \ona \fluidS \cup \{ y = 0 \} \\
      \psi & = -\sqrt{\rho_2}  & \qquad & \ona y = 1 \\
    \end{aligned}
  \right.
\end{equation}
along with the dynamic condition on $\fluidS$, which in the non-Boussinesq case is
\begin{equation}
  \label{elevation gravity current dynamic non-boussinesq} 
  \abs{\nabla\psi_2}^2    = \frac {2 (\rho_1-\rho_2)}{F^2} \left( y + Q \right)  \qquad \ona \fluidS 
\end{equation}
and, in the Boussinesq limit, becomes 
\begin{equation}
    \label{elevation gravity current dynamic boussinesq}
    \tag{\ref*{elevation gravity current dynamic non-boussinesq}$'$}
     |\nabla \psi_2|^2 = 8 \rho_2 y  \qquad \textrm{on } \fluidS.
  \end{equation}	
\end{subequations}

\subsection{Main results}

\begin{figure}
	\centering
  \includegraphics[page=1]{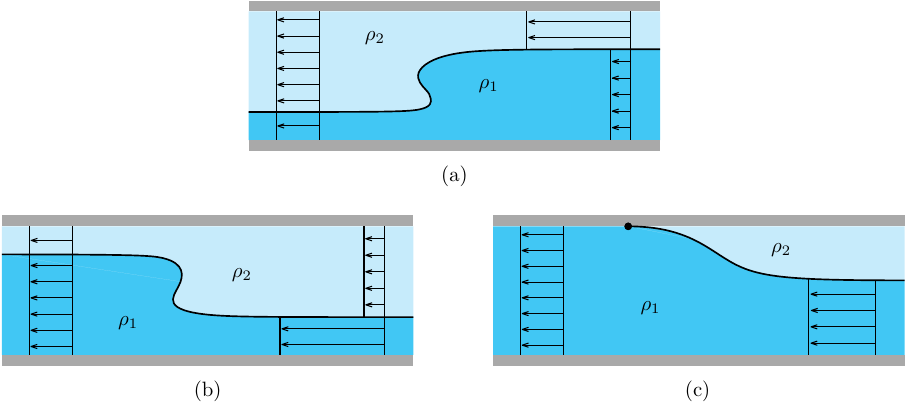}
  \caption{Limiting configurations observed numerically by Dias and Vanden-Broeck~\cite{dias2003internal} and verified in Theorems~\ref{overturning theorem}--\ref{non-boussinesq gravity current theorem} in the non-Boussinesq setting. (a) Bores of elevation $\cme$ invariably overturn. Along the curve of depression bores $\cmd$, either (b) overturning occurs or (c) the free boundary limits to the upper wall, resulting in a gravity current where the free boundary meets the wall tangentially.}
	\label{non-boussinesq bore alternatives figure}
\end{figure}

We now give an informal statement of the contributions of the present paper. Our first result completely resolves the limiting behavior along the family of elevation bores in the non-Boussinesq case.

\begin{theorem}[Overturning]
  \label{overturning theorem}
  For any $0 < \rho_2 \leq \rho_1$, there exist a global $C^0$ curve
  \[ \cme = \left\{ (\psi(s), \eta(s), h_1(s)) : s \in (-\infty,0) \right\}\]
  of classical solutions to the internal front problem~\eqref{eqn:stream} exhibiting the H\"older regularity~\eqref{definition classical regularity internal wave}.  If $\rho_2 < \rho_1$, then in the limit along $\cme$, the interface \emph{overturns}:
  \begin{equation}
    \label{intro overturning}
    \lim_{s \to -\infty} \sup \partial_x \eta(s) = \infty. 
  \end{equation}
\end{theorem}

To the best of our knowledge, this theorem represents the first rigorous proof that overturning definitively occurs along a bifurcation curve of water waves with $O(1)$ gravity. It also verifies the numerical computations of Dias and Vanden-Broeck~\cite{dias2003internal}.
% Miles: Happy to say ``just beyond'' etc. below?
We caution that Theorem~\ref{overturning theorem} only guarantees that the slope of the free boundary tends to $\infty$, it does not assert that there is a limiting overhanging bore. See Figure~\ref{non-boussinesq bore alternatives figure}(a) for an illustration of a bore just beyond the overturning limit \eqref{intro overturning} as predicted by the numerics in~\cite{dias2003internal}. Extending $\cme$ to this regime requires a different formulation of the problem than in~\cite{chen2023global}, which will be carried out in a future work.

For the curve of depression bores in the non-Boussinesq setting, we prove that either overturning occurs or else one can extract a limiting gravity current solution for which the internal interface contacts the upper wall tangentially; see Figure~\ref{non-boussinesq bore alternatives figure}(b,c). Note that in this case, the Bernoulli constant is nonzero, so the tangential intersection agrees with the formal prediction of Chandler and Trinh~\cite{chandler2018complex}.

\begin{theorem}[Gravity current or overturning]
  \label{non-boussinesq gravity current theorem}
For any $0 < \rho_2 \leq \rho_1$, there exist a global $C^0$ curve
  \[ \cmd = \left\{ (\psi(s), \eta(s), h_1(s)) : s \in (0,\infty) \right\}\]
  of classical solutions to the internal front problem~\eqref{eqn:stream} exhibiting the H\"older regularity~\eqref{definition classical regularity internal wave}. If $\rho_2 < \rho_1$, then in the limit along $\cmd$, either the interface \emph{overturns}
  \begin{equation}
    \label{intro overturning depression}
    \lim_{s \to \infty} \inf \partial_x \eta(s)  = -\infty
  \end{equation}
  or there is a limiting \emph{gravity current}: there exists a bounded sequence $\{ x_n \} \subset \mathbb{R}$ and a sequence $s_n \nearrow \infty$ such that, for all $\varepsilon \in (0,1)$,
  \begin{equation}
    \label{depression gravity current limit statement}
    \psi(s_n)(\placeholder-x_n, \placeholder) \xrightarrow{H^1_\loc \cap C_\loc^\varepsilon} \psi, \qquad \eta(s_n)(\placeholder-x_n) \xrightarrow{C_\loc^\varepsilon}  \eta, \qquad h_1(s_n) \longrightarrow 1,
  \end{equation}
  where $(\psi, \fluidS)$ is a nontrivial solution to the gravity current problem~\eqref{gravity current problem} in the domain variational sense with $\fluidS$ being the graph of the Lipschitz continuous function $\eta$. Moreover, the limiting interface meets the wall tangentially. 
  \end{theorem}

Finally, in the Boussinesq setting for both the curve of depression bores $\cmd$ as well as the curve $\cme$, we prove that if overturning does not occur, then one can likewise extract a limiting solution to the gravity current problem in the domain variational sense. However, in this case the interface will make a $60\degree$ contact angle with the rigid boundary as conjectured by von Kármán; see Figure~\ref{boussinesq bore alternatives figure}.

\begin{figure}
	\centering
  \includegraphics[page=2]{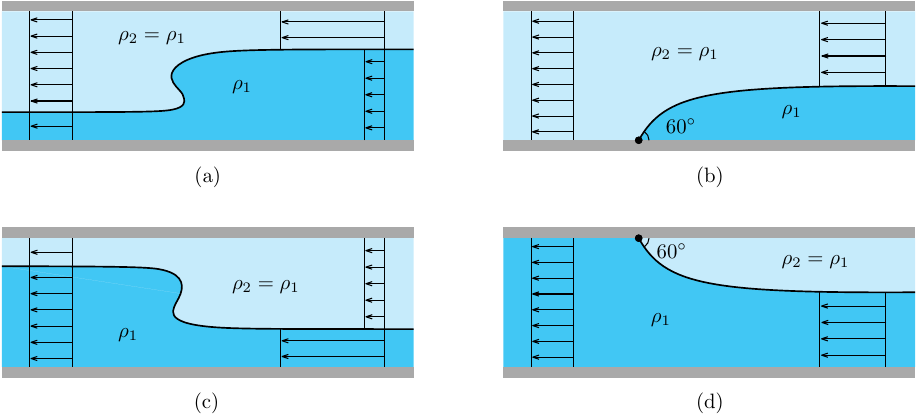}
  \caption{Limiting configurations in Theorem~\ref{boussinesq gravity current theorem} in the Boussinesq setting. Along the curve of elevation bores $\cme$, either (a) overturning occurs or (b) the free boundary limits to the upper wall, resulting in a gravity current where the free boundary makes a contact angle of exactly $60\degree$. The alternatives (c) and (d) for depression bores $\cmd$ are similar. Dias and Vanden-Broeck~\cite{dias2003internal} numerically observe (b) and (d) but not (a) or (c).
  \label{boussinesq bore alternatives figure}}
\end{figure}

\begin{theorem}[Limiting Boussinesq bores] 
  \label{boussinesq gravity current theorem}
Consider the curves $\cme$ and $\cmd$ in the Boussinesq setting $\rho_1 = \rho_2$.
\begin{enumerate}[label=\rm(\alph*\rm)]
	\item \textup{(Elevation)} \label{boussinesq elevation part} 
	In the limit along $\cme$ either overturning~\eqref{intro overturning} occurs or there is a limiting \emph{gravity current}: there is a bounded sequence $\{ x_n \} \subset \mathbb{R}$ and a sequence $s_n \searrow -\infty$ such that, for all $\varepsilon \in (0,1)$,
  \begin{equation}
    \label{elevation gravity current limit statement}
    \psi(s_n)(\placeholder-x_n, \placeholder) \xrightarrow{H^1_\loc \cap C_\loc^\varepsilon} \psi, \qquad \eta(s_n)(\placeholder-x_n) \xrightarrow{C_\loc^\varepsilon}  \eta, \qquad h_1(s_n) \longrightarrow 0,
  \end{equation}
  where $(\psi, \fluidS)$ is a nontrivial solution to the gravity current problem~\eqref{elevation gravity current problem} in the domain variational sense with $\fluidS$ being the graph of the Lipschitz continuous function $\eta$. Moreover, the limiting interface meets the lower wall at a $60\degree$ angle.
  	
	\item \textup{(Depression)} \label{boussinesq depression part}
	Following $\cmd$, either overturning~\eqref{intro overturning depression} occurs or else there is a limiting gravity current~\eqref{depression gravity current limit statement} whose  interface meets the upper wall at a $60\degree$ angle.
\end{enumerate}
\end{theorem}

\begin{remark}
In the numerical results of Dias and Vanden-Broeck~\cite{dias2003internal}, the Boussinesq bore families are observed to always limit to gravity currents. Our method, however, does not allow us to rule out overturning as an alternative.
\end{remark}

In fact, we give a more general result that applies not only to the limiting gravity currents along $\cme$ and $\cmd$, but to a much broader class of solutions to the gravity current problem; see Theorem~\ref{contact angle theorem}. The original argument of von Kármán~\cite{vonkarman1940engineer} essentially adapts Stokes' reasoning. In particular, it is based on using a complex analytic formulation of the problem, and it assumes the (complex) velocity potential admits a power series representation in a neighborhood of the stagnation point. Other authors have expanded and improved upon this basic idea, including extending it to more general geometries \cite{dagan1972twodimensional,holyer1980gravity,vandenbroeck1994flow,vandenbroeck2010gravitycapillary,chandler2018complex}.  Compared to these works, the geometric method has the advantage of assuming much less regularity a priori. As it does not rely on conformal mappings or other complex-analytic tools, it can also be generalized to treat bores with vorticity.

Despite this wealth of literature on the hypothetical contact angle made by a gravity current, no rigorous constructions of gravity currents have previously been made. Benjamin~\cite[Section 4.3]{benjamin1968gravity} offered an approximate solution that has subsequently been shown to have good agreement with physical and numerical experiments. Exact solutions, however, have been quite difficult to obtain, as the problem is highly nonlinear and there are no explicit solutions nearby from which to perturb.

Lastly, let us make some remarks about the time-dependent system. It is well-known that the gravity internal wave problem is ill-posed in Sobolev spaces, say, due to Kelvin--Helmholtz instability. However, this situation is immediately ameliorated by incorporating small surface tension effects~\cite{shatah2011interface}. In physical experiments \cite{lowe2005non} and numerical simulations \cite{rotunno2011models} of gravity currents in lock-exchange flows, one indeed observes that the interface exhibits rollup phenomena characteristic of Kelvin--Helmholtz instability. Nonetheless, its mean qualitative form is remarkably well-described by the inviscid gravity current model, and it maintains this shape over an extended period of time. This is especially true near the wall, where the $60\degree$ contact angle is observable.

\subsection{Outline of the argument and plan of the paper} 
\label{intro outline section}

The proof of Theorem~\ref{overturning theorem} is an extended argument by contradiction and proceeds roughly as follows. The existence of a global curve of depression bores was proved in~\cite{chen2023global}; see Theorem~\ref{global bore theorem}. Suppose that overturning does not occur. Thus, $\eta(s)$ is uniformly Lipschitz in $s$, and double stagnation \eqref{double stagnation alternative} must occur. Note that by the structure of the dynamic condition~\eqref{eqn:stream:dynamic}, if double stagnation were to occur, it can only be at the unique height $y = F^2/2$. For simplicity, let us reset the coordinates so that this at the origin. Anywhere else along the free boundary, we have by classical results of Caffarelli~\cite{caffarelli1987harnack1} that the Lipschitz continuity of $\eta(s)$ can be upgraded to $C_\loc^{1+\beta}$ regularity for some $\beta \in (0,1)$. Of course, the heart of the matter is to resolve the structure of the interface at the critical height where the dynamic condition~\eqref{eqn:stream:dynamic} is degenerate and standard methods in Bernoulli free boundary regularity theory are inapplicable. 

In a deep paper, Vărvăruca and Weiss~\cite{varvaruca2011geometric} answer the analogous question for the one-fluid case, though they work in arbitrary dimensions and purposely avoid relying on the types of monotonicity properties that we will use extensively.  In Section~\ref{model problem section}, we begin by introducing a model free boundary elliptic problem that shares the same features as~\eqref{eqn:stream}. Note that the free boundary $\fluidS$ is precisely the zero-set of $\psi$, while $\psi$ has (differing) distinguished signs in $\fluidD_1$ and $\fluidD_2$. Following Vărvăruca and Weiss, we therefore study domain variational solutions, which are critical points of the energy with respect to inner variations that maintain the structure of the level sets in a certain sense.  It is important to note that internal waves are typically not extrema of the energy, and this fact precludes the use of some of the stronger results in free boundary regularity. See, however, the very recent paper of Kriventsov and Weiss~\cite{kriventsov2025minmax} where a min-max variational construction  is successfully carried out for a related water wave problem. Exploiting very strongly the qualitative properties of the waves on $\cme$ given by Theorem~\ref{global bore theorem}, we are able to show that if overturning does not occur, then there exists a limiting variational solution $\psi$ to this model elliptic problem. Moreover, $\psi$ enjoys certain monotonicity properties as well as Lipschitz continuity.

One of the key insights of~\cite{varvaruca2011geometric} is that, in the one-phase case, Bernoulli's law justifies the assumption that $|\nabla \psi|^2 = O(y)$; one can see this formally by setting $\rho_2 = 0$ and $\psi_2 = 0$ in \eqref{eqn:stream:dynamic}.  This decay bound is sometimes called a Bernstein-type inequality; it is stated as an assumption throughout~\cite{varvaruca2011geometric}, but it has been verified for weak solutions to the irrotational water wave problem in the periodic and finite-depth cases by Vărvăruca~\cite[Proof of Theorem 3.6]{varvaruca2008bernoulli}, generalizing an idea of Spielvogel~\cite{spielvogel1970variational}.  The decay rate is absolutely essential to the arguments in~\cite{varvaruca2011geometric}. For instance, it is used to justify the subsequential convergence of the blowup limit $\psi( r \placeholder)/r^{3/2}$ as $r \searrow 0$, which is the unique scaling that preserves the dynamic condition. The monotonicity and frequency formulas that constitute the main tools in the analysis are all derived from these types of limits.

However, point-wise decay bounds of this type are not available in the two-fluid problem~\eqref{eqn:stream}. Indeed, arguably the primary distinction between the one- and two-layer cases is that, for the latter, the dynamic condition only controls the decay of the \emph{jump} in the tangential velocity across the boundary, not its magnitude in either phase.  Nonetheless, we are able to prove that, for variational solutions whose level sets are monotone Lipschitz graphs in $x$, an {averaged decay bound} does hold; see Section~\ref{energy bound section} and Theorem~\ref{energy bound theorem}. In particular, we show that the limiting stream function obeys
\begin{equation}
  \label{intro key energy bound}
  \frac{1}{r^2} \int_{B_r} |\nabla \psi|^2 \, dx \, dy = O(r) \qquad \textrm{as } r \searrow 0.
\end{equation}
To our knowledge, this result is new and of some independent interest. Note that an application of the well-known Alt--Caffarelli--Friedman monotonicity formula here would give roughly that
\[
	 \left(\frac{1}{r^2} \int_{B_r \cap \{\psi > 0\}} |\nabla \psi|^2 \, dx \, dy \right) \left(\frac{1}{r^2} \int_{B_r \cap \{\psi < 0\}} |\nabla \psi|^2 \, dx \, dy \right) \searrow 0 \qquad \textrm{as } r \searrow 0.
\]
At a double stagnation point, formally speaking, we expect both of the factors above would limit to $0$ as $r \searrow 0$. The estimate in~\eqref{intro key energy bound} is more useful for our purposes as it gives an explicit upper bound on the decay of the average kinetic energies in both phases, rather than simply saying their product vanishes.

Using \eqref{intro key energy bound}, in Section~\ref{lower bound section} we carry out a blowup argument similar to~\cite{varvaruca2011geometric}. Ultimately, we find that the limiting free boundary is globally $C^1$, and it is asymptotically flat at the origin and the normal vector is purely vertical. It then follows from a comparison principle argument due to Oddson~\cite{oddson1968boundary}, that there is a point-wise \emph{lower bound} on the decay of the gradient: for any $0 < \varepsilon \ll 1$, there exists a constant $C_\varepsilon > 0$ such that
\[
	| \psi_\pm(x,y) | \geq C_\varepsilon  |(x,y)|^{1+\varepsilon} \qquad \textrm{on }\mathcal{D}_\varepsilon^\pm \cap B_R,
\]
where $\mathcal{D}_\varepsilon^\pm$ is a cone oriented along the $\pm y$-axis with aperture angle $\nearrow \pi$ as $\varepsilon \searrow 0$, and $0 < R \ll 1$.  We prove in Section~\ref{proof of the main theorem section} that this lower bound is incompatible with the upper bound in~\eqref{intro key energy bound}. Having arrived at a contradiction, we infer that in fact overturning~\eqref{overturning} must indeed occur.

The proof of Theorems~\ref{non-boussinesq gravity current theorem} and~\ref{boussinesq gravity current theorem} is carried out in Section~\ref{gravity current section}. Our main tools are the same: we introduce a model free boundary problem for~\eqref{gravity current problem}, and an associate notion of domain variational solutions. In Theorem~\ref{limiting gravity current theorem}, we prove that if overturning does not occur along $\cmd$, then one can extract a subsequence converging to a domain variational solution. Moreover, this limiting solution will have as its free boundary the graph of a Lipchitz continuous monotonically decreasing function, and the energy decays like~\eqref{intro key energy bound}. The main new idea here is looking at an auxiliary problem where the upper and lower walls are removed and the fluid is extended as uniform laminar flow, with the velocity matching the upstream limit. Doing so, allows us to use elliptic theory to control the Lipschitz norm of $\psi$ even as the interface draws arbitrarily close to the upper wall. 

We then introduce a Vărvăruca--Weiss-style monotonicity formula tailored to the model gravity current problem, which enables us to carry out a blowup analysis, characterizing the structure of the flow in a neighborhood of the contact point. This is done for both the non-Boussinesq and Boussinesq cases. Based on that local information, we give in Theorem~\ref{contact angle theorem} a partial proof of von Karman's conjecture; the statements in Theorems~\ref{non-boussinesq gravity current theorem} and~\ref{boussinesq gravity current theorem} then come as a simple corollary.

\section{Preliminaries}

\subsection{Notation}

Throughout the paper, we denote one-dimensional Hausdorff measure by $\mathcal{H}^1$ . The ball of radius $r$ centered at $z \in \mathbb{R}^2$ is written $B_r(z)$; when $z=0$, we abbreviate this to $B_r$. Likewise, 
\[
	B_r^\pm \colonequals \left\{ (x,y) \in B_r : \pm y > 0 \right\}
\]
are the upper and lower half balls of radius $r$ centered at the origin. We denote by $\nu$ the outer unit normal to a set, either in the classical sense if the boundary is $C^1$, or the measure theoretic sense if the set is of finite perimeter. 

For $D \subset \mathbb{R}$ or $\mathbb{R}^2$, $k \geq 0$, and $\alpha \in (0,1)$, the space $C^{k+\alpha}(D)$ consists of $k$-times Hölder continuous functions of exponent $\alpha$ on $D$. We use $\Lip(D)$ as notation for the set of Lipschitz continuous function on $D$, which as usual means the set of continuous functions on $D$ for which the Lipschitz seminorm is finite:
\[
	[u]_{\Lip; D} \colonequals \sup_{\substack{z_1, z_2 \in D \\ z_1 \neq z_2}} \frac{|u(z_1) - u(z_2)|}{|z_1 - z_2|} < \infty.
\]
Finally, we write $H^k(D)$ for the standard $k$-th order $L^2$-based Sobolev space on $D$. For any of these spaces, a subscript of ``c'' indicates the subset of functions whose support is compactly contained inside $D$. When $D$ is unbounded, for clarity we use a subscript of ``$\bdd$'' to indicate functions that are uniformly bounded in that their norm on $D$ is finite. Conversely, a subscript of ``$\loc$'' denotes the set of functions whose norm is bounded on each compact subset of $D$. Thus, for example, 
\[
	C_\bdd^{k+\alpha}(D) \colonequals \left\{ u \in C_\loc^{k+\alpha}(D) : \| u \|_{C^{k+\alpha}(D)} < \infty\right\},
\]
where $\| \placeholder \|_{C^{k+\alpha}(D)}$ is the standard Hölder norm on $D$.

\subsection{Some tools from elliptic theory}

For later reference, we record here two important results from elliptic theory that will be used at multiple points in the analysis. First, we recall the classical monotonicity formula of Alt, Caffarelli, and Friedman~\cite[Lemma 5.1]{alt1984variational}. Here, we give only the two-dimensional case, as that is all we need.

\begin{theorem}[ACF monotonicity formula]
\label{acf theorem}
Suppose  that $u_1, u_2 \in C^0(B_1)$ are nonnegative subharmonic function such that $u_1 u_2 = 0$ on $B_1$. For $0 < r < 1$, define the function
\[
	\Phi(u_1, u_2; r) \colonequals \left(\frac{1}{r^2} \int_{B_r} |\nabla u_1|^2 \, dx \, dy \right) \left(\frac{1}{r^2} \int_{B_r} |\nabla u_2|^2 \, dx \, dy \right).
\]
Then $\Phi(u_1, u_2; \placeholder)$ is monotonically decreasing and the right-hand limit $\Phi(u_1, u_2; 0+)$ exists and is finite.
\end{theorem}

We mention that the original statement in~\cite{alt1984variational} assumes both $u_1$ and $u_2$ are harmonic, but this can easily be generalized to subharmonic functions as noted, for example, in~\cite{caffarelli2002monotonicity}. The utility of Theorem~\ref{acf theorem} has already been discussed in Section~\ref{intro outline section}. We will typically apply it with $u_1$ and $u_2$ being $|\psi_1|$ and $|\psi_2|$, each extended by zero so they are defined on $\overline{\fluidD}$.

Next, we give a slight restatement of a classical work of Oddson~\cite{oddson1968boundary}. This phrasing can be found, for example, in~\cite[Proposition 5.12]{varvaruca2009extreme}. One can understand this result as a quantified Hopf edge-point lemma specialized to conical domains. 

\begin{theorem}[Oddson]
  \label{oddson theorem}
  Consider the truncated conical region 
  \[
    \mathcal{D}_{\mu,R} \colonequals \left\{ (x, y) : 0 < x^2+y^2 < R^2,~ |\theta-\tfrac{\pi}{2}| < \tfrac{\pi}{2\mu} \right\} \subset \mathbb{R}^2,
  \]
  where $\theta = \arg{(x,y)}$, and $\mu > 1$ and $R > 0$ are fixed. Suppose that $u \in C^2(\mathcal{D}_{\mu,R}) \cap C^0(\overline{\mathcal{D}_{\mu,R}})$ satisfies
  \[
    \left\{ \begin{aligned}
      \Delta u & \leq 0 & \qquad&  \textrm{on } \mathcal{D}_{\mu,R} \\
      u & > 0 & \qquad & \textrm{on } \mathcal{D}_{\mu,R} \setminus \{ 0\} \\
      u & = 0 & \qquad & \textrm{at } (0,0). 
    \end{aligned} \right.
  \]
  Then,
  \begin{equation}
    \label{oddson bound}
    u(x,y) \geq C |(x,y)|^{\mu} \cos{\left(\mu (\theta - \pi/2)\right)} \qquad \textrm{in } \mathcal{D}_{\mu,R}
  \end{equation}
  for a constant $C > 0$ depending only on $\mu$.
\end{theorem}

%As $\mu \searrow 1$, then conical region $\mathcal{D}_{\mu,R}$ tends to the half ball $B_R^+$, in which case the Hopf lemma would imply that $u$ is uniformly strictly decreasing along any non-tangential limit. In that sense, one can view~\eqref{oddson bound} as a quantitative (partial) converse to the Hopf lemma, in that it gives a lower bound on the decay of $u$ as we approach $0$ in the cone. 

\subsection{The curves \texorpdfstring{$\cme$}{C\_elev} and \texorpdfstring{$\cmd$}{C\_depr}}
\label{preliminaries existence section}

For the convenience of the reader, we record here the results from~\cite{chen2023global} establishing the existence and qualitative features of the two families of large-amplitude bores that form the starting point for the analysis in the present paper. In keeping with the notation in~\cite{chen2023global}, we write $\cm^+$ for $\cmd$ and $\cm^-$ for $\cme$; the sign indicates whether the upstream depth of the lower layer is larger ($+$) or smaller ($-$) than the trivial solution. 

\begin{theorem}[Large-amplitude bores] 
\label{global bore theorem} 
Fix an $\alpha \in (0,1)$ and densities $0 < \rho_2 < \rho_1$.  There exist global $C^0$ curves
\[ \cm^\pm = \left\{ (\psi(s), \eta(s), h_1(s)) : \pm s \in (0,\infty) \right\}\]
 of solutions to the internal wave problem \eqref{eqn:stream} with $h_1 = \lambda(s)$, $h_2 = 1-h_1$, and $F$ given by \eqref{definition front Froude number}.   They enjoy the  H\"older regularity  
 \[ \psi(s) \in C_\bdd^{2+\alpha}(\overline{\fluidD_1(s)}) \cap C_\bdd^{2+\alpha}(\overline{\fluidD_2(s)}) \cap C_\bdd^0(\overline{\fluidD(s)}), \qquad \eta(s) \in C_\bdd^{2+\alpha}(\mathbb{R}), \]
 where $\fluidD(s)$ denotes the fluid domain corresponding to $\eta(s)$ and $\lambda(s)$.  
\begin{enumerate}[label=\rm(\alph*)]
\item \label{global bore monotone part} \textup{(Strict monotonicity)} Each solution on $\cm^\pm$ is a strictly monotone bore:
\begin{equation}
  \begin{aligned}
    \pm \partial_x \eta(s) &< 0 & \qquad &  \textrm{on } \mathbb{R}, \\
    \pm \partial_x \psi_i(s) & < 0 & \qquad & \textrm{in } \fluidD_i(s) \cup \fluidS(s), \\
    \partial_y \psi_i(s) & < 0 & \qquad & \textrm{in } \overline{\fluidD_i(s)},
  \end{aligned} \qquad \textrm{for } i = 1,2. \label{monotonicity Euler variables}
\end{equation}
\item \textup{(Elevation limit)} \label{elevation part} In the limit along $\cm^-$, either \emph{overturning} occurs in that 
  \begin{equation}
    \limsup_{s \to -\infty} \| \partial_x \eta(s) \|_{L^\infty(\mathbb{R})} = \infty, \label{overturning} 
  \end{equation}
or else a double stagnation point develops: 
\begin{equation}
\label{double stagnation alternative}
  \begin{gathered}
    \eta(s) \longrightarrow \eta^* \in \Lip(\mathbb{R}) \quad \textrm{in $C_\loc^\varepsilon$ for all } \varepsilon \in (0,1) \\
    \textrm{and} \\
    \inf_{\fluidS(s)} \left( |\nabla \psi_1(s)| + |\nabla \psi_2(s)|\right) \longrightarrow 0. 
  \end{gathered}
\end{equation}
	\item \textup{(Depression limit)} \label{depression part} In the limit along $\cm^+$, either \emph{overturning} occurs or the interface comes arbitrarily close to the upper wall: 
  \begin{equation}
    \limsup_{s \to \infty} \| \partial_x \eta(s) \|_{L^\infty(\mathbb{R})} = \infty \quad \textrm{or} \quad \lim_{s \to \infty} h_1(s) = 1. \label{depression limit alternatives} 
  \end{equation}
  In either case, 
  \[
  	\lim_{s \to \infty} \sup_{\fluidS(s)} \partial_y \psi_2(s) = 0.
  \]
\end{enumerate}
\end{theorem}

\begin{theorem}
  Theorem~\ref{global bore theorem} continues to hold in the Boussinesq limit where $\rho_2 = \rho_1$ and \eqref{eqn:stream:dynamic} is replaced by \eqref{eqn:stream:dynamic-boussinesq}, except that the double stagnation alternative \eqref{double stagnation alternative} for elevation bores is replaced by 
    $\lim_{s \to \infty} h_1(s)=0$,
  i.e.~that the interface comes arbitrarily close to the lower wall.
  \begin{proof}
    % Still very rough
    In the context of \eqref{eqn:stream}, the Boussinesq limit is not particularly singular, and most of the arguments in \cite{chen2023global} go through with essentially no modifications. Some care needs to be taken, though, when working with variational or weak formulations of the problem. These involve terms which are formally motivated by the expression
    \begin{align}
      \label{gravitational potential energy}
      V \colonequals \int_{\fluidD_1} \frac{\rho_1}{F^2} y\, dx\, dy
      + 
      \int_{\fluidD_2} \frac{\rho_2}{F^2} y\, dx\, dy,
    \end{align}
    for the total gravitational potential energy, where for the moment we ignore convergence issues related to the unboundedness of the fluid layers $\fluidD_1,\fluidD_2$. Substituting \eqref{definition front Froude number} and sending $\rho_2 \to \rho_1$, the integrands in \eqref{gravitational potential energy} become infinite. To avoid this, one splits the right hand side into a term coming from the average density $(\rho_1+\rho_2)/2$ and a remainder,
    \begin{align*}
      V 
      = 
      \int_{\fluidD_1 \cup \fluidD_2} \frac{\rho_1+\rho_2}{2F^2} y\, dx\, dy
      + \int_{\fluidD_1} \frac{\rho_1-\rho_2}{2F^2} y\, dx\, dy
      + \int_{\fluidD_2} \frac{\rho_2-\rho_1}{2F^2} y\, dx\, dy.
    \end{align*}
    The first integral is now independent of the position of the interface, and so from a variational point of view can be discarded. Substituting \eqref{definition front Froude number} into the second two terms and sending $\rho_2 \to \rho_1$ we obtain a more sensible limit
    \begin{align}
      \label{boussinesq potential energy}
      V' \colonequals
      +\int_{\fluidD_1} 2\rho_1 y\, dx\, dy
      - \int_{\fluidD_2} 2\rho_1 y\, dx\, dy.
    \end{align}

    Based on this physical reasoning, one guesses that in the Boussinesq limit the Lagrangian density $\mathscr L$ in section~5.2 of \cite{chen2023global} should be replaced with
    \begin{align*}
      % Miles: does this \mp business match sufficiently well with the convention later in the paper?
      \mathscr L(p,z,\xi,\lambda) = \rho_1 \left[ h^2 \frac{1+\xi_1^2}{2(h+\xi_2)^2} + \frac 12 \mp 2(hp+z) \right] (h+\xi_2)
    \end{align*}
    with the $\mp$ resolving to $-$ in the lower layer and $+$ in the upper layer, and that a similar modification should made to the second weak formulation used in section~5.5. It is then relatively straightforward to verify that these are indeed valid formulations of the Boussinesq system \eqref{eqn:psi harmonic}--\eqref{eqn:stream:dynamic-boussinesq}, and that the conjugate flow arguments and velocity estimates in \cite[sections~5.3 and 5.5]{chen2023global} are unaffected. 
    
    Note that in the non-Boussinesq case along $\cme$, the situation where $h_1 \to 0$ was ruled out in~\cite[Lemma 5.22]{chen2023global}. That argument, however, relied on the Bernoulli constant being nonzero, and hence it does not apply in the Boussinesq setting.
  \end{proof}
\end{theorem}

\section{Overturning bores}
\label{overturning section}

\subsection{Model free boundary problem for internal waves}
\label{model problem section}

We study the following two-phase free boundary elliptic PDE, which serves as a prototype for the limiting bore problem along $\cmd$:
\begin{equation}
\label{elliptic PDE}
  \left\{ \begin{aligned}
    \Delta u & = 0 & \qquad & \textrm{in } \Omega^+(u) \cap B_1  \\
    \Delta u & = 0 & \qquad & \textrm{in } \Omega^-(u)  \cap B_1 \\
    | \nabla u_+ |^2 - | \nabla u_-|^2 & = -\jump{\rho} y & \qquad & \textrm{on } \partial \Omega^+(u) \cap B_1,
  \end{aligned} \right.
\end{equation}
where $\rho_+ \neq \rho_-$ are distinct positive constants, $\jump{\rho} \colonequals \rho_+-\rho_-$, and
\[
	\Omega^+(u) \colonequals \inter{\{ u \geq 0 \}}, \quad 
	\Omega^-(u) = \{ u < 0 \}, \quad 
	u_\pm \colonequals 
		\left\{ \begin{aligned}
			u & \qquad \textrm{on } \Omega^\pm(u) \\
			0 & \qquad \textrm{on } \Omega^\mp(u).
		\end{aligned} \right.
\]
% Miles: For my own sanity, would be great to have the explicit recipe here. Is it really just changing a sign, or do we also need scaling somewhere?
The above problem is obtained from the (pseudo) stream function formulation of the Euler equations~\eqref{eqn:stream} by normalizing several constants to unity and shifting the axes. We have taken as the unknown $u = -\psi$, since $\mp \psi_\pm > 0$ on $\fluidD_\pm$.  While~\eqref{elliptic PDE} may arise in other contexts, we will continue to refer to the sets $\Omega_+(u)$ and $\Omega_-(u)$ as the ``phases'' or ``fluids.'' It is interesting to note that all of the results in Section~\ref{model problem section}--\ref{lower bound section} hold whenever $\rho_+ \neq \rho_-$; in particular, they are not restricted to stably stratified internal waves. The choice to let $\Omega^+(u)$ be the interior of the set where $u$ is nonnegative is motivated by the situation along $\cmd$, where the upper fluid tends to stagnation. 

The \emph{free boundary} of $u$ is defined to be the set 
\[
	\Gamma(u) \colonequals \partial \Omega^+(u) \setminus \partial B_1.
\]
We are interested in determining the regularity of $\Gamma(u)$ near $y = 0$, where it is possible a priori that there may be a double stagnation point on the boundary.  In \cite{chen2023global}, we arrive at a problem similar to \eqref{elliptic PDE}, but in a very specific way that justifies making additional assumptions.  In particular, we will always suppose that the zero-set of $u$ is the graph of a Lipschitz continuous function $\eta$.  Without loss of generality, pick axes so that $\eta(0) = 0$, meaning the origin is on the boundary.   

Following the general idea of \cite{varvaruca2011geometric}, we introduce the following notion of a solution to~\eqref{elliptic PDE} that is well-adapted to analysis via monotonicity functions.

\begin{definition}
\label{variational solution definition}
A function $u \in H^1(B_1) \cap C^0(B_1)$ is said to be a \emph{(domain) variational solution} to the free boundary elliptic PDE~\eqref{elliptic PDE} provided $u_\pm \in C^2(\Omega^\pm(u))$ and $u$ is critical point with respect to domain variations of the functional
\[
	E(v) \colonequals \int_{B_1} \left( | \nabla v|^2 -  y \left( \rho_+ \chi^+(v) + \rho_-  \chi^-(v) \right) \right) \, dx \, dy,
\]
where $\chi^\pm(v) \colonequals \chi_{\Omega^\pm(v)}$. More precisely, we require that
\begin{equation}
\label{definition critical point}
  \begin{aligned}
    0  & =  \int_{B_1} \left( |\nabla u|^2 (\nabla \cdot \phi) - 2 D \phi[\nabla u,\nabla u] - ( \rho_+ \chi^+(u) + \rho_- \chi^-(u) ) \nabla \cdot (y \phi) \right) \, dx \, dy,
  \end{aligned}
\end{equation}
for all $\phi \in C_0^1(B_1; \mathbb{R}^2)$.  
\end{definition}

The right-hand side of \eqref{definition critical point} formally results from setting
\[
	0 = - \partial_\varepsilon E\left(u \circ ( \id + \varepsilon \phi)\right)\big|_{\varepsilon= 0}.
\]
Note that by testing against vector fields supported in $\Omega^\pm(u)$, \eqref{definition critical point} implies that $u_\pm$ is harmonic in $\Omega^\pm(u)$.  Under the assumption that $\partial \Omega^+(u) \setminus \partial B_1 = \partial \Omega^-(u) \setminus \partial B_1$ is a smooth curve, integrating by parts in \eqref{definition critical point} gives
\[
	0 = \int_{B_1 \cap \partial \Omega^+(u)} \left( |\nabla u_+|^2 - |\nabla u_-|^2 + \jump{\rho} y \right) \phi \cdot \nu \, d \mathcal{H}^1, 
\] 
with $\nu$ being the outward unit normal to $\Omega^+(u)$ along $\partial \Omega^+(u)$.  Thus, the jump condition in \eqref{elliptic PDE} is recovered for suitably well-behaved variational solutions.

\begin{definition}[Monotonicity]
\label{definition monotone}
We call a variational solution $u$ \emph{monotone} provided that 
\begin{subequations} \label{monotonicity assumptions} 
  \begin{equation}
    \label{sign assumption} \textrm{$u_x$ and $u_y$ have non-strict signs in $\Omega^+(u) \cup \Omega^-(u)$},
  \end{equation}
 and
  \begin{equation}
    \label{Lipschitz assumption} |u_x| \leq M |u_y| \qquad \textrm{in } \Omega^+(u) \cup \Omega^-(u).
  \end{equation}
\end{subequations}
%In particular, $\Gamma(u) = \partial \Omega_+(u) \setminus \partial B_1 =  \partial \Omega_-(u) \setminus \partial B_1$ and all the level sets of $u$ are monotone Lipschitz graphs.
\end{definition}
\begin{remark}
We caution that the term ``monotonicity'' is used in this paper in several different ways. For the ACF monotonicity formula Theorem~\ref{acf theorem}, monotonicity is of the product of the mean Dirichlet energy functionals with respect to the radius of the ball. Relatedly, we establish the monotonicity of the so-called boundary adjusted energy functions $\mathcal{M}$ and $\mathcal{M}_\gc$ with respect to the radius in Theorems~\ref{monotonicity theorem} and \ref{gc monotonicity formula theorem}. On the other hand, in Definition~\ref{definition monotone}, monotonicity refers to sign conditions on the derivatives of the solution $u$ to the free boundary problem. The first two uses of monotonicity are common in studies of the Bernoulli free boundary problems, while the latter is widespread in the water waves literature.
\end{remark}

While the monotonicity asked for in Definition~\ref{definition monotone} is a very strong assumption, it is precisely what we expect to obtain in the hypothetical limit along $\cme$ (suitably translated and scaled) of a non-overturning sequence of classical solutions thanks to the qualitative properties~\eqref{monotonicity Euler variables}.  That is the content of the next lemma, whose proof is deferred to Section~\ref{proof of the main theorem section} as it requires several of the tools to be developed in the next section. 

\begin{lemma}[Limiting variational solution]
  \label{limiting solution lemma}
  Suppose that in the limit along $\cme$ overturning does not occur in that
  \begin{equation}
    % TODO : multiply defined?
     \label{elevation no overturning}
    \limsup_{s \to -\infty} \| \partial_x \eta(s) \|_{L^\infty} < \infty.
  \end{equation}
  Then there exists a nontrivial variational solution $u$ to the elliptic PDE~\eqref{elliptic PDE} with 
  \[
    \rho_+ = \frac{2}{F^2} \rho_2, \qquad \rho_- = \frac{2}{F^2} \rho_1.
  \]
  Moreover, $u$ is monotone in the sense of Definition~\ref{definition monotone}, Lipschitz continuous on $B_1$, and for any sequence $r_m \searrow 0$, the corresponding blowup sequence
  \begin{equation}
    \label{blowup sequence}
    u_m \colonequals \frac{u(r_m \placeholder)}{r_m^{3/2}} \qquad \textrm{for all } m \geq 1,
  \end{equation}
  is uniformly locally Lipschitz: for any $K \subset\subset \mathbb{R}^2$,
  \begin{equation}
    \label{blowup sequence Lipschitz bound}
    \sup_{m \gg 1} \| u_m \|_{\Lip(K)} \lesssim 1.
  \end{equation}
\end{lemma}

\begin{remark}
Notice that as a consequence of the Lipschitz bound~\eqref{blowup sequence Lipschitz bound}, the gradient exhibits the decay
\begin{equation}
\label{double stagnation decay}
  \begin{aligned}
    \frac{1}{r_m^2} \int_{B_{r_m}} |\nabla u|^2 \, dx \,dy & = r_m \int_{B_1} |\nabla u_m|^2 \,dx \, dy =  O(r_m).
  \end{aligned}
\end{equation}

\end{remark}

For monotone solutions, both the free boundary and nonzero level sets of a monotone variational solution have graph geometry. In the physical context of the internal wave or gravity current problem, these corresponds to the (relative) streamlines, that is, the integral curves of the relative velocity field. 

\begin{lemma}[Streamlines]
\label{streamlines lemma}
If $u \in C^0(B_1) \cap C^2(\Omega^+(u) \cup \Omega^-(u))$ is harmonic on $\Omega^+(u) \cup \Omega^-(u)$, $\Gamma(u) \neq \emptyset$, and $u$ satisfies \eqref{monotonicity assumptions}, then the following statements hold.
\begin{enumerate}[label=\rm(\alph*)]
	\item \label{nonzero level sets part} Each connected component of a nonzero level set of $u$ is the graph of a monotone real-analytic function of $x$ with Lipschitz constant not exceeding $M$. 
	\item \label{free boundary part} Each connected component of $\Gamma(u)$ is the graph of a monotone Lipschitz continuous function of $x$ with Lipschitz constant not exceeding $M$.
\end{enumerate}
\end{lemma}
\begin{proof}
First, observe that at any $z \in \Omega^-(u)$, we must have that $u_y(z) \neq 0$. Were this not the case, then by \eqref{Lipschitz assumption}, we would have that $|\nabla u(z)| = 0$. The strong maximum principle and \eqref{sign assumption} then imply that $u_x$ and $u_y$ vanish identically on the connected component of $\Omega^-(u)$ containing $z$. Since $\Gamma(u) \neq \emptyset$, this forces $u$ to vanish identically on that connected component, which contradicts the definition of $\Omega^-(u)$. Thus, the assumption that $u_y$ has a non-strict sign \eqref{sign assumption} can be strengthened  to the statement that $u_y$ has a strict sign on $\Omega^-(u)$. The (real-analytic) implicit function theorem and a continuity argument guarantee that the connected component of the level set $\{ u = u(z) \}$ containing $z$ extends to $\partial B_1$ and is the graph of a real-analytic function $\eta^z = \eta^z(x)$. Note, moreover, that $\eta^z$ is nondecreasing if $u_x u_y \leq 0$ on $\Omega^-(u) \cup \Omega^+(u)$ and nonincreasing if $u_x u_y \geq 0$ on $\Omega^-(u) \cup \Omega^+(u)$. These cases are exhaustive by assumption 
\eqref{sign assumption}. Thanks to the bound in \eqref{Lipschitz assumption}, the Lipschitz constant of $\eta^z$ cannot exceed $M$.

The same argument applies to any point $z \in \Omega^+(u)$ with $u(z) \neq 0$. We conclude that $u_y$ has a strict sign on $\{ u = 0 \}^c$, and the connected components of the positive level sets of $u$ are likewise monotone real-analytic graphs with Lipschitz constant not exceeding $M$. This proves the statement in part~\ref{nonzero level sets part}.

Next, let $z_0 = (x_0, y_0) \in B_1$ be given with $u(z_0) = 0$. As one useful consequence of the above paragraphs, we have that  
\begin{equation}
\label{must be above or below Gamma}
  \Omega^\pm(u) \cap \{ x = x_0 \} \subset \left\{ \begin{aligned}
    \{ (x_0, y) : \pm (y- y_0) < 0 \} \cap B_1 & \qquad \textrm{if } u_y < 0 \textrm{ on } \{ u = 0 \}^c \\
    \{ (x_0, y) : \pm (y- y_0) > 0 \} \cap B_1  & \qquad \textrm{if } u_y >  0  \textrm{ on } \{ u = 0 \}^c.
  \end{aligned} \right.
\end{equation}
If, in addition, $z_0 \in \Gamma(u)$, then there exists a sequence $\{ z_n \} \subset \Omega^-(u)$ with $z_n \to z_0$ and $u(z_n) \nearrow 0$. Let $\{ \eta_n \}$ be the corresponding sequence of real-analytic and uniformly Lipschitz continuous functions such that the connected component of $\{ u = u(z_n) \}$ containing $z_n$ is given by the graph of $\eta_n$. Note that there exist an open interval $I \ni x_0$, whose length is uniformly bounded from below in terms of $M$, and such that $I$ is in the domain of $\eta_n$ for all $n \gg 1$. The strict sign of $u_y$ on $\Omega^-(u)$ therefore implies that 
\[
	\eta_n < \eta_{n+1} \quad \textrm{on } I \textrm{ if } u_y > 0 \textrm{ on } \Omega^-(u) \qquad \textrm{and} \qquad \eta_{n+1} < \eta_{n} \quad \textrm{on } I \textrm{ if } u_y < 0 \textrm{ on } \Omega^-(u).
\]
In either case, the region lying between the graphs of $\eta_n|_I$ and $\eta_{n+1}|_I$ must be a subset of $\Omega^-(u)$. If not, it would contain a point $z = (x,y) \in \Gamma(u)$ with $y$ strictly between $\eta_n(x)$ and $\eta_{n+1}(x)$, which contradicts~\eqref{must be above or below Gamma}. 

For the remainder of the proof, let us assume that $u_y > 0$ on $\{ u = 0 \}^c$; the other case follows by a straightforward modification of the argument. Because the sequence $\{ \eta_n \}$ is uniformly Lipschitz on $I$, it has a subsequence converging uniformly to a Lipschitz continuous function $\eta_0$ on $I$. Since each $\eta_n$ is nondecreasing or nonincreasing, depending on the non-strict sign of $u_x u_y$, the same is true of $\eta_0$. By construction, we see that $\eta_0(x_0) = y_0$ and $u$ vanishes on the graph of $\eta_0$. Moreover, from the previous paragraph, there is an open neighborhood $\mathcal{O}$ of $z_0$ such that the subset of $\mathcal{O}$ lying below the graph of $\eta_0$ is in $\Omega^-(u)$. Applying~\eqref{must be above or below Gamma} at each point on the graph of $\eta_0$, we infer further that the subset of $\mathcal{O}$ lying above the graph of $\eta_0$ must be in $\Omega^-(u)^c$. This confirms that $\Gamma(u)$ and the graph of $\eta_0$ coincide on $\mathcal{O}$. The statement in part~\ref{free boundary part} now follows easily by a continuity argument and the compactness of $\Gamma(u)$.
\end{proof}

%
%
%
%We mention in particular that monotonicity implies that 
%\begin{equation}
%  \label{monotonicity assumption}
%  u_x u_y \leq 0 \quad \textrm{on } B_1 \setminus \Gamma(u) \qquad \textrm{or} \qquad u_x u_y \geq 0 \quad \textrm{on } B_1 \setminus \Gamma(u).
%\end{equation}
%We will frequently assume that, in addition to being a monotone variational solution, $u \in \Lip(B_1)$. Again, this rather strong hypothesis is satisfied by the hypothetical limiting bore solution as a consequence of the a priori estimates in~\cite{chen2023global}. The key tool is the celebrated Alt--Caffarelli--Friedman monotonicity formula; see Theorem~\ref{acf theorem}. 
We conclude the section with an easy but useful lemma establishing a ``localized'' version of~\eqref{definition critical point}. In particular, this will allow us to evaluate it with explicit test vector fields without having to introduce cut-off functions.

\begin{lemma} \label{variational solution identity}
Suppose that $u$ is a variational solution to~\eqref{elliptic PDE}.  Then, for all $\phi \in C^1(B_1; \mathbb{R}^2)$,
\begin{align*}
	& \int_{B_r} \left( |\nabla u|^2 \nabla\cdot \phi - 2 D\phi [\nabla u, \nabla u] - \left( \rho_+ \chi^+(u)+\rho_- \chi^-(u) \right) \nabla \cdot (y \phi) \right) \,dx \, dy \\
	& \qquad \qquad = \int_{\partial B_r} \left( |\nabla u|^2 \phi - 2 (\phi \cdot \nabla u) \nabla u -  \left( \rho_+ \chi^+(u)+\rho_- \chi^-(u) \right)  y \phi \right) \cdot \nu \,d\mathcal{H}^1,
\end{align*}
for almost every $0 < r < 1$.
\end{lemma}
\begin{proof}
For $0 < \delta \ll r$, let $\zeta_\delta : \mathbb{R}_+ \to \mathbb{R}$ be a smooth function with 
\[
	\zeta_\delta(t) = \left\{ \begin{aligned}
		1 & \qquad \textrm{for } t  \in (0,r-\delta) \\
		0 & \qquad \textrm{for } t \in (r,\infty)
	\end{aligned} \right.,
		\qquad
	0 \leq \zeta_\delta \leq 1,
		\qquad
	(t-r) \zeta_\delta^\prime(t) = O(1) \quad \textrm{on } (r-\delta, r).
\]
Then $\zeta_\delta(|\placeholder|) \phi \in C_0^1(B_1;\mathbb{R}^2)$, and so by the definition of a variational solution~\eqref{definition critical point}, it follows that $u$ satisfies 
\begin{align*}
	0 &= \int_{B_1}  |\nabla u|^2  \left( \zeta_\delta \nabla \cdot \phi   +  \zeta_\delta^\prime \frac{(x,y)}{|(x,y)|} \cdot \phi  \right) \, dx \, dy \\
	& \qquad - 2 \int_{B_1} \left( \zeta_\delta D\phi [ \nabla u, \nabla u]   + \zeta_\delta^\prime \phi \cdot \nabla u  \frac{(x,y)}{|(x,y)|}   \cdot \nabla u   \right) \, dx \, dy \\
	& \qquad-  \int_{B_1} \zeta_\delta \left( \rho_+ \chi^+ + \rho_- \chi^- \right)  \nabla \cdot (y\phi) \, dx \, dy  
		 -\int_{B_1} \zeta_\delta^\prime \left( \rho_+ \chi^+ + \rho_- \chi^- \right) \frac{(x,y)}{|(x,y)|} \cdot (y\phi) \, dx \, dy,
\end{align*}
where we are using the shorthand $\zeta_\delta = \zeta_\delta(|z-z^0|)$ and $\chi^\pm = \chi^\pm(u)$.  Taking the limit $\delta \searrow 0$ gives the stated identity for almost every $r \in (0,1)$.
\end{proof}

\subsection{Energy density bound}
\label{energy bound section}

In this section, we prove the crucial estimate on the average Dirichlet energy for monotone solutions to~\eqref{elliptic PDE}.  When this is applied to a variational solution $u$ that is obtained in the double stagnation limit, it will give the decay~\eqref{double stagnation decay}.

With that in mind, for a fixed variational solution $u$ and $0 < r < 1$, define 
\begin{equation}
\label{definition A and B functionals}
  \begin{aligned}
    \mathcal{A}(u;r)   \colonequals \frac{1}{r^2} \int_{B_r} u_x u_y \, dx \, dy, \qquad
    \mathcal{B}(u;r)     \colonequals \frac{1}{r^2} \int_{B_r} \left( u_y^2 - u_x^2 \right) \, dx \, dy. 
  \end{aligned}
\end{equation}
The main theorem of this section is then as follows.

\begin{theorem}[Energy bound]
\label{energy bound theorem}
Let $u$ be a variational solution to the free boundary elliptic equation~\eqref{elliptic PDE}.
\begin{enumerate}[label=\rm(\alph*\rm)]
	\item \label{A B Lipschitz part}
	$\mathcal{A}(u;\placeholder)$ and $\mathcal{B}(u;\placeholder)$ admit Lipschitz continuous extension to $[0,1)$ with Lipschitz constants depending only on $\rho_+$ and $\rho_-$.
	
	\item \label{energy bound part}
	If $u$ is monotone in the sense of Definition~\ref{definition monotone}, then
  \begin{equation*}
    	\frac{1}{r^2} \int_{B_r} |\nabla u|^2 \, dx \, dy \leq  \max{\{ 2M+1,2\}} |\mathcal{A}(u; \, r)| + \mathcal{B}(u; \, r)
  \end{equation*}
	for all $0 < r \ll 1$. In particular, there is a constant $C = C(\rho_+, \rho_-) > 0$ such that
	\begin{equation}  
	    \label{energy bound}
    	\frac{1}{r^2} \int_{B_r} |\nabla u|^2 \, dx \, dy \leq Cr + \max{\{ 2M+1,2\}} |\mathcal{A}(u; \, 0+)| + \mathcal{B}(u; \, 0+).
  	\end{equation}
\end{enumerate}
\end{theorem}
\begin{remark}
As we will see in the proof, it is sufficient to have that $u_x u_y$ has a non-strict sign in $\Omega^-(u) \cup \Omega^+(u)$, which is slightly weaker than \eqref{sign assumption}.
\end{remark}
\begin{proof}[Proof of Theorem~\ref{energy bound theorem}]
Let $u$ be given; throughout the proof we abbreviate $\mathcal{A}(u; \placeholder)$ by $\mathcal{A}$ and likewise for $\mathcal{B}$. Derivatives with respect to $r$ are denoted with a prime.  First, observe that $\mathcal{A}$ and $\mathcal{B}$ are continuous for $r \in (0,1)$ as $u \in H^1(B_1)$. To prove that they are (globally) Lipschitz on $[0,1)$, we will obtain uniform estimates for their weak derivatives. These will follow from applying Lemma~\ref{variational solution identity} with two specific test fields.  

Starting with the divergence-free vector field $\phi = (y,x)$ we see that
\begin{equation}\label{eq ibp1}
4 \int_{B_r} u_x u_y \,dx\, dy + \int_{B_r} \rho x \,dx\,dy = 2r \int_{\partial B_r} u_x u_y \,d\mathcal{H}^1 + \frac{2}{r} \int_{\partial B_r} \rho x y^2 \, d\mathcal{H}^1.
\end{equation}
Now, for almost every $r \in (0,1)$, 
\[
	\mathcal{A}^\prime(r) = -\frac{2}{r^3} \int_{B_r} u_x u_y \,dx\, dy + \frac{1}{r^2} \int_{\partial B_r} u_x u_y \, d\mathcal{H}^1,
\]
and hence it follows from \eqref{eq ibp1} that
\[
	\mathcal{A}^\prime(r) = \frac{1}{2r^3} \int_{B_r} \rho x \,dx \, dy - \frac{1}{r^4} \int_{\partial B_r} \rho x y^2 \, d\mathcal{H}^1.
\]
The right-hand side is uniformly bounded in $r$ with an upper bound depending only on $\rho_+$ and $\rho_-$, and hence $\mathcal{A}$ is Lipschitz. 
%Integrating, we find that for almost every $r, r_0 > 0$, 
%\begin{align*}
%\mathcal{A}(r) & = \mathcal{A}(r_0) + \frac{1}{4r_0^2} \int_{B_{r_0}} \rho x \,dx\,dy - \frac{1}{4r^2} \int_{B_{r}} \rho x \,dx\,dy - \frac14 \int_r^{r_0}  \frac{1}{t^2}\int_{\partial B_t} \rho x \,d\mathcal{H}^1dt \\
%	& \qquad + \int_r^{r_0}  \frac{1}{t^4}\int_{\partial B_t} \rho x y^2 \,d\mathcal{H}^1dt.
%\end{align*}
%Note that 
%\begin{equation*}
%-\frac14 \int_r^{r_0}  \frac{1}{t^2}\int_{\partial B_t} \rho x \,d\mathcal{H}^1dt + \int_r^{r_0}  \frac{1}{t^4}\int_{\partial B_t} \rho x y^2 \,d\mathcal{H}^1dt = \int^{r_0}_r \int^{2\pi}_0 \rho \cos\theta \left( -\frac14 + \sin^2\theta \right) \,d\theta dt.
%\end{equation*}
%We see that 
%\begin{equation*}
%\left| \int^{2\pi}_0 \rho \cos\theta \left( -\frac14 + \sin^2\theta \right) \,d\theta \right| \le \frac32 \rho_- \pi
%\end{equation*}
%and
%\begin{equation*}
%\frac{1}{4r^2} \int_{B_{r}} \rho x \,dx\,dy = \frac{r}{4} \int_{B_1} \rho(rx, ry) x \,dx\,dy, \quad
%\frac{1}{4r_0^2} \int_{B_{r_0}} \rho x \,dx\,dy = \frac{r_0}{4} \int_{B_1} \rho(r_0 x, r_0 y) x \,dx\,dy.
%\end{equation*}
%Thus,
%\[
%	\frac{1}{4r_0^2} \int_{B_{r_0}} \rho x \,dx\,dy = \frac{r_0}{4} \int_{B_1} \rho(r_0 x, r_0 y) x \,dx\,dy \to 0 \qquad \textrm{as } r_0 \searrow 0.
%\]
%It follows that 
%\begin{equation}\label{Lip A}
%\mathcal{A}(r) - \mathcal{A}(0+) = -\frac{r}{4} \int_{B_1} \rho(rx, ry) x \,dx\,dy - \int^{r}_0 \int^{2\pi}_0 \rho \cos\theta \left( -\frac14 + \sin^2\theta \right) \,d\theta dt = O(r).
%\end{equation}
%In particular, $\mathcal{A}(r)$ is Lipschitz continuous at $r = 0$. 

Choosing another divergence-free vector field $\phi = (x, -y)$ and repeating the above argument we have
\begin{equation}\label{eq ibp2}
2\int_{B_r} \left( u_y^2 - u_x^2 \right) \,dx\,dy + \int_{B_r} \rho y \,dx\,dy = r \int_{\partial B_r} \left( u_y^2 - u_x^2 \right) \,d\mathcal{H}^1 + \frac1r \int_{\partial B_r} \rho y (y^2 - x^2) \,d\mathcal{H}^1, 
\end{equation}
which leads directly to
\begin{equation*}
	\mathcal{B}^\prime(r) = \frac{1}{r^3} \int_{B_r} \rho y \,dx\,dy - \frac{1}{r^4} \int_{\partial B_r} \rho y (y^2 - x^2) \,d\mathcal{H}^1.
\end{equation*}
Again, the right-hand side above is uniformly bounded in $r$ with a bound depending only on $\rho_+$ and $\rho_-$, and hence $\mathcal{B}$ is Lipschitz.
%Explicitly, for almost every $r, r_0 > 0$, 
%\begin{align*}
%\mathcal{B}(r) = \ & \mathcal{B}(r_0) + \frac{1}{2r_0^2} \int_{B_{r_0}} \rho y \,dx\,dy - \frac{1}{2r^2} \int_{B_{r}} \rho y \,dx\,dy - \frac12 \int_r^{r_0}  \frac{1}{t^2}\int_{\partial B_t} \rho y \,d\mathcal{H}^1dt \\
%& + \int_r^{r_0}  \frac{1}{t^4}\int_{\partial B_t} \rho y( y^2 - x^2) \,d\mathcal{H}^1dt.
%\end{align*}
%Similarly as before,
%\begin{align*}
%& \left| -\frac12 \int_r^{r_0}  \frac{1}{t^2}\int_{\partial B_t} \rho y \,d\mathcal{H}^1dt + \int_r^{r_0}  \frac{1}{t^4}\int_{\partial B_t} \rho y (y^2 - x^2) \,d\mathcal{H}^1dt \right| \\
%& \quad = \left| \int^{r_0}_r \int^{2\pi}_0 \rho \sin\theta \left( \frac12 + \cos2\theta \right) \,d\theta dt \right| \le 3\pi \rho_- |r - r_0|,
%\end{align*}
%and
%\begin{equation*}
%\frac{1}{2r^2} \int_{B_{r}} \rho y \,dx\,dy - \frac{1}{2r_0^2} \int_{B_{r_0}} \rho y \,dx\,dy = \frac{r}{2} \int_{B_1} \rho(rx, ry) y \,dx\,dy - \frac{r_0}{2} \int_{B_1} \rho y(r_0 x, r_0 y) \,dx\,dy.
%\end{equation*}
%Thus the limit $\mathcal{B}(0+)$ exists, and sending $r_0 \searrow 0$ above we obtain again the Lipschitz continuity of $\mathcal{B}(r)$ at $r = 0$:
%\begin{equation}\label{Lip B}
%\mathcal{B}(r) - \mathcal{B}(0+) = -\frac{r}{2} \int_{B_1} \rho(rx, ry) y \,dx\,dy - \int^{r}_0 \int^{2\pi}_0 \rho \sin\theta \left( \frac12 + \cos2\theta \right) \,d\theta dt = O(r).
%\end{equation} 
This proves part~\ref{A B Lipschitz part}.

Next we turn to the bound~\eqref{energy bound}. Thanks to~\eqref{sign assumption},
\begin{equation}\label{eq slope bound}
	\int_{B_r} |u_x u_y| \, dx \, dy = \left| \int_{B_r} u_x u_y \, dx \, dy\right|.
\end{equation}
We begin with the elementary inequalities
\begin{align*}
& u_x^2 + u_y^2 = \sqrt{(u_y^2 - u_x^2)^2 + 4(u_x u_y)^2} \le  |u_y^2 - u_x^2| + 2 |u_x u_y|.
\end{align*}
%where in the second inequality we must assume (temporarily) that $u_x \neq 0$.
Therefore
\begin{equation*}
u_x^2 + u_y^2 \le 
\left\{\begin{aligned}
  &  u_y^2 - u_x^2 + 2 |u_x u_y|  \quad & \text{ when } \quad & u_y^2 \ge u_x^2, \\
  & (M+1) |u_x u_y| & \text{ when } \quad & u_y^2 \le u_x^2.
\end{aligned}\right.
\end{equation*}
We can then estimate
\begin{align*}
\int_{B_r} (u_x^2 + u_y^2) \,dx\,dy & \le  \int_{B_r \cap \{u_y^2 \ge u_x^2\}} \left( u_y^2 - u_x^2+ 2|u_x u_y| \right) \,dx\,dy + (M+1) \int_{B_r \cap \{u_y^2 \le u_x^2\}} | u_x u_y | \,dx\,dy. 
%\\
%& \le  \int_{B_r \cap \{u_y^2 \ge u_x^2\}} (u_y^2 - u_x^2) \,dx\,dy + \max{\{M+1,2\}} r^2 |\mathcal{A}(r)|.
\end{align*}
On the other hand,
\begin{equation}
  \label{first B estimate}
  r^2 \mathcal{B}(r) % & = \int_{B_r} \left( u_y^2 - u_x^2 \right) \,dx\,dy \\
  = \int_{B_r \cap \{u_y^2 \ge u_x^2\}} \left( u_y^2 - u_x^2 \right) \,dx\,dy - \int_{B_r \cap \{u_y^2 \le u_x^2\}} \left( u_x^2 - u_y^2 \right) \,dx\,dy.
\end{equation}
From \eqref{eq slope bound}, it follows that
\[
	\int_{B_r \cap \{u_y^2 \le u_x^2\}} \left( u_x^2 - u_y^2 \right) \,dx\,dy \le M \int_{B_r \cap \{u_y^2 \le u_x^2\}} |u_x u_y| \,dx\,dy
\]
and thus rearranging~\eqref{first B estimate} yields
\[
	\int_{B_r \cap \{u_y^2 \ge u_x^2\}} \left( u_y^2 - u_x^2 \right) \,dx\,dy \le r^2 \mathcal{B}(r) +M \int_{B_r \cap \{u_y^2 \le u_x^2\}} |u_x u_y| \,dx\,dy.
\]
In total, then, we find that
\begin{align*}
\frac{1}{r^2} \int_{B_r} (u_x^2 + u_y^2) \,dx\,dy & \le  \mathcal{B}(r) + \frac{2M+1}{r^2}  \int_{B_r \cap \{ u_y^2 \leq u_x^2\}} |u_x u_y| \,dx\,dy \\
	& \qquad + \frac{2}{r^2}  \int_{B_r \cap \{ u_y^2 \geq u_x^2\}} |u_x u_y| \,dx\,dy \\
	& \leq  \mathcal{B}(r)  + \max{\{ 2M+1,2\}} |\mathcal{A}(r)|.
\end{align*}
Using the Lipschitz continuity of $\mathcal{A}$ and $\mathcal{B}$ at $0$, this yields the claimed bound~\eqref{energy bound}.
\end{proof}

\subsection{Monotonicity formula for internal waves}
\label{monotonicity formula section}

The main goal of this section is to develop an analogue of the Vărvăruca--Weiss monotonicity formula that applies to the internal wave problem.  We will then be able to extract convergent blowup subsequences that resolve the structure of the free boundary in a neighborhood of the origin.

\begin{theorem}[Internal wave monotonicity formula] \label{monotonicity theorem}
  Suppose that $u$ is a variational solution to the free boundary elliptic problem~\eqref{elliptic PDE}.  For $0 < r < 1$, define 
  \begin{equation}
    \label{definition M(r)}
    \begin{aligned}
      \mathcal{M}(u; \, r)  & \colonequals \frac{1}{r^3} \int_{B_r} \left( |\nabla u|^2 - y \left( \rho_+ \chi^+(u) + \rho_- \chi^-(u) \right) \right) \, dx \, dy \\
      & \qquad - \frac{3}{2} \frac{1}{r^4} \int_{\partial B_r} u^2 \, d\mathcal{H}^1.
    \end{aligned}
  \end{equation}
  Then, $\mathcal{M}(u; \placeholder)$ is almost everywhere differentiable with derivative given by 
  \begin{equation}
    \label{dMdr formula}
    \partial_r \mathcal{M}(u; \, r) \equalscolon \mathcal{M}^\prime(r) = \frac{2}{r^3} \int_{\partial B_r} \left( \nabla u \cdot \nu - \frac{3}{2} \frac{u}{r} \right)^2 \, d\mathcal{H}^1.
  \end{equation}
\end{theorem}
\begin{remark}
Here and in what follows, we adopt the common convention of suppressing the dependence of $\mathcal{M}$ on $u$, so that it is a real-valued function solely of $r$. 
\end{remark}
\begin{proof}
  Before beginning the main argument, let us pause to note that for any variational solution $u$, we have the integration by parts identity
  \begin{equation}
    \label{integration by parts identity}
    \int_{B_r} |\nabla u|^2 \, dx \, dy = \int_{\partial B_r} u \nabla u \cdot \nu \, d \mathcal{H}^1
  \end{equation}
  for a.e.~$r \in (0,1)$.  This can be derived by setting 
  \[
  	\zeta_\delta = \max\{ u- \delta, 0\}^{1+\delta} - |\min\{ u+\delta, 0\}|^{1+\delta} \qquad \textrm{for } \delta > 0,
\]
and computing
  \[
    \int_{B_r} \nabla u \cdot \nabla \zeta_\delta \, dx \, dy = \int_{\partial B_r} \zeta_\delta \nabla u \cdot \nu \, d\mathcal{H}^1,
  \]
  which becomes~\eqref{integration by parts identity} upon sending $\delta \searrow 0$.

  Now, fixing a variational solution $u$ and, for each $0 < r \ll 1$, define 
  \begin{equation}
    \label{definition I(r) and J(r)}
    \mathcal{I}(r) \colonequals \frac{1}{r^3} \int_{B_r} \left( |\nabla u|^2 - y \left( \rho_+ \chi^+(u) + \rho_- \chi^-(u) \right) \right) \, dx \, dy, 
    \quad
    \mathcal{J}(r) \colonequals \frac{1}{r^4} \int_{\partial B_r} u^2 \, d\mathcal{H}^1
  \end{equation}
  so that 
  \[
    \mathcal{M} = \mathcal{I} - \frac{3}{2} \mathcal{J}.
  \]
  Since $u \in H^1(B_1) \cap C^0(B_1)$, both $\mathcal{I}$ and $\mathcal{J}$ are clearly almost every differentiable, and an elementary computation shows that
  \begin{equation}
    \label{dJdr formula}
    \mathcal{J}^\prime(r) = \frac{2}{r^4} \int_{\partial B_r} u \nabla u \cdot \nu \, d\mathcal{H}^1 - \frac{3}{r^5} \int_{\partial B_r} u^2 \, d\mathcal{H}^1
  \end{equation}
  for almost every $0 < r \ll 1$. It remains only to compute the derivative of $\mathcal{I}$.  Applying Lemma~\ref{variational solution identity} using the test vector field $\phi = (x,y)/r$ yields 
  \begin{equation}
  \label{main testing identity boundary}
    \begin{aligned}
      0 & = -3\int_{B_r} \left( \rho_+ \chi^+ + \rho_- \chi^- \right) y \, dx \, dy \\
      & \qquad -r \int_{\partial B_r} \left( |\nabla u|^2 - \left( \rho_+ \chi^+ + \rho_- \chi^- \right) y -2 (\nabla u \cdot \nu)^2\right) \, d\mathcal{H}^1.
    \end{aligned}
\end{equation}
On the other hand, simply differentiating $\mathcal{I}$, we find that
\begin{align*}
  \mathcal{I}^\prime(r) & = -\frac{3}{r^4} \int_{B_r} \left( |\nabla u|^2 - y \left( \rho_+ \chi^+ + \rho_- \chi^- \right) \right) \, dx \, dy \\
    & \qquad + \frac{1}{r^3} \int_{\partial B_r} \left( |\nabla u|^2 - y \left( \rho_+ \chi^+ + \rho_- \chi^- \right) \right) \, d\mathcal{H}^1 \\
    & = \frac{2}{r^3} \int_{\partial B_r} \left( \nabla u \cdot \nu \right)^2 \, d\mathcal{H}^1 - \frac{3}{r^4} \int_{B_r} |\nabla u|^2 \, dx \, dy, 
\end{align*}
where the second equality follows from~\eqref{main testing identity boundary}.  Using the integration by parts identity~\eqref{integration by parts identity} on the second term on the right-hand side above then leads to
\[
  \mathcal{I}^\prime(r) = \frac{2}{r^3} \int_{\partial B_r} \left( \nabla u \cdot \nu \right)^2 \, d\mathcal{H}^1 - \frac{3}{r^4} \int_{\partial B_r}u \nabla u \cdot \nu \, d \mathcal{H}^1.
\]

Finally, combining the formulas above $\mathcal{I}^\prime$ with the formula for $\mathcal{J}^\prime$ in~\eqref{dJdr formula} gives
\begin{align*}
  \mathcal{M}^\prime(r) & = \mathcal{I}^\prime(r) - \frac{3}{2} \mathcal{J}^\prime(r) \\
    & = \frac{2}{r^3} \int_{\partial B_r} \left( \nabla u \cdot \nu \right)^2 \, d\mathcal{H}^1 - \frac{6}{r^4} \int_{\partial B_r} u \nabla u \cdot \nu \, d\mathcal{H}^1 + \frac{9}{2 r^5} \int_{\partial B_r} u^2 \, d\mathcal{H}^1,
\end{align*}
which again holds for almost every $0 < r \ll 1$. Factoring the integrand yields~\eqref{dMdr formula}.
\end{proof}

Assuming additional decay properties, the monotonicity formula implies the existence of a $3/2$-homogeneous blowup limit. 

\begin{lemma}[Blowup limit]
\label{blowup limits lemma}
Suppose that $u$ is a variational solution to the free boundary problem~\eqref{elliptic PDE} exhibiting the decay~\eqref{blowup sequence Lipschitz bound}, and define $\mathcal{M}$ as in~\eqref{definition M(r)}.
\begin{enumerate}[label=\rm(\alph*\rm)]
	\item The right-hand side limit $\mathcal{M}(u; \, 0+)$ exists and is finite.
	\item Given any sequence $r_m \searrow 0$, let $u_m$ be the corresponding blowup sequence~\eqref{blowup sequence}. Possibly passing to a subsequence, there exists $u_0 \in H_\loc^1(\mathbb{R}^2)$ such that $u_m \to u_0$ strongly in $H_\loc^1(\mathbb{R}^2)$ and $C_\loc^\varepsilon(\mathbb{R}^2)$ for all $\varepsilon \in (0,1)$.  Moreover,  $u_0$ is homogeneous of degree $3/2$ in that $u_0(\lambda^{3/2} \placeholder) = \lambda^{3/2} u_0$ for all $\lambda \in [0,\infty)$.
\end{enumerate}
\end{lemma}
\begin{proof}
We verified in~\eqref{dMdr formula} that $\mathcal{M}(u; \placeholder)$ is nondecreasing, and it is bounded due to \eqref{double stagnation decay} and \eqref{blowup sequence Lipschitz bound}. Thus  $\mathcal{M}(u; \, 0+)$ exists and is finite. 

Let $r_m \searrow 0$ be given and define $u_m $ as in \eqref{blowup sequence}.  Then on each compact subset of $\mathbb{R}^2$, the sequence $\{u_m\}$ is uniformly bounded in $H^1$ by \eqref{double stagnation decay} and uniformly Lipschitz by~\eqref{blowup sequence Lipschitz bound}.  Thus there exists $u_0 \in H_\loc^1(\mathbb{R}^2) \cap \Lip_\loc(\mathbb{R}^2)$ so that $u_m$ converges to $u_0$ weakly in $H_\loc^1(\mathbb{R}^2)$ and strongly in $C^\varepsilon_\loc(\mathbb{R}^2)$ for all $\varepsilon \in (0,1)$.  

Our first claim is that this can be upgraded to strong convergence in $H_\loc^1$.  As $\nabla u_m \to \nabla u$ weakly in $L^2_\loc(\mathbb{R}^2)$, it suffices to prove that
\[
	\limsup_{m \to \infty} \int_{\mathbb{R}^2} |\nabla u_m |^2 \zeta \, dx \, dy \leq \int_{\mathbb{R}^2} |\nabla u_0|^2 \zeta \, dx \, dy
\]
for all $\zeta \in C_c^1(\mathbb{R}^2)$.  Each $u_m$ is harmonic outside its zero-set, and so the $C^\varepsilon_\loc$ convergence implies the same is true of $u_0$. Then, applying a version of the integration by parts identity~\eqref{integration by parts identity}, we infer that
\begin{align*}
	\int_{\mathbb{R}^2} |\nabla u_m|^2 \zeta \, dx \, dy & = -\int_{\mathbb{R}^2} u_m \nabla u_m \cdot \nabla \zeta \, dx \, dy \\
		& \leq \int_{\mathbb{R}^2} |u_m - u_0| |\nabla u_m \cdot \nabla \zeta| \, dx \, dy - \int_{\mathbb{R}^2} u_0 \nabla u_m \cdot \nabla \zeta \, dx \, dy \\
		& \longrightarrow -\int_{\mathbb{R}^2} u_0 \nabla u_0 \cdot \nabla \zeta \, dx \, dy = \int_{\mathbb{R}^2} |\nabla u_0|^2 \zeta \, dx \, dy
\end{align*}
as $m \to \infty$. Thus $u_m \to u_0$ strongly in $H_\loc^1$.

Next, we show that $u_0$ is homogeneous of degree $3/2$.  Indeed, integrating the formula for $\mathcal{M}^\prime$ in~\eqref{dMdr formula} gives
\begin{align*}
	\mathcal{M}(u; \, R_2) - \mathcal{M}(u; \, R_1) & =  \int_{R_1}^{R_2} \frac{2}{r^3} \int_{\partial B_r} \left( \nabla u\cdot \nu - \frac{3}{2} \frac{u}{r} \right)^2 \, d\mathcal{H}^1 \, dr ,
\end{align*}
for almost every $0 < R_1 < R_2$.  Now, setting $R_1 = r_m \sigma_1$ and $R_2 = r_m \sigma_2$ and rescaling above yields 
\begin{align*}
	\mathcal{M}(u; \,  r_m \sigma_2) - \mathcal{M}(u; \, r_m \sigma_1) & = \int_{R_1}^{R_2} \frac{2}{r^3} \int_{\partial B_r} \left( \nabla u\cdot \nu - \frac{3}{2} \frac{u}{r} \right)^2 \, d\mathcal{H}^1 \, dr  \\
		& = 2 \int_{B_{\sigma_2}\setminus B_{\sigma_1}} \frac{1}{|(x,y)|^5} \left( (x,y) \cdot \nabla u_m - \frac{3}{2} u_m \right)^{2} \, dx \, dy.
\end{align*}
Sending $m \to \infty$, and recalling that $u_m \to u_0$ in $H_\loc^1$, we then arrive at the identity
\[
	\int_{B_{\sigma_2} \setminus B_{\sigma_1}} \frac{1}{|(x,y)|^5} \left( (x,y) \cdot \nabla u_0 - \frac{3}{2}u_0 \right)^2 \, dx \, dy = 0, 
\]
which verifies the claimed homogeneity property.
\end{proof}

\subsection{Lower bounds on the velocity}
\label{lower bound section}

The purpose of this section is to prove that if $u$ is a Lipschitz continuous monotone variational solution and if the Dirichlet energy of $u$ decays according to~\eqref{double stagnation decay}, then $\Gamma(u)$ is locally flat at $0$.  By a classical maximum principle result, flatness furnishes a lower bound on the decay of $|\nabla u|$ as we approach a stagnation point, which we will show is incompatible with~\eqref{double stagnation decay}.  This is the central contradiction used in the next section to rule out double stagnation formation. 

Define the \emph{Stokes corner} (pseudo) stream function in the lower half-plane by
\begin{equation}
\label{definition Stokes corner}
	\Stokesu(x,y) \colonequals 
    \left\{ \begin{aligned}
      -\frac{\sqrt{2}}{3} |(x,y)|^{3/2} \cos{\left( \frac{3}{2} \left( \theta - \frac{3\pi}{2} \right) \right)} & \qquad \textrm{for } \theta \in (\tfrac{7\pi}{6}, \tfrac{11\pi}{6}) \\
      0 & \qquad \textrm{for } \theta \not\in (\tfrac{7\pi}{6}, \tfrac{11\pi}{6}),
    \end{aligned} \right.
\end{equation}
where $\theta \in [0,2\pi)$ is the polar angle for $(x,y)$. One can verify that $\Stokesu$ is a variational solution to the free boundary problem~\eqref{elliptic PDE} with $\rho_+ = 0$ and $\rho_- = 1$. The next lemma states that the blowup limit either vanishes identically or is a linear combination of rotated Stokes corner solutions having disjoint support. Monotonicity prevents all but the vanishing scenario, but we make the effort to treat the general case. Although we are argue along the lines of \cite[Proposition 4.7]{varvaruca2011geometric}, having two fluids necessitates a more complicated analysis to discern the limiting configuration.

\begin{figure}
	\centering
	\includegraphics{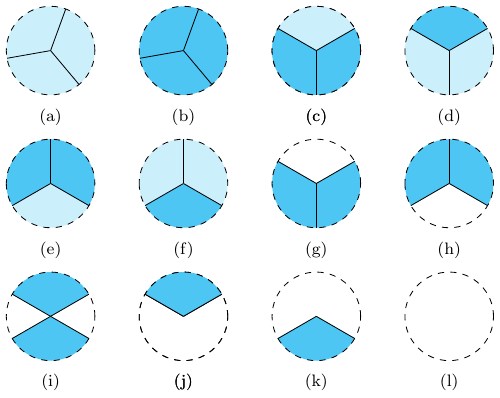}
	\caption{Possible blowup limits $u_0$ in a neighborhood of the origin. The components of $\{u_0 < 0\}$ are shaded in darker blue, components of  $\{u_0 > 0\}$ are shaded in lighter blue, $\inter\{ u =0\}$ is left white, and $\Gamma(u_0)$ is indicated by the (non-dashed) lines. Both (a) and (b) correspond to~\eqref{three component single phase blowup u}; note that any rotation of these two configurations is also possible. Cases (c)--(h) are captured by the ansatz~\eqref{three component two phase blowup u}, while (i)--(k) correspond to~\eqref{two Stokes flows}. For monotone solutions, only the vanishing limit of Case (l) can occur.}
	\label{u0 figure}
\end{figure}

\begin{lemma}[Stokes corner(s) or vanishing]
\label{limit of chi lemma}
Suppose that $u$ is a variational solution that exhibits the decay~\eqref{blowup sequence Lipschitz bound}, and in a neighborhood of $0$, the free boundary $\Gamma(u)$ is a union of finitely many Lipschitz curves.
\begin{enumerate}[label=\rm(\alph*\rm)]
	\item  For a sequence $r_m \searrow 0$, let $\{ u_m \}$ be the blowup sequence given by \eqref{blowup sequence} and let $u_0$ be a blowup limit as in Lemma~\ref{blowup limits lemma}. Then one of the following alternatives must hold. 
	\begin{enumerate}[label=\rm(\roman*\rm)]
		\item \textup{(Single phase)} There exists $\alpha \in \mathbb{R} \setminus \{ 0\}$ and $\theta \in [0,2\pi)$ such that
      \begin{equation}
        \label{three component single phase blowup u}
        u_0 = \alpha\left( \Stokesu + \Stokesu \circ A_{\tfrac{2\pi}{3}} + \Stokesu \circ A_{\tfrac{2\pi}{3}}^{-1} \right) \circ A_{\theta}  ,
      \end{equation}
		where $A_\theta$ denotes rotation by $\theta$ in the counterclockwise direction.
		\item  \textup{(Adjacent Stokes corners)} For some $\theta \in \{ 0, \pi\}$, and  either $\alpha_1 \leq 0$ and $\alpha_2 > 0$, or $\alpha_1 < 0$ and $\alpha_2 > 0$,
      \begin{equation}
        \label{three component two phase blowup u}
        u_0  = \left( \alpha_1 \Stokesu + \alpha_2 \Stokesu \circ A_{\tfrac{2\pi}{3}} + \alpha_2 \Stokesu \circ A_{\tfrac{2\pi}{3}}^{-1} \right) \circ A_{\theta}.
      \end{equation}
		\item \textup{(Opposing Stokes corners)} For some $\alpha_1, \alpha_2 \in \{ 0,\sqrt{|\rho_+-\rho_-|}\}$,
      \begin{equation}
        \label{two Stokes flows} 
        u_0 = \alpha_1 \Stokesu + \alpha_2 \Stokesu \circ A_{\pi}.
      \end{equation}
		Note that this includes the degenerate case $u_0 = 0$ and the situation where there is a single Stokes corner.
	\end{enumerate}
	\item \label{limiting chi part} Possibly passing to a subsequence, we have that 
	\[
		\chi^\pm(u_m) \longrightarrow \chi_0^\pm \qquad \textrm{in } L_\loc^1(\mathbb{R}^2), 
	\]
	where $\chi_0^\pm$ takes values in $\{0,1\}$, $\chi_0^\pm = 1$ and $\chi_0^\mp =0$ on $\{ \pm u_0 > 0\}$. Moreover, $\rho_+ \chi_0^+ +\rho_- \chi_0^-$ is constant on each connected component of $\inter\{ u_0 = 0 \} \cap \{ y > 0\}$ and $\inter\{ u_0 = 0\} \cap \{ y < 0\}$.
	\item \label{limiting u and chi with monotonicity part} \label{limiting chi monotone part} If $u$ is monotone, then necessarily $u_0 = 0$ and $\chi_0^\pm = \chi_{\{ \pm y > 0 \}}$ or $\chi_0^\pm = \chi_{\{ \mp y > 0\}}$.
\end{enumerate}
\end{lemma}

\begin{proof}
  Denote $\chi_m^\pm \colonequals \chi^\pm(u_m)$.  By hypothesis, $\Gamma(u)$ is a finite union of Lipschitz graphs near $0$, and hence $\Gamma(u_m)$ is likewise the union of Lipschitz graphs on $B_1$ for $m$ sufficiently large (with the same Lipschitz constant as $\Gamma(u)$). In particular,  $\{ \chi_m^\pm \}$ is uniformly bounded in $BV$.  By the compactness of the embedding $BV \subset\subset L_\loc^1$, passing to a subsequence we have that $\chi_m^\pm \to \chi_0^\pm$ in $L_\loc^1$, for some $\chi_0^\pm \in L_\loc^1(\mathbb{R}^2)$.  

  It follows from the definition of a variational solution \eqref{definition critical point} and the fact that $u_m \to u_0$ in $H_\loc^1 \cap C_\loc^\varepsilon$ for each $\varepsilon \in (0,1)$, that
  \begin{equation}
    \label{u0 variational equation}
    0 = \int_{\mathbb{R}^2} \left(|\nabla u_0|^2 \nabla \cdot \phi - 2 D\phi[ \nabla u_0, \nabla u_0] + \left(  \rho_+ \chi_0^+ + \rho_- \chi_0^- \right) \nabla\cdot (y \phi) \right) \, dx \,dy
  \end{equation}
  for all vector fields $\phi \in C_c^1(\mathbb{R}^2; \mathbb{R}^2)$. Since the range of $\chi_m^\pm$ is a subset of $\{0,1\}$, the same is true of $\chi_0^\pm$.  Then we may infer from~\eqref{u0 variational equation} that $\rho_+ \chi_0^+ + \rho_- \chi_0^-$ is constant on the interior of the sets $\{ u_0 = 0\} \cap \{ y > 0\}$ and $\{ u_0 = 0 \} \cap \{ y < 0 \}$.   The uniform convergence of $u_m \to u_0$ further implies that $\chi_0^\pm = 1$ and $\chi_0^\mp = 0$ on $\{ \pm u_0 > 0\}$. This proves~\ref{limiting chi part}.

  First suppose that $u_0$ vanishes identically, in which case the ansatz \eqref{two Stokes flows} holds with $\alpha_1 = \alpha_2 = 0$. If in addition $u$ is monotone, then $\Gamma(u)$ is a Lipschitz graph over $x$ and includes $0$, hence there is a cone above the origin that lies in one fluid and cone below the origin that lies in the other fluid. If $\pm u > 0$ locally above $0$ and $\pm u < 0$ locally below $0$, it follows that $\chi_0^\pm \equiv 1$ $\chi_0^\mp \equiv 0$ on $\{ y > 0\}$ while $\chi_0^\pm \equiv 0$ and $\chi_0^\mp \equiv 1$ on $\{ y < 0\}$. This proves that $\chi_0^\pm$ are as in~\ref{limiting u and chi with monotonicity part} provided $u_0 \equiv 0$. 
  %The same reasoning shows that it also holds if $u_0 = \alpha \Stokesu$, since $\Stokesu$ vanishes identically on $\{ y > 0\}$.

  Assume next that $u_0 \not\equiv 0$. As $u_m \to u_0$ locally uniformly, we have that $u_0$ is harmonic on each connected component of $\{ u_0 = 0\}^c$ and homogeneous of degree $3/2$ by Lemma~\ref{blowup limits lemma}.  The connected components of $\{ u_0 = 0 \}^c$ are necessarily cones centered at the origin with interior angle $120\degree$. Indeed, an elementary computation shows that if $\mathcal{K}$ is such a cone with sides at $\theta = \theta_{\mathcal{K}}$ and $\theta = \theta_{\mathcal{K}}+2\pi/3$, then 
  \begin{equation}
    \label{cone blowup ansatz}
    u(x,y) = \alpha_{\mathcal{K}} |(x,y)|^{3/2} \cos{\left( \frac{3}{2} \left( \theta - \theta_{\mathcal{K}} \right) + \frac{\pi}{2} \right)} \qquad \textrm{on } \mathcal{K},
  \end{equation}
  for some $\alpha_{\mathcal{K}} \in \mathbb{R}$. Let $\rho_{\mathcal{K}} = \rho_+$ if $\mathcal{K} \subset \Omega^+(u_0)$ and set $\rho_{\mathcal{K}} = \rho_-$ otherwise. 
  By the same reasoning, along $\{ \theta=\theta_{\mathcal{K}}\}$ is $\mathcal{K}$ is adjacent to either (i) a cone where $u_0 \equiv 0$ or (ii) a $120\degree$ cone on which $u_0 > 0$ or $u_0 < 0$. Let $\mathcal{D}_1$ denote this region, and likewise let $\mathcal{D}_2$ be the cone adjacent to $\mathcal{K}$ along $\{ \theta = \theta_{\mathcal{K}} + 2\pi/3\}$. It is also important to remark that if we assume monotonicity of $u$, then uniform convergence of $u_m \to u_0$ forces
  \begin{equation}
    \label{monotonicity consequence u0}
    \{ y = 0 \} \setminus \{ 0 \} \not\subset \{ u_0 > 0\} \qquad \textrm{and} \qquad  \{ y = 0 \} \setminus \{ 0 \} \not\subset  \{ u_0 < 0\}.
  \end{equation}
  Were this to fail, then there would be a sequence of points along the $x$-axis converging to $0$ from the right and from the left that lie in the same fluid layer.

  The outward unit normal $\nu$ to $\mathcal{K}$ is constant on each component of $\partial \mathcal{K} \setminus \{ 0\}$. Fixing any $(x_0, y_0) \in \partial\mathcal{K} \setminus \{ 0\}$, and taking $\phi = \nu \zeta$ in \eqref{u0 variational equation} for $\zeta \in C_0^1(B_1)$ a cutoff function, we find that 
  \begin{equation}
    \label{u0 boundary condition}
    \left\{
      \begin{aligned}
        \frac{9}{4} \left( | \alpha_{\mathcal{K}}|^2 - |\alpha_1|^2 \right) & = (\rho_1-\rho_{\mathcal{K}})  \sin{(\theta_{\mathcal{K}})}  \\
        \frac{9}{4} \left( | \alpha_{\mathcal{K}}|^2 - |\alpha_2|^2 \right) & = (\rho_2 - \rho_{\mathcal{K}})  \sin{(\theta_{\mathcal{K}}+ \tfrac{2\pi}{3})},
      \end{aligned}
    \right.
  \end{equation}
  where $\alpha_j$ is the coefficient for $u$ in $\mathcal{D}_j$ and $\rho_j = \rho_\pm$ if $\mathcal{D}_j \subset \Omega^\pm(u_0)$. Note that $\alpha_j = 0$ if $\mathcal{D}_j$ is not a $120\degree$ cone with vertex at the origin. We pause here to gather some consequences of~\eqref{u0 boundary condition}. First, observe that if $\rho_{\mathcal{K}} = \rho_j$, then $\alpha_{\mathcal{K}} = \alpha_j$. In particular, a cone on which $u_0 > 0$ cannot be adjacent to a region where $u_0 \equiv 0$. Likewise, if $|\alpha_1|^2 = |\alpha_2|^2$ and $\rho_1 = \rho_2 \neq \rho_{\mathcal{K}}$, then one can solve~\eqref{u0 boundary condition} to find that 
  \begin{equation}
    \label{thetaK alphaK formulas}
    \theta_{\mathcal{K}} \in \left\{  \frac{\pi}{6},  \frac{7\pi}{6}\right\} \qquad \textrm{and} \qquad  |\alpha_{\mathcal{K}}|^2 =  |\alpha_1|^2\pm\frac{2}{9} ( \rho_{\mathcal{K}} - \rho_1),
  \end{equation}
  with $+$ corresponding to the case $\theta_{\mathcal{K}} = \pi/6$ and $-$ to $\theta_{\mathcal{K}}=7\pi/6$. 

  Consider first the possibility that $\{ u_0 = 0\}^c$ has exactly three connected components
  \[
    \mathcal{K}_1 \colonequals \{ 0 < \theta-\theta_0 < \tfrac{2\pi}{3} \}, \quad 
  \mathcal{K}_2 \colonequals \{ \tfrac{2\pi}{3} < \theta-\theta_0 < \tfrac{4\pi}{3} \}, \quad
  \mathcal{K}_3 \colonequals \{ \tfrac{4\pi}{3} < \theta-\theta_0 < 2\pi \},
\]
  where we may assume that $\theta_0 \in [0,2\pi/3)$. The restriction $u|_{\mathcal{K}_j}$ therefore has the form \eqref{cone blowup ansatz}; we will abbreviate the coefficients $\alpha_j \colonequals \alpha_{\mathcal{K}_j}$ for $j = 1, 2, 3$. Since each $\mathcal{K}_j$ has internal angle $120\degree$, there are at least two cones on which $u$ has the same sign. Without loss of generality, suppose that $\mathcal{K}_1$ is the other cone. If $u$ has the same sign on $\mathcal{K}_1$ as on $\mathcal{K}_2$ and $\mathcal{K}_3$, then it must take the form~\eqref{three component single phase blowup u}. It is easy to see that this incompatible with monotonicity. If the signs differ, then it follows from~\eqref{thetaK alphaK formulas} that $\theta_0 = \pi/6$ or $\theta_0 = 7\pi/6$, and hence $u_0$ is of the form~\eqref{three component two phase blowup u}. Were this to hold, then $\{ y =0\} \setminus \{ 0\} \subset \mathcal{K}_2 \cup \mathcal{K}_3$. As this contradicts monotonicity~\eqref{monotonicity consequence u0}, we infer that $\{ u_0 = 0\}^c$ cannot have three connected components in the setting of part~\ref{limiting u and chi with monotonicity part}. 

  Suppose instead that $\{ u_0 = 0\}^c$ has two connected components. Recycling notation, we call them $\mathcal{K}_1$ and $\mathcal{K}_2$. Again, each must be a $120\degree$ cone with vertex on the origin. If $\mathcal{K}_j \subset \{ u_0 > 0 \}$ for $j =1$ or $j=2$, then necessarily $\mathcal{K}_j$  is adjacent to a region where $u \equiv 0$, which we have already seen forces $u_0 \equiv 0$ on $\mathcal{K}_j$, a contradiction. Thus, $\mathcal{K}_1, \mathcal{K}_2 \subset \{ u < 0\}$. From~\eqref{u0 boundary condition} it is easy to see that there are only two possible configurations: 
  \[
    \overline{\mathcal{K}_1 \cup \mathcal{K}_2} = 
    \left\{ \begin{aligned}
      \{ \theta \in [\tfrac{\pi}{6}, \tfrac{5\pi}{6} ] \cup [\tfrac{7\pi}{6}, \tfrac{11\pi}{6}] \}  & \qquad \textrm{if } \partial\mathcal{K}_1 \cap \partial \mathcal{K}_2 = \{0\} \\
      \{ \theta \in (\tfrac{\pi}{6}, \tfrac{5\pi}{6} ) \}^c \quad \textrm{or} \quad \{ \theta \in (\tfrac{7\pi}{6}, \tfrac{11\pi}{6} ) \}^c	& \qquad \textrm{if } \partial\mathcal{K}_1 \cap \partial \mathcal{K}_2 \neq \{0\}.
    \end{aligned} \right.
  \]
The first of these corresponds to two Stokes corner flows, one centered along the positive $y$-axis and the second along the negative $y$-axis; this is captured by~\eqref{two Stokes flows} with $\alpha_1 = \alpha_2 = \sqrt{|\rho_+-\rho-|}$. The second possibility is contained in~\eqref{three component two phase blowup u} taking $\alpha_1 = 0$ and $\alpha_2 = \sqrt{|\rho_+-\rho_-|}$. Clearly, both of these cases violate monotonicity, and so we may exclude them in the setting of part~\ref{limiting u and chi with monotonicity part}.

Finally, if $\{u_0=0\}^c$ has exactly one connected component, then by the argument in the previous paragraphs, it must be a $120\degree$ cone $\mathcal{K} \subset \{u_0 < 0\}$. As it is bounded on both sides by a cone where $u_0 \equiv 0$, from~\eqref{thetaK alphaK formulas} we infer that $\mathcal{K} = \{ \pi/6 < \theta < 5\pi/6\}$ or $\mathcal{K} = \{ 7\pi/6 < \theta < 11 \pi/6\}$; these correspond to~\eqref{two Stokes flows} with precisely one of $\alpha_1, \alpha_2$ vanishing. Of course, thanks to Lemma~\ref{streamlines lemma}, both of these can be ruled out if $u$ is monotone, and hence in the setting of part~\ref{limiting u and chi with monotonicity part} we must have that $u_0 = 0$.
\end{proof}

\begin{figure}
	\centering
	\includegraphics{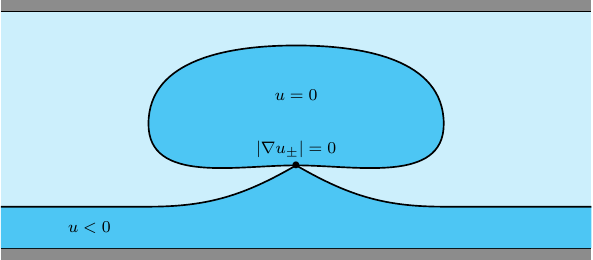}
	\caption{Conjectured limiting solitary internal wave along the family computed numerically in \cite{guan2021local}. The heavier fluid region (shaded in darker blue) is the union of what appears to be an extreme solitary gravity wave with a $120\degree$ internal angle at its crest and a ``bubble'' or ``mushroom head'' of fluid above it that is at rest in the moving frame.} 
	\label{mushroom figure}
\end{figure}

\begin{remark}
Part of the reason for characterizing $u_0$ without assuming monotonicity is to potentially address the mushroom-like limiting internal solitary waves observed numerically by~\cite{guan2021local}; see Figure~\ref{mushroom figure}. In that case, we would expect that: 
\[
	u_0 = \Stokesu, \quad \chi_0^- = \chi_{\{ y > 0 \} \cup\,\supp{\Stokesu}}, \quad \chi_0^+ = \chi_{\{ y < 0\} \setminus \supp{\Stokesu}}. 
\]
Note, however, that the solution is certainly not monotone, so a new justification for the Lipschitz bounds~\eqref{blowup sequence Lipschitz bound} would be needed.
\end{remark}

The next lemma exploits the monotonicity and the characterization of the blowup limit in Lemma~\ref{limit of chi lemma} to establish that the free boundary is flat at $0$.

\begin{figure}
\centering
\includegraphics{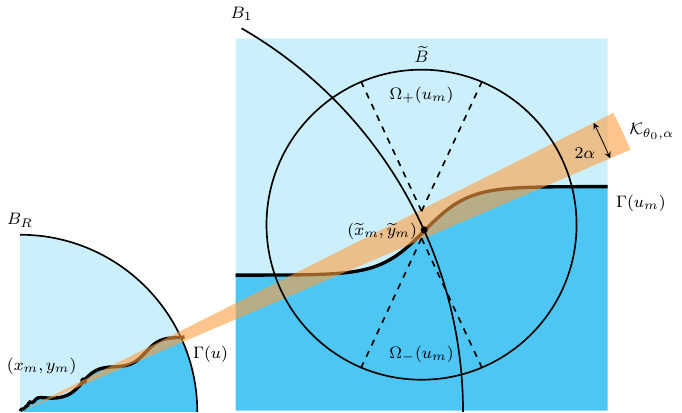}
\caption{The blowup region from the proof of Lemma~\ref{flatness lemma}. The cone $\mathcal{K}_{\theta_0,\alpha}$ is shaded in orange.  A hypothetical free boundary $\Gamma(u_m)$ for the blown up $u_m$ is draw in thick black; here, for definiteness, the upper layer (shaded in lighter blue) is taken to be $\Omega_+(u_m)$ while the lower layer (shaded in darker blue) is $\Omega_-(u_m)$. By assumption, there is a point $(\tilde x_m, \tilde y_m)$ on the unit circle and lying inside $\mathcal{K}_{\theta_0,\alpha} \cap \tilde B \cap \Gamma(u_m)$. The Lipschitz continuity of the free boundary implies that the dashed conical region above $(\tilde x_m, \tilde y_m)$ must lie in the upper fluid, while the dashed conical region below must be in the lower fluid.
}
\label{cone blowup figure}
\end{figure}

\begin{lemma}[Free boundary flatness]
\label{flatness lemma}
If $u$ is a variational solution that exhibits the decay~\eqref{blowup sequence Lipschitz bound} and is monotone, then $\Gamma(u)$ is locally the graph of a $C^1$ function of $x$ near $0$ and the normal at $0$ is purely vertical. 
\end{lemma}
\begin{proof}
First, we recall that by Lemma~\ref{streamlines lemma}, monotonicity of $u$ implies the free boundary is a Lipschitz graph. Without loss of generality, we may assume that locally, the region above the graph lies in $\Omega^+(u)$ and the region below lies in $\Omega^-(u)$. Observe then that at any point $(x_0,y_0) \in \Gamma(u)$, we have that $u$ is a solution in the viscosity sense to the free boundary elliptic PDE~\eqref{elliptic PDE}. Moreover, if $y_0 \neq 0$, then the boundary condition is locally non-degenerate, and so classical free boundary regularity theory implies that for any $r > 0$, there exists $\beta = \beta(r) \in (0,1)$ such that $\Gamma(u) \setminus B_r$ is of class $C^{1+\beta}$.

Let $\mathcal{K}_{\theta,\mu}$ denote the open cone with axis in the $\theta$ direction and having aperture $2\mu$.  We claim that for any $\mu < \pi/2$, there exists $R = R(\mu) > 0$ such that
\[
	\left( \mathcal{K}_{\frac{\pi}{2}, \mu} \cap B_R \right) \subset \Omega^+(u).
\]
Seeking a contradiction, suppose there is some $0 < \mu < \pi/2$ so that $\Gamma(u) \cap B_R \cap \mathcal{K}_{\pi/2,\mu}$ is nonempty for all $0 < R \ll 1$.  It follows then that for any $0 < \alpha \ll 1$, there exists an angle $\theta_0 \in (\pi/2-\mu,\pi/2+\mu)$ and a sequence $\{(x_m, y_m)\} \subset \Gamma(u) \cap \mathcal{K}_{\theta_0,\alpha}$ so that $|x_m| \searrow 0$.
% \[
% 	|x_m| \searrow 0 \quad \textrm{and} \quad | \arctan{(y_m/x_m)} - \theta_0| < \alpha.
%\]

Let $\tilde \mu \in (\mu, \pi/2)$ be given, and set $z_0 = (\cos{\theta_0}, \sin{\theta_0})$. Then there exists $0 < \tilde r \ll 1$, depending only on $\mu-\tilde \mu$, such that 
\[
	\tilde B \colonequals B_{\tilde r}(z_0) \subset \mathcal{K}_{\frac{\pi}{2},\tilde{\mu}}.
\]
For $r_m \colonequals |(x_m, y_m)|$, letting $u_m$ be defined by \eqref{blowup sequence}, we then have the rescaled point 
\[
	(\tilde x_m, \tilde y_m) \colonequals \left( \frac{x_m}{r_m}, \, \frac{y_m}{r_m} \right) \in \tilde B^\prime \cap \Gamma(u_m),
\]
where $\tilde B^\prime = B_{2\sin{(\alpha/2)}}(z_0)$ is a smaller ball concentric with $\tilde B$. As noted above, the free boundary is locally the graph of a Lipschitz continuous function of $x$; let $M$ be the Lipschitz constant and set $\vartheta \colonequals \pi/2 -\arctan{M}$. It follows that there exists a truncated conical region above $(\tilde x_m,\tilde y_m)$ with aperture $2 \vartheta$ that lies in either $\Omega^+(u_m) \cap \tilde B$ or $\Omega^-(u_m) \cap \tilde B$ and a truncated conical region with the same aperture below $(\tilde x_m, \tilde y_m)$ that lies in the opposite region.  Note that each of these conical regions has area that is at least $\tilde{r}^2 \vartheta-O(\alpha)$. This situation is illustrated in Figure~\ref{cone blowup figure}.  For definiteness, assume that $\rho_+ < \rho_-$, Then, for $\alpha$ sufficiently small, 
\begin{align*}
	\rho_+ & < \rho_+ \left( 1- \frac{\vartheta}{2\pi }\right) + \rho_- \frac{\vartheta}{2\pi}  \leq \frac{1}{|\tilde B|} \int_{\tilde B} \left( \rho_+ \chi_m^+ + \rho_- \chi_m^- \right) \, dx \, dy. 
\end{align*}
But this is impossible, since $\tilde B \subset \mathcal{K}_{\pi/2,\tilde \mu} \subset \{ y > 0\}$ and so by Lemma~\ref{limit of chi lemma}\ref{limiting chi monotone part} we know that
\[
	\int_{\tilde B} \left( \rho_+ \chi_m^+ + \rho_- \chi_m^- \right) \, dx \, dy \longrightarrow \rho_+ |\tilde B|.  \qedhere
\]
\end{proof}

As a consequence of the flatness, we get a {lower bound} on the decay rate of $u$ in a neighborhood of $0$.  

\begin{corollary}[Lower bound on decay]
  \label{lower bound decay lemma}
  Under the hypotheses of Lemma~\ref{flatness lemma}, for every $\mu > 1$, there exists $R = R(\mu) > 0$ and a constant $c_\mu > 0$ such that
  \begin{equation}
    \label{lower bound decay}
    |u| \geq c_\mu |(x,y)|^{\mu} \cos{\left(\mu (\theta - \pi/2)\right)} \qquad \textrm{in } \mathcal{D}_{\mu,R},
  \end{equation}
  where $\theta = \arg{(x,y)}$, and $\mathcal{D}_{\mu,R}$ is the truncated conical region 
  \begin{equation}
    \label{definition truncated cone}
    \mathcal{D}_{\mu,R} \colonequals \{ (r, \theta) : 0 < r < R,~ |\theta-\pi/2| < \pi/(2\mu) \}.
  \end{equation}
\end{corollary}
\begin{proof}
Let $\mu > 1$ be given. By the flatness established in Lemma~\ref{flatness lemma}, we know that $\mathcal{D}_{\mu,R} \subset \Omega^+(u)$ for $R$ sufficiently small.  Since $u$ is harmonic on $\mathcal{D}_{\mu,R}$, non-vanishing on $\overline{\mathcal{D}_{\mu,R}} \setminus \{0\}$, and $u(0) = 0$, the lower bound \eqref{lower bound decay} follows from Theorem~\ref{oddson theorem}. 
\end{proof}

\subsection{Proof of overturning}
\label{proof of the main theorem section}

We are now positioned to prove the main result of the paper. First, we provide the proof of Lemma~\ref{limiting solution lemma}, which allows us to conclude that if overturning does not occur, then there exists a monotone variational solution to the model free boundary problem~\eqref{elliptic PDE} that is locally Lipschitz and exhibits the energy decay~\eqref{double stagnation decay}. 

\begin{proof}[Proof of Lemma~\ref{limiting solution lemma}]
From \cite[Lemma~5.26]{chen2023global}, we see that if overturning does not occur, then we can extract a translated subsequence so that in the limit $s_n \searrow -\infty$ it holds that
\begin{equation}
 \label{elevation convergence} 
  \begin{aligned}
    \eta(s_n) & \xrightarrow{C^\varepsilon} \eta \in \Lip([-1,1]) &\quad &\textrm{for all } \varepsilon \in (0,1), \\
     -\psi(s_n) & \xrightarrow{C^\varepsilon}  u \in \Lip(B_1) &  \quad & \textrm{for all } \varepsilon \in (0,1), \\
    -\nabla  \psi(s_n) & \xrightarrow{\textrm{weak-$\ast$ } L^\infty}  \nabla u  \\
    \nabla \psi_\pm(s_n)(0) & \longrightarrow (0,0). & \quad &
  \end{aligned}
\end{equation}
Note that here, to simplify the presentation, we have performed a further rescaling and translation so that the limiting domain is $B_1$ and the stagnation point is at the origin.  It is also important to observe that the construction in~\cite{chen2023global} ensures that $u$ does not vanish identically.  Moreover, $u$ is harmonic in $\Omega_+(u) \cup \Omega_-(u)$, and these are the subset of $B_1$ above and below the graph of $\eta$, respectively. 

By construction, each $\psi(s_n)$ is a classical solution to the free boundary elliptic problem~\eqref{elliptic PDE}.  Let
\[
	\vartheta_n \colonequals \arctan{\left( -\frac{\partial_x \psi(s_n)}{\partial_y \psi(s_n)} \right)}
\] 
denote the angle relative to the positive $x$-axis made by tangent to the level sets of $\psi(s_n)$.  Then $\vartheta$ is nonnegative by \cite[Theorem 5.1(a)]{chen2023global},  harmonic in each fluid region, vanishes on the upper and lower rigid boundaries, and is bounded above $\arctan{M}$ on the free boundary, where $M \colonequals \limsup_{s} \| \partial_x \eta(s)\|_{L^\infty}$.  By the maximum principle and the non-overturning assumption~\eqref{elevation no overturning}, we therefore have $\tan{\vartheta_n} \leq M$.  It follows that $-\psi(s_n)$ is a monotone solution in the sense of Definition~\ref{definition monotone}.  

Next, we confirm that the blowup sequences of $u$ satisfy the Lipschitz bound~\eqref{blowup sequence Lipschitz bound}. Let $r_m \searrow 0$ be given and define 
\[
	u_n \colonequals -\psi(s_n), \quad u_{m,n} \colonequals -\frac{\psi(s_n)(r_m \placeholder)}{r_m^{3/2}} \qquad \textrm{for } m, n \geq 1.
\]
Note that each of these are monotone solutions to the free boundary problem with classical regularity. It suffices to prove that 
  \begin{equation}
    \label{uniform Lipschitz bound umn}
    \sup_{m \gg 1} \limsup_{n \to \infty} \| u_{m,n} \|_{\Lip(B_R)} < \infty,
  \end{equation}
for any $R > 0$.  With that in mind, let $R >0$ be given and take $m \gg 1$ so that $3 r_m R < 1$. For any $z_0 = (x_0,y_0) \in \Gamma(u_{m,n}) \cap B_{2R}$ we have that $B_{R}(z_0) \subset B_{3R}$, and hence
\begin{align*}
	& \left( \frac{1}{R^2} \int_{B_{R}(z_0)} |\nabla u_{m,n}^+ |^2 \, dx \, dy \right) \left( \frac{1}{R^2} \int_{B_{R}(z_0)} |\nabla u_{m,n}^- |^2 \, dx \, dy \right) \\
	& \qquad  \lesssim \left( \frac{1}{9R^2} \int_{B_{3R}} |\nabla u_{m,n}^+ |^2 \, dx \, dy \right) \left( \frac{1}{9R^2} \int_{B_{3R}} |\nabla u_{m,n}^- |^2 \, dx \, dy \right) \\
	& \qquad \leq \left( \frac{1}{9r_m^3R^2} \int_{B_{3 r_m R}} |\nabla u_{n} |^2 \, dx \, dy \right)^2  \\
	& \qquad \lesssim \frac{1}{r_m^2} \left( \mathcal{A}_n(0+)^2 +\mathcal{B}_n(0+)^2 \right) + \| \mathcal{A}_n\|_{\Lip}^2  + \| \mathcal{B}_n \|_{\Lip}^2,
\end{align*}
where $\mathcal{A}_n \colonequals \mathcal{A}(u_n; \placeholder)$ and $\mathcal{B}_n \colonequals \mathcal{B}(u_n; \, \placeholder).$ In the last inequality, we have appealed to Theorem~\ref{energy bound theorem}\ref{energy bound part}. It is important to note that all of these suppressed constants above depend only on $R$, $M$, $\rho_1$ and $\rho_2$; they are independent of both $m$ and $n$. Since each $\nabla\psi_i(s_n)$ is $C^{1+\alpha}$ up to the boundary, the double stagnation limit in~\eqref{elevation convergence} implies that the densities $\mathcal{A}_n(0+)$ and $\mathcal{B}_n(0+)$ limit to $0$. Theorem~\ref{energy bound theorem}\ref{A B Lipschitz part} gives uniform bounds on Lipschitz norms of $\mathcal{A}_n$ and $\mathcal{B}_n$. For each fixed $m$, we therefore have 
\begin{equation}
  \label{uniform product of energies bound}
  \limsup_{n \to \infty} \left( \frac{1}{R^2} \int_{B_{R}(z_0)} |\nabla u_{m,n}^+ |^2 \, dx \, dy \right) \left( \frac{1}{R^2} \int_{B_{R}(z_0)} |\nabla u_{m,n}^- |^2 \, dx \, dy \right)  \lesssim 1.
\end{equation}
Now, applying the Alt--Caffarelli--Friedman monotonicity formula in Theorem~\ref{acf theorem} to $u_{m,n}$ at $z_0$ furnishes the bound
\begin{align*}
	\pi^2 |\partial_\nu u_{m,n}^+(z_0)|^2 |\partial_\nu u_{m,n}^-(z_0) |^2 & = \lim_{r \searrow 0} \left( \frac{1}{r^2} \int_{B_r(z_0)} |\nabla u_{m,n}^+|^2 \, dx \, dy \right) \left( \frac{1}{r^2} \int_{B_r(z_0)} |\nabla u_{m,n}^-|^2 \, dx \, dy \right) \\
		& \leq \left( \frac{1}{R^2} \int_{B_{R}(z_0)} |\nabla u_{m,n}^+ |^2 \, dx \, dy \right) \left( \frac{1}{R^2} \int_{B_{R}(z_0)} |\nabla u_{m,n}^- |^2 \, dx \, dy \right),
\end{align*}
which combined with the estimate~\eqref{uniform product of energies bound} gives
\begin{equation}
\label{product of normal bounds}
    \limsup_{n \to \infty} |\partial_\nu u_{m,n}^+(z_0)|^2 |\partial_\nu u_{m,n}^-(z_0) |^2  \lesssim 1.
\end{equation}
 As in the proof of \cite[Theorem 5.10]{chen2023global}, we use the dynamic condition~\eqref{eqn:stream:dynamic} at $z_0$ --- with the axes shifted and the Froude number normalized as discussed above ---  to control the difference between the normal gradients:
\[
	 \left| |\partial_\nu u_{m,n}^+(z_0)|^2 - |\partial_\nu u_{m,n}^-(z_0)|^2 \right| \lesssim 1+|y_0| \lesssim 1.
\]
Together with~\eqref{product of normal bounds}, this implies that for each fixed $m$, $\partial_\nu u_{m,n}^+$ and $\partial_\nu u_{m,n}^-$ are uniformly bounded along $\Gamma(u_{m,n}) \cap B_{2R}$ in terms of just $R$, $M$, $\rho_1$, and $\rho_2$. Because each $\pm u_{m,n}^\pm$ is harmonic and nonnegative on $\Omega^\pm(u_{m,n})$, it then follows from the Harnack inequality that $u_{m,n} \in \Lip(B_{R})$ with a Lipschitz bound likewise independent of $n$; see, for example, \cite[Lemma 11.19]{caffarelli2005geometric}. This completes the proof of \eqref{uniform Lipschitz bound umn}, and hence that $u$ satisfies~\eqref{blowup sequence Lipschitz bound}.

By construction, $u$ satisfies~\eqref{monotonicity assumptions}, so it remains only to show that it is a variational solution. First, we claim that $u$ is a classical solution on $B_1 \setminus B_\delta$ for each $\delta > 0$. To see this, note that because $\eta$ is monotone, $|y|$ is uniformly bounded away from $0$ on $\Gamma(u) \setminus B_\delta$. Were this not true, then it would have to be that $\Gamma(u) \cap B_\delta$ is purely horizontal, but this is impossible as the Hopf lemma could then be applied at $0$, and this would contradict the decay bound~\eqref{double stagnation decay}. It follows that $|y|$ is bounded uniformly away from $0$ on $\Gamma_n \setminus B_\delta$, where $\Gamma_n \colonequals \Gamma(\psi(s_n))$. Arguing as in \cite[Corollary 5.27]{chen2023global}, we conclude that there exists $\beta \in (0,1)$ such that $\Gamma_n \setminus B_\delta$ is of class $C^{1+\beta}$ for all $n \gg 1$, and hence $\Gamma(u) \setminus B_\delta$ is of class $C^{1+\beta/2}$. The elliptic estimates from~\cite[Section 5.6]{chen2023global} then imply that $u$ is a classical solution on this set, and so in particular it satisfies~\eqref{definition critical point} for all vector fields $\phi \in C_0^1(B_1\setminus B_\delta ; \mathbb{R}^2)$.  

Now, consider the situation near the origin.  Let $\zeta_\delta \in C_c^\infty(B_1)$ be a cutoff function satisfying
\begin{align*}
	0 \leq \zeta_\delta \leq 1, \qquad
	\supp{\zeta_\delta} \subset B_\delta, \qquad \zeta_\delta = 1 \ona B_{\delta/2}, \qquad |\nabla \zeta_\delta| \lesssim \frac{1}{\delta}.
\end{align*}
For a vector field $\phi \in C_0^1(B_1; \mathbb{R}^2)$, we may therefore write 
\[
	\phi = \zeta_\delta \phi + (1-\zeta_{\delta}) \phi \equalscolon \phi_0 + \phi_1.
\]
Since $\phi_1 \in C_0^1(B_1 \setminus B_\delta; \mathbb{R}^2)$, it follows from the previous paragraph that
\begin{align*}
	& \int_{B_1} \left( |\nabla u|^2 (\nabla \cdot \phi) - 2 D \phi[\nabla u,\nabla u] - ( \rho_+ \chi^+(u) + \rho_- \chi^-(u) ) \nabla \cdot (y \phi) \right) \, dx \, dy \\
	& \qquad = \int_{B_1} \left( |\nabla u|^2 (\nabla \cdot \phi_0) - 2 D \phi_0[\nabla u,\nabla u] - ( \rho_+ \chi^+(u) + \rho_- \chi^-(u) ) \nabla \cdot (y \phi_0) \right) \, dx \, dy \\
	& \qquad = \int_{B_\delta} \zeta_\delta \left( |\nabla u|^2 (\nabla \cdot \phi) - 2 D \phi[\nabla u,\nabla u] - ( \rho_+ \chi^+(u) + \rho_- \chi^-(u) ) \nabla \cdot (y \phi) \right) \, dx \, dy \\
	& \qquad \qquad + \int_{B_{\delta}} \left( |\nabla u|^2 ( \phi \cdot \nabla \zeta_\delta) - 2 (\phi \cdot \nabla u) ( \nabla \zeta_\delta \cdot \nabla u)  - ( \rho_+ \chi^+(u) + \rho_- \chi^-(u) ) (y\phi \cdot \nabla \zeta_\delta \right) \, dx \, dy.
\end{align*}
In light of the decay bound~\eqref{double stagnation decay}, both of the terms on the right-hand side vanish in the limit $\delta \searrow 0$.  We conclude, therefore, that $u$ is a monotone variational solution to~\eqref{elliptic PDE}.
\end{proof}

Combining the upper and lower bounds on the decay of the Dirichlet energy, the main result now follows very quickly.

\begin{proof}[Proof of Theorem~\ref{overturning theorem}]
Seeking a contradiction, suppose that overturning~\eqref{overturning} does not occur along $\cme$ and let $u$ be the variational solution to~\eqref{elliptic PDE} furnished by Lemma~\ref{limiting solution lemma}.  Applying Corollary~\ref{lower bound decay lemma} with $\mu = 5/4$, say, we see that there exists $R > 0$ such that $u$ obeys the lower decay bound~\eqref{lower bound decay} in the truncated conical region $\mathcal{D}_{\mu,R}$ with aperture $144\degree$.  But, since $0 \in \Gamma(u)$, we have the Poincar\'e-type inequality
\begin{align*}
	\frac{1}{r^2} \int_{B_r} |\nabla u|^2 \, dx \, dy & \geq \frac{1}{r^2} \int_{\mathcal{D}_{\mu,r}} |\nabla u|^2 \, dx \, dy 
		 \gtrsim \frac{1}{r^4} \int_{\mathcal{D}_{\mu,r}} u^2 \, dx \, dy  
		 \gtrsim r^{\frac{1}{2}}
\end{align*}
for all $0 < r \ll 1$.  The last inequality is obtained by explicitly integrating the right-hand side of~\eqref{lower bound decay} with $\mu = 5/4$. But this  lower bound contradicts the $O(r)$ upper decay bound from~\eqref{double stagnation decay}.  We conclude, therefore, that overturning~\eqref{overturning} must indeed occur.
\end{proof}

\section{Gravity currents and von Kármán's conjecture}
\label{gravity current section}

This section is devoted to the proof of Theorems~\ref{non-boussinesq gravity current theorem} and \ref{boussinesq gravity current theorem}. For simplicity, we will only give the argument for the limit along $\cmd$, as the proof for the limiting behavior along $\cme$ in the Boussinesq setting can be obtained through an easy adaptation.

Following $\cmd$ to its extreme, we know from~\eqref{depression limit alternatives} that if the interface does not overturn then it will approach the upper wall. Formally, one expects this to converge to a gravity current, as discussed in Section~\ref{intro gravity current section}, but the reformulation of the internal wave problem~\eqref{eqn:stream} in Dubreil-Jacotin variables used in~\cite{chen2023global} to construct $\cmd$ does not admit such solutions.  
In Section~\ref{limiting gravity currents section}, we introduce a model free boundary problem that permits gravity currents, and prove rigorously that, in the absence of overturning, there is a limiting domain variational solution to it along $\cmd$. 
In Section~\ref{gc monotonicity formula section}, we derive a Vărvăruca--Weiss monotonicity formula for this gravity current problem, analogous to Theorem~\ref{monotonicity theorem}, and use it to determine the possible blowup limits. Much of this analysis is carried out for the general case; in particular, it is not restricted the specific limiting bore along $\cmd$ but rather holds for general solutions to the gravity current problem. Finally, the proofs of Theorems~\ref{non-boussinesq gravity current theorem} and \ref{boussinesq gravity current theorem} are given in Section~\ref{von Karman section}.

\subsection{Limiting gravity currents}
\label{limiting gravity currents section}

Consider now a gravity current solution to the internal wave problem as depicted in Figure~\ref{gravity current figure}. Since the velocity is uniform in one layer, we can use an essentially single-phase model elliptic PDE. Formally, we look for nonpositive solutions to
\begin{equation}
\label{vK elliptic PDE}
  \left\{ \begin{aligned}
    \Delta u & = 0 & \qquad & \textrm{in } \Omega^-(u)  \cap B_1^- \\
    | \nabla u|^2 & = \jump{\rho}( y + Q) & \qquad & \textrm{on } \partial \Omega^-(u) \cap B_1^- \\
    u  & = 0 & \qquad & \textrm{on } T,
  \end{aligned} \right.
\end{equation}
where $B_1^-$ is the open half-ball in the lower half-plane and $T = \{ y = 0 \} \cap \overline{B_1^-}$ is the ``top'' of the domain, representing the rigid horizontal boundary. The sets $\Omega^+(u)$ and $\Omega^-(u)$ are defined as before, and we continue to assume that each of these phases has distinct constant density $\rho_+$ and $\rho_-$, respectively. Now, however, $\Omega^+(u)$ is a completely stagnant region where $u$ vanishes identically. Abusing terminology, we will call $Q$ the Bernoulli constant, as it will in fact be a rescaling of the Bernoulli constant in our application.

The free boundary of $u$ is again defined to be $\Gamma(u) \colonequals \partial \Omega^+(u) \setminus \partial B_1.$ As we are concerned with the local form of the free boundary at a point where it intersects the upper lid, we work under the assumption that $0 \in \Gamma(u)$. There are two distinct scenarios depending on the Bernoulli constant. If $Q = 0$, then classical solutions would have to satisfy $\nabla u_-(0) = 0$, meaning the contact occurs at a stagnation point. This is the gravity current model proposed by von Kármán~\cite{vonkarman1940engineer} and reconsidered by Benjamin~\cite{benjamin1968gravity}.  Based on their formal analysis, the expectation is that the free boundary must make a $60\degree$ angle with the upper wall. The second situation is where $Q \neq 0$, which again formally means $\nabla u_-(0) \neq 0$. In light of the boundary condition~\eqref{eqn:stream:dynamic}, this is what we will encounter in the limit along the depression bore family $\cmd$ for the non-Boussinesq case where $\rho_1 \neq \rho_2$. Conversely, the dynamic condition for the Boussinesq limit~\eqref{eqn:stream:dynamic-boussinesq} corresponds to a stagnation point at $0$.

As in Section~\ref{overturning section}, we will study \emph{domain variational solutions} of the model elliptic equation~\eqref{vK elliptic PDE}. Specifically, these are defined as follows. 

\begin{definition}
\label{vK variational solution definition}
A function $u \in H_\loc^1(B_1^-) \cap C^0(B_1^- \cup T)$ and constant $Q \leq 0$ are said to comprise a \emph{variational solution} to the free boundary elliptic PDE~\eqref{vK elliptic PDE}  provided $u \leq 0$ on $B_1^-$, $u_- \in C^2(\Omega^-(u))$, $u = 0$ on $T$, and
\begin{equation}
\label{definition vK critical point}
  \begin{aligned}
    0  & =  \int_{B_1^-} \left( |\nabla u|^2 (\nabla \cdot \phi) - 2 D \phi[\nabla u,\nabla u] - ( \rho_+ \chi^+(u) + \rho_- \chi^-(u) ) \nabla \cdot \left( (y+Q) \phi \right) \right) \, dx \, dy,
  \end{aligned}
\end{equation}
for all $\phi \in C_c^1(B_1^- \cup T; \mathbb{R}^2)$ with $\phi \cdot \nu = 0$ on $T$.
\end{definition}

Once again, blowup sequences will be a basic tool in discerning the local structure of the free boundary near the contact point at the origin. Now, however, the correct scaling depends on whether $Q$ is vanishing or non-vanishing. If $Q = 0$, then the Bernoulli condition is invariant under the change of variables $u \mapsto u(r \placeholder)/r^{3/2}$ used in Sections~\ref{model problem section}--\ref{proof of the main theorem section}. When $Q \neq 0$, though, as we approach $0$ along $\Gamma(u)$, we expect the kinetic energy term $|\nabla u_-|^2$ to balance against the constant $\jump{\rho} Q$, which suggests instead using the scaling $u \mapsto u(r \placeholder) / r$. With that in mind, for a fixed variational solution $u$ and sequence $r_m \searrow 0$, let
\begin{equation}
\label{vK blowup sequence}
  u_m \colonequals \left\{ \begin{aligned} 
    \frac{u(r_m \placeholder)}{r_m} & \qquad  \textrm{if } Q \neq 0 \\
    \frac{u(r_m \placeholder)}{r_m^{3/2}} & \qquad  \textrm{if } Q = 0.
  \end{aligned} \right.
\end{equation}

In order to infer the existence of suitable (sub)sequential blowup limit of $\{ u_m\}$, it will often be necessary to impose further assumption on $u$. For that reason, several of our results require that
\begin{equation}
  \label{Lipschitz constant decay}
  [ u ]_{\Lip; \, B_r^- } = \left\{ \begin{aligned}
    O(1) & \qquad \textrm{if } Q \neq 0 \\
    O(r^{1/2}) & \qquad \textrm{if } Q = 0
  \end{aligned} \right.  
    \qquad \textrm{as } r \searrow 0,
\end{equation}
where $[\placeholder]_{\Lip;\, B_r^-}$ denotes the Lipschitz semi-norm on $B_r^-$. Note that decay of the Lipschitz semi-norm is consistent with having a stagnation point at $0$ in the case $Q = 0$. The specific decay rate is naturally motivated by the observation above that the dynamic condition formally requires $|\nabla u|$ to balance $|y|^{1/2}$ along the free boundary. For a single (irrotational) fluid beneath vacuum, it follows for sufficiently regular solutions from the maximum principle, which is one of the main ideas in~\cite{varvaruca2011geometric}. Indeed, in that setting $y$ is a (harmonic) majorant for $|\nabla u|^2$ throughout the fluid domain, so one has the much stronger point-wise decay $|\nabla u|^2 = O(y)$ as $y \nearrow 0$.   For internal waves with two rigid boundaries, however, this is not true in general, as the pressure may be minimized on the upper wall rather than along the free boundary. Nonetheless, as we will show below, the limiting solution along $\cmd$ \emph{will} exhibit the decay~\eqref{Lipschitz constant decay}.

\begin{remark}
  It is useful to observe that assumption~\eqref{Lipschitz constant decay} can be equivalently expressed in terms of blowup sequences. Indeed, $u$ satisfies~\eqref{Lipschitz constant decay} if and only if, for any sequence $r_m \searrow 0$, the corresponding blowup sequence $\{u_m\}$ defined via~\eqref{vK blowup sequence} is uniformly locally Lipschitz in that, for any $K \subset\subset \mathbb{R}_{\leq 0}^2$,
  \begin{equation}
    \label{vK blowup sequence Lipschitz bound}
    \sup_m \| u_m \|_{\Lip(K)} \lesssim 1.
  \end{equation}

  In the case $Q \neq 0$, then the blowup sequence is clearly uniformly Lipschitz, so the statement in~\eqref{vK blowup sequence Lipschitz bound} is only meaningful when $Q = 0$. Arguing as in~\eqref{double stagnation decay}, the bound~\eqref{vK blowup sequence Lipschitz bound} in particular implies that
  \begin{equation}
    \label{vK grad u bound}
    \frac{1}{r_m^2} \int_{B_{r_m}^-} |\nabla u|^2 \, dx \, dy = 
    \left\{
      \begin{aligned}
        O(1) & \qquad  \textrm{if } Q \neq 0 \\
        O(r_m) & \qquad  \textrm{if } Q = 0
      \end{aligned}
    \right.
    \qquad \textrm{as } r_m \searrow 0.
  \end{equation}
\end{remark}

We are particularly interested in monotone variational solutions, which are variational solutions to the gravity current problem as in Definition~\ref{vK variational solution definition} that are also monotone in the sense of Definition~\ref{definition monotone}. Specializing to this class is justified by the following theorem.
 
\begin{theorem}[Limiting gravity current]
  \label{limiting gravity current theorem}
  Suppose that in the limit along $\cmd$ overturning does not occur:
  \begin{equation}
    \label{depression no overturning}
    \limsup_{s \to \infty} \| \partial_x \eta(s) \|_{L^\infty} < \infty.
  \end{equation}
  \begin{enumerate}[label=\rm(\alph*\rm)]
  \item There exists a bounded sequence $\{x_n\} \subset \mathbb{R}$, a sequence $s_n \nearrow \infty$, and a nontrivial domain variational solution $(u,Q)$ to the gravity current problem~\eqref{vK elliptic PDE} such that 
    \begin{equation}
      \label{depression convergence} 
      \begin{aligned}
        \eta(s_n)(\placeholder -x_n) & \xrightarrow{C^\varepsilon} \eta \in \Lip([-1,1]) &\quad &\textrm{for all } \varepsilon \in (0,1), \\
        -\psi(s_n)(\placeholder-x_n, \placeholder) & \xrightarrow{C^\varepsilon} u \in \Lip(B_1^-) &  \quad & \textrm{for all } \varepsilon \in (0,1), \\
        -\nabla  \psi(s_n)(\placeholder-x_n, \placeholder) & \xrightarrow{\textrm{weak-$\ast$ } L^\infty}  \nabla u.  
      \end{aligned}
    \end{equation}
  The Bernoulli constant and densities are given by 
  \[
    Q = \left\{ \begin{aligned}
      -\frac{F^2}{2} & \qquad \textrm{if } \rho_1 \neq \rho_2 \\
      0 & \qquad \textrm{if } \rho_1 = \rho_2,
    \end{aligned} \right.
  \quad 
  \rho_- = \left\{ \begin{aligned}
    \frac{2}{F^2} \rho_1 & \qquad \textrm{if } \rho_1 \neq \rho_2 \\
    8 \rho_1 & \qquad \textrm{if } \rho_1 = \rho_2,
  \end{aligned} \right.
  \quad
  \rho_+ = \left\{ \begin{aligned}
    \frac{2}{F^2} \rho_2 & \qquad \textrm{if } \rho_1 \neq \rho_2 \\
    0 & \qquad \textrm{if } \rho_1 = \rho_2.
  \end{aligned} \right.
\]
  \item The limiting $u$ is monotone in the sense of Definition~\ref{definition monotone} with $\Omega^+(u)$ being the region lying to the right of the graph of $\eta$ and $\Omega^-(u)$ the region lying to the left. The free boundary meets the rigid boundary in that $\eta(0) = 0$. Moreover, $u$ obeys~\eqref{Lipschitz constant decay}
  \end{enumerate}
\end{theorem}
Our proof of Theorem~\ref{limiting gravity current theorem} involves a so-called ``conjugate flow'' analysis of the solutions along $\cmd$ and their translated limits. For this we avoid using a model problem as in Section~\ref{model problem section}, and work on the entire strip $R \colonequals \R \times (-\lambda,1-\lambda)$. Similarly to Definitions~\ref{variational solution definition} and \ref{vK variational solution definition}, we say that $\psi \in H^1_\mathrm{loc}(\overline\fluidD) \cap C^0(\overline\fluidD)$ is a domain variational solution of \eqref{eqn:stream} if $\psi_i \in C^2(\Omega^i(\psi))$ and, for all $\phi \in C^1_\mathrm c(\overline R;\R^2)$ with $\phi \cdot \nu = 0$ on $\partial R$ we have
\begin{align*}
  0  & =  \int_R \left( |\nabla \psi|^2 (\nabla \cdot \phi) - 2 D \phi[\nabla \psi,\nabla \psi] - ( \rho_1 \chi^1(\psi) + \rho_2 \chi^2(\psi) ) \nabla \cdot (2F^{-2} y  - 1)\phi \right) \, dx \, dy,
\end{align*}
where here $\Omega^1,\Omega^2$ and $\chi^1,\chi^2$ are the appropriate analogues of $\Omega^\pm$ and $\chi^\pm$. We do not require the asymptotic condition \eqref{eqn:stream:asym}, but do require kinematic boundary conditions \eqref{psi value on top} on $\partial R$.

\begin{lemma}[Flow force]\label{flow force lemma}
  Let $\psi$ be a domain variational solution of \eqref{eqn:stream} in the above sense. Then there exists a constant $S$, called the \emph{flow force}, such that, for almost all $x \in \R$,
  \begin{align*}
    S = \frac 12 \int_0^1 \big(\psi_y^2 - \psi_x^2 - (\rho_1 \chi^1(\psi) + \rho_2 \chi^2(\psi)) (2F^{-2}y - 1) \big)\, dy.
  \end{align*}
  \begin{proof}
    The proof is very similar to the proof of Lemma~\ref{variational solution identity}, but with the family of balls $B_r$ replaced by rectangles $(a,b) \by (-\lambda,1-\lambda)$ and with the special choice of vector field $\phi = (1,0)$.
  \end{proof}
\end{lemma}

The flow force allows us to classify the limiting states of monotone solutions. First we consider $\lambda \in (0,1)$, which will hold for the smooth solutions on $\cmd$.
\begin{lemma}[Conjugate flows for $\lambda \ne 0,1$]\label{conjugate flow lemma}
  In the setting of Lemma~\ref{flow force lemma}, suppose further that 
  \begin{enumerate}[label=\rm(\roman*)]
  \item $\psi$ is monotone in the sense of Definition~\ref{definition monotone};
  \item the free boundary $\Gamma(\psi)$ is the graph of a monotone Lipschitz function $\eta$ of $x$; 
  \item $\lambda \in (0,1)$; and
  \item the value of the flow force constant $S$ in Lemma~\ref{flow force lemma} agrees with the value $S_0(\lambda)$ for the trivial solution with $\eta \equiv 0$.
  \end{enumerate}
  Then $F^2$ is given explicitly by \eqref{definition front Froude number}, and the upstream and downstream thicknesses 
  \begin{align*}
    H_\mathrm u \colonequals \lim_{x \to -\infty} \eta(x) - \lambda,
    \qquad 
    H_\mathrm d \colonequals \lim_{x \to +\infty} \eta(x) - \lambda
  \end{align*}
  of the bottom layer can only take the values
  \begin{align*}
    H_\mathrm u, H_\mathrm d \in \left\{ \lambda, \frac{\sqrt{\rho_1}}{\sqrt{\rho_1}+\sqrt{\rho_2}}, 0, 1 \right\}.
  \end{align*}
\end{lemma}
% For the smooth solutions along $\cmd$, we in fact know that $H_\mathrm u$ takes the first of these values while $H_\mathrm d$ takes the second.
\begin{proof}
  For smooth solutions this is well known; see, e.g.~\cite[Appendix~A]{Laget1997interfacial} or \cite[Lemma~5.5]{chen2023global}. The modifications for domain variational solutions satisfying the above hypotheses are relatively straightforward. The arguments for $H_\mathrm u,H_\mathrm d$ are identical, and so we only consider $H_\mathrm d$. Monotonicity and a translation argument imply that
  \begin{align*}
    \Psi^\mathrm d(y)
    \colonequals
    \lim_{x \to +\infty} \psi(x,y) 
  \end{align*}
  exists and is itself a domain variational solution of \eqref{eqn:stream}. This combined with \eqref{psi value on top} yields
  \begin{align*}
    \Psi^\mathrm d = 
    \begin{cases}
      -\dfrac \lambda{H_\mathrm d} \sqrt{\rho_1} (y+\lambda-H_\mathrm d) & y \in [-\lambda,- \lambda + H_\mathrm d], \\[2.5ex]
      -\dfrac {1-\lambda}{1-H_\mathrm d} \sqrt{\rho_2} (y+\lambda-H_\mathrm d) & y \in [-\lambda + H_\mathrm d, 1-\lambda].
    \end{cases}
  \end{align*}
  If $H_\mathrm d \ne 0,1$, then the appropriate weak form of \eqref{eqn:stream:dynamic} also gives
  \begin{align}
    \label{eqn:conj:Q}
    (\Psi^\mathrm d_{2,y})^2
    - (\Psi^\mathrm d_{1,y})^2
    + \frac{2(\rho_2-\rho_1)}{F^2}(H_\mathrm d - \lambda) = \rho_2 - \rho_1
    \qquad \text{ at } y = -\lambda + H_\mathrm d.
  \end{align}
  Arguing as in the proof of Lemma~\ref{flow force lemma}, one finds that $\Psi^\mathrm d$ has flow force constant $S=S_0(\lambda)$. This together with \eqref{eqn:conj:Q} forms a system of two algebraic equations for $H_\mathrm d$ and $F^2$, called ``conjugate flow equations''. The result now follows from an analysis of these equations exactly as in \cite{Laget1997interfacial} or \cite{chen2023global}.
\end{proof}

Our limiting gravity currents will have $\lambda = 1$, and in this extreme case one can prove a stronger result.

\begin{lemma}[Conjugate flows for $\lambda = 1$]\label{conjugate flow lemma stagnant}
  In the setting of Lemma~\ref{conjugate flow lemma}, suppose instead that $\lambda = 1$, $F^2$ is given by \eqref{definition front Froude number}, and 
  \begin{align*}
    H_\mathrm u = \lim_{x \to -\infty} \eta(x) + 1 > \frac{\sqrt{\rho_1}}{\sqrt{\rho_1}+\sqrt{\rho_2}}.
  \end{align*}
  Then there exists $x_0 \in \R$ so that $\eta \equiv 0$ on $(-\infty,x_0]$.
  \begin{proof}
    Since $\lambda = 1$, the boundary conditions \eqref{psi value on top} force $\psi_2 \equiv 0$. Arguing as in the proof of Lemma~\ref{conjugate flow lemma}, we find that
    \begin{align*}
      \lim_{x \to +\infty} \psi(x,y) 
      \equalscolon
      \Psi^\mathrm u(y)
      = 
      \begin{cases}
        -\dfrac 1{H_\mathrm u} \sqrt{\rho_1} (y+1-H_\mathrm u) & y \in [-1,- 1 + H_\mathrm u], \\[1.5ex]
        0 & y \in [-1 + H_\mathrm u, 0].
      \end{cases}
    \end{align*}
    has flow force $S = S_0(1)$. This leads to a single algebraic equation for $H_\mathrm u$. Since by assumption $H_\mathrm u > \sqrt{\rho_1}/(\sqrt{\rho_1}+\sqrt{\rho_2}) > 0$ and $F^2$ is given by \eqref{definition front Froude number}, one can check that the only solution of this equation is $H_\mathrm u = 1$.

    Suppose for contradiction that we also have $\eta(x) < 0$ for all $x \in \R$. By assumption $\eta(x) > -1$ for $x$ sufficiently large and negative, and so arguing as in the proof of Lemma~\ref{conjugate flow lemma} we obtain a weak form of \eqref{eqn:stream:dynamic}, namely \eqref{eqn:conj:Q} with $H_\mathrm d,\Psi^\mathrm d$ replaced by $H_\mathrm u,\Psi^\mathrm u$. Subsisting $H_\mathrm u = 1$ into this equation yields $\rho_2 = 0$, which is the desired contradiction.
  \end{proof}
\end{lemma}

\begin{proof}[Proof of Theorem~\ref{limiting gravity current theorem}]
  The argument will in many ways mirror that of Lemma~\ref{limiting solution lemma}. The main difference is that, because the interface does not remain uniformly separated from the upper rigid boundary, the Alt--Caffarelli--Friedman monotonicity formula cannot be applied directly to control the normal derivative of $\psi(s_n)$ along the free boundary. To circumvent this issue, we imagine extending the fluid domain as a uniform laminar flow. Physically, this amounts to replacing the rigid boundaries with vortex sheets.  We must ensure that the vortex sheet strength is nonnegative in order to apply Theorem~\ref{acf theorem}, as this implies the extended stream function is superharmonic.

  With that in mind, we first observe that the qualitative properties of the solutions along $\cmd$ given in~\eqref{monotonicity Euler variables} ensure that $\partial_x \psi(s) < 0$ in $\fluidD(s) \cup \fluidS(s)$ for all $s > 0$, and hence by the Hopf lemma,
  \[
    -\partial_x \partial_y \psi(s) < 0 \qquad \textrm{on } \fluidT(s). 
  \] 
  Here, we are writing $\fluidT(s) \colonequals \{ y = 1-\lambda(s)\}$ for the lid corresponding to $\fluidD(s)$. In particular, the above inequality implies that along $\fluidT(s)$, $-\partial_y \psi(s)$ is maximized in the upstream limit $x \to -\infty$. The same reasoning applied to the trace of $-\partial_y \psi(s)$ on the bed $\fluidB(s) \colonequals \{ y =-\lambda(s)\}$ reveals that $-\partial_y \psi(s)$ is minimized in the upstream limit. The explicit formula for the upstream state~\eqref{eqn:stream:asym} therefore implies the tangential velocity bounds
  \begin{equation}
    \label{lid horizontal velocity bound}
    -\partial_y \psi(s) < \sqrt{\rho_2} \quad \textrm{on }  \fluidT(s) \qquad -\partial_y \psi(s) > \sqrt{\rho_1} \quad \textrm{on }  \fluidB(s).
  \end{equation}

  For each $s > 0$, consider the extended pseudo stream function 
  \[
    \underline \psi(s) \in \Lip_\bdd(\overline{\mathscr{R}}) \cap C_\bdd^{2+\alpha}(\overline{\fluidD_1(s)}) \cap C_\bdd^{2+\alpha}(\overline{\fluidD_2(s)})
  \]
  defined on the ($s$-independent) strip $\extD \colonequals \mathbb{R} \times (-1,1)$ by
  \[
    \underline \psi(s) (x,y) \colonequals 
    \left\{ \begin{aligned}
      \psi(s) & \qquad \textrm{on } \overline{\fluidD(s)} \\
      -\sqrt{\rho_2} y & \qquad \textrm{for } y > 1-\lambda(s) \\
      -\sqrt{\rho_1} y & \qquad \textrm{for } y < -\lambda(s).
    \end{aligned} \right.	
  \]
  Denote by $\underline\psi_1(s)$ the restriction of $\underline\psi(s)$ to the region below $\fluidS(s)$ and $\underline\psi_2(s)$ for the restriction the region above $\fluidS(s)$. From~\eqref{lid horizontal velocity bound}, it then follows that
  \begin{align*}
    -\Delta \underline \psi_2(s) & = \left( \sqrt{\rho_2} + \partial_y \psi(s) \right) d\mathcal{H}^1 \resmes \fluidT(s) \geq 0, \\
      -\Delta \underline \psi_1(s) & = - \left(  \sqrt{\rho_1} + \partial_y \psi(s)  \right) d\mathcal{H}^1 \resmes \fluidB(s) \geq 0,
  \end{align*}
  where the inequalities above are interpreted in the sense of distributions on $\mathscr{R}$. Note that the upstream velocity field corresponding to $\underline \psi(s)$ is the same uniform laminar flow as before, except that it has been extended from $y = -1$ to $y = 1$.

  Now, supposing that overturning does not occur, we have by~\eqref{depression limit alternatives} that $\lambda(s) \to 1$ as $s \to \infty$. On the other hand, from \cite[Lemma 5.9]{chen2023global}, the gradient $\nabla \underline \psi(s)$ is bounded uniformly in $L^2(K)$ for any $K \subset\subset \overline{\extD}$. As $\underline\psi_1(s)$ and $\underline\psi_2(s)$ are superharmonic, vanish on $\fluidS(s)$, and single-signed, the same argument as in \cite[Theorem 5.7]{chen2023global} confirms that $\| \underline \psi(s) \|_{\Lip(\extD)}$ is uniformly bounded. Again, this is possible because Theorem~\ref{acf theorem} controls the Lipchitz norm in an $O(1)$ neighborhood of $\fluidS(s)$.

  Fix a sequence $s_n \nearrow 0$. By (strict) monotonicity, there exists a sequence $\{ x_n \} \subset \mathbb{R}$ such that
   \begin{equation}
    \label{depression half height}
    \eta(s_n)(x_n) =  \frac{1}{2} \lim_{x \to \infty} \eta(s_n)(x)  =   \frac{1}{2} \left(  \frac{\sqrt{\rho_1}}{\sqrt{\rho_1}+\sqrt{\rho_2}} - \lambda(s_n) \right).
  \end{equation}
%  \begin{equation}
%    \label{depression stagnation point limit}
%    \lim_{n \to \infty} \partial_y \underline\psi_2(s_n)(x_n, \eta(s_n)(x_n)) = 0.
%  \end{equation}
The previous paragraph furnishes uniform Lipschitz bounds on the translated sequences
  \[
    \sup_{n} \left( \|  \underline \psi(s_n)(\placeholder-x_n, \placeholder) \|_{\Lip(\extD)} + \| \eta(s_n)(\placeholder-x_n)\|_{\Lip(\mathbb{R})} \right) < \infty.
  \]
  It follows that there exists a subsequential limits  $\eta(s_n) \to \eta$ and $\psi(s_n) \to \psi$ that are locally uniform, for $\eta \in \Lip_\bdd(\mathbb{R})$ and $\psi \in \Lip_\bdd(\overline{\mathscr{R}})$. 
  
  Following the same strategy as in Lemma~\ref{limiting solution lemma}, we can prove that $-\psi$ is domain variational solution to the gravity current problem~\eqref{vK elliptic PDE} on $\mathbb{R} \times (-1,0)$, that is monotone in the sense of Definition~\ref{definition monotone}, and such each blowup sequence~\eqref{blowup sequence} has the uniform Lipschitz bound~\eqref{vK blowup sequence Lipschitz bound}. The values of $Q$, $\rho_+$, and $\rho_-$ asserted above are easily deduced from the dynamic condition in the internal wave problem~\eqref{eqn:stream:dynamic} and \eqref{eqn:stream:dynamic-boussinesq}. In view of~\eqref{depression half height}, we know that 
  \[
  	\lim_{x \to -\infty} \eta(x) + 1 > \eta(0) + 1 > \frac{\sqrt{\rho_1}}{\sqrt{\rho_1}+\sqrt{\rho_2}}.
  \]  
  Invoking Lemma~\ref{conjugate flow lemma stagnant}, we conclude that $\eta$ vanishes on some semi-infinite interval $(-\infty, x_0]$.  Since $\eta(0) < 0$, we know that $x_0 < 0$, and as $\eta$ is non-increasing, we can take $\eta(x) < 0$ for $x > x_0$. Performing a final horizontal translation so that $x_0$ is at the origin completes the proof.
\end{proof}

\subsection{A Vărvăruca--Weiss monotonicity formula for gravity currents}
\label{gc monotonicity formula section}

For both the stagnation case ($Q = 0$) and non-stagnation case ($Q \neq 0$), we will use a monotonicity formula in the spirit of~\cite{varvaruca2011geometric} to infer the existence of a blowup limit, which will ultimately allow us to understand the geometry of the interface near the contact point at $0$. Convergence of the blowup sequence requires the Lipschitz bound~\eqref{Lipschitz constant decay}, however, we do not need the must stronger assumption that $u$ is monotone in the sense of Definition~\ref{definition monotone}.

\begin{theorem}[Gravity current monotonicity formula]
\label{gc monotonicity formula theorem}
Suppose that $u$ is a variational solution to the free boundary elliptic problem~\eqref{vK elliptic PDE} and let
\[
\begin{aligned}
	\mathcal{I}_\gc(u; \, r) & \colonequals 
		\left\{ \begin{aligned}
			\frac{1}{r^2} \int_{B_r^-} \left( |\nabla u|^2 - Q \left( \rho_+ \chi^+(u) + \rho_- \chi^-(u) \right) \right) \, dx \, dy & \qquad \textrm{if } Q \neq 0 \\
			\frac{1}{r^3} \int_{B_r^-} \left( |\nabla u|^2 - y \left( \rho_+ \chi^+(u) + \rho_- \chi^-(u) \right) \right) \, dx \, dy & \qquad \textrm{if } Q = 0
		\end{aligned} \right. \\
	\mathcal{J}_\gc(u; \, r) & \colonequals
		\left\{ \begin{aligned}
			 \frac{1}{r^3} \int_{\partial B_r^-} u^2 \, d\mathcal{H}^1 & \qquad \textrm{if } Q \neq 0 \\
			 \frac{1}{r^4} \int_{\partial B_r^-} u^2 \, d\mathcal{H}^1 & \qquad \textrm{if } Q = 0,
		\end{aligned} \right.
\end{aligned}
\]
and
\[
\label{gc M definition}
	\mathcal{M}_\gc = \mathcal{M}_\gc(u; \, r) \colonequals
		\left\{ \begin{aligned}
			\mathcal{I}_\gc - \mathcal{J}_\gc & \qquad \textrm{if } Q \neq 0\\
			\mathcal{I}_\gc - \frac{3}{2} \mathcal{J}_\gc & \qquad \textrm{if } Q = 0.
		\end{aligned} \right.
\]
Then, $\mathcal{M}_\gc(u; \, \placeholder)$ is almost everywhere differentiable with derivative given by
\[
	\mathcal{M}_\gc^\prime(r) = 
		\left\{ \begin{aligned}
			\frac{2}{r^2} \int_{\partial B_r^- \setminus T} \left( \nabla u \cdot \nu - \frac{u}{r} \right)^2 \, d\mathcal{H}^1 - \frac{3}{r^3} \int_{B_r^-} y \left(\rho_+ \chi^+(u) + \rho_- \chi^-(u) \right) \, dx \, dy & \qquad \textrm{if } Q \neq 0 \\ 
			\frac{2}{r^{3}} \int_{\partial B_r^- \setminus T} \left( \nabla u \cdot \nu - \frac{3}{2} \frac{u}{r} \right)^2 \, d\mathcal{H}^1 & \qquad \textrm{if } Q = 0.
	\end{aligned} \right.
\]
\end{theorem}
\begin{proof}
  The formula for the stagnation case ($Q = 0$) follows from an easy adaptation of the proof of Theorem~\ref{monotonicity theorem}, so we only present the proof for the non-stagnation case ($Q \neq 0$). An elementary computation shows that in this setting
  \begin{equation}
    \label{J prime formula Q neq 0}
    \mathcal{J}_\gc^\prime(u; \, r) = \frac{2}{r^3} \int_{\partial B_r^- \setminus T} u  \nabla u \cdot \nu \, d\mathcal{H}^1  -\frac{2}{r^4} \int_{\partial B_r^- \setminus T} u^2 \, d \mathcal{H}^1,
  \end{equation}
  for almost every $0 < r \ll 1$. On the other hand, 
  \begin{align*}
    \mathcal{I}_\gc^\prime(u; \, r) & = -\frac{2}{r^3} \int_{B_r^-} \left( |\nabla u|^2 - Q \left( \rho_+ \chi^+ + \rho_- \chi^- \right) \right) \, dx \, dy \\
    & \qquad + \frac{1}{r^2} \int_{\partial B_r^- \setminus T} \left( |\nabla u|^2 - Q \left( \rho_+ \chi^+ + \rho_- \chi^- \right) \right) \, d\mathcal{H}^1,
  \end{align*}
  where we are abbreviating $\chi^\pm \colonequals \chi^\pm(u)$. Using the test field $\phi = (x,y)/r$, which note satisfies $\phi \cdot \nu = 0$ on $T$, the analogue of Lemma~\ref{variational solution identity} for the gravity current problem yields the identity
  \begin{align*}
    & -\int_{B_r^-}  \left( 3y + 2Q \right) \left( \rho_+ \chi^+ + \rho_- \chi^- \right) \, dx \, dy \\
    & \qquad \qquad = r\int_{\partial B_r^- \setminus T} \left( |\nabla u|^2 -2 (\nabla u \cdot \nu)^2 - (\rho_+ \chi^+ + \rho_- \chi^-) ( y + Q) \right) \, d\mathcal{H}^1.
  \end{align*}
  Using this to eliminate the bulk integral of $Q (\rho_+ \chi^+ + \rho_- \chi^-)$ in the formula for $\mathcal{I}_\gc^\prime$ above, we find that
  \begin{align*}
    \mathcal{I}_\gc^\prime(u; \, r) & = - \frac{1}{r^3} \int_{B_r^-} \left( 2 |\nabla u|^2 + 3y \left( \rho_+ \chi^+ + \rho_- \chi^- \right)  \right)\, dx \, dy + \frac{2}{r^2} \int_{\partial B_r^- \setminus T} (\nabla u \cdot \nu)^2 \, d\mathcal{H}^1 \\
    & =  -\frac{3}{r^3} \int_{B_r^-}  y \left( \rho_+ \chi^+ + \rho_- \chi^- \right)  \, dx \, dy + \frac{2}{r^2} \int_{\partial B_r^- \setminus T} \left(  (\nabla u \cdot \nu)^2 - \frac{u}{r} \nabla u \cdot \nu \right) \, d\mathcal{H}^1,
  \end{align*}
  where the last line follows from the integration by parts identity~\eqref{integration by parts identity}. Combining this with the calculation of $\mathcal{J}_\gc^\prime$ in~\eqref{J prime formula Q neq 0} leads finally to
  \begin{align*}
    \mathcal{M}_\gc^\prime(r) & = \frac{2}{r^2} \int_{\partial B_r^- \setminus T} \left( (\nabla u \cdot \nu)^2 - 2\frac{u}{r} \nabla u \cdot \nu + \frac{u^2}{r^2} \right) \, d\mathcal{H}^1 -\frac{3}{r^3} \int_{B_r^-} y \left( \rho_+ \chi^+ + \rho_- \chi^- \right) \, dx \,dy,
  \end{align*}
  which becomes the claimed formula upon factoring the integrand in the first integral.
\end{proof}

With the monotonicity formula in hand, we can proceed in roughly the same fashion as in Section~\ref{monotonicity formula section} to discern the possible blowup limits. First consider the case where the Bernoulli constant is nonzero. 

\begin{lemma}[No stagnation blowup limit] \label{vK blowup lemma Q not 0}
Suppose that $u$ is a variational solution to the free boundary elliptic problem~\eqref{vK elliptic PDE} with $Q \neq 0$ and $u \in \Lip_\loc(B_1^- \cup T)$.  
\begin{enumerate}[label=\rm(\alph*\rm)]
	\item \textup{(Density)} The right-hand side limit $\mathcal{M}_\gc(u; \, 0+)$ exists and is finite. 
	\item \textup{(Blowup $u$)} \label{vK blowup up part Q not 0} For a sequence $r_m \searrow 0$, consider the blowup sequence $\{u_m\}$ defined by \eqref{vK blowup sequence}. Perhaps passing to a subsequence, 
	\[
		u_m \xrightarrow{H_\loc^1 \cap C_\loc^\varepsilon} \alpha y \qquad  \textrm{for all } \varepsilon \in (0,1),
	\]
	for some $\alpha \geq 0$.
	\item \textup{(Blowup domain)} \label{vK blowup domain part Q not 0} 
	Suppose further that in a neighborhood of $0$, the free boundary $\Gamma(u)$ is a union of finitely many Lipschitz curves. Set $\chi_m^\pm \colonequals \chi^\pm(u_m)$. Then there exists $\chi_0^+, \chi_0^- \in BV_\loc(\mathbb{R}_{\leq 0}^2)$ such that $\chi_m^\pm \to \chi_0^\pm$ in $L_\loc^1(\mathbb{R}_{\leq 0}^2)$. Moreover, $\chi_0^+ + \chi_0^- = 1$ on $\mathbb{R}_{\leq 0}^2$ and either $\chi_0^+$ or $\chi_0^-$ vanishes identically.  
\end{enumerate}
\end{lemma}
\begin{proof}
Since $\mathcal{M}^\prime_\gc$ is clearly positive and bounded, the existence of the right-hand limit $\mathcal{M}_\gc(u; \, 0+)$ is immediate.
The argument for part~\ref{vK blowup up part Q not 0} closely follows that of Lemma~\ref{blowup limits lemma}, with the essential point being that the additional bulk integral term in the formula for $\mathcal{M}_\gc^\prime$ does not contribute in the blowup limit. Because the corresponding blowup sequence $\{ u_m \}$ is equi-Lipschitz, passing to a subsequence, we may infer the existence of a locally uniform limit  $u_0 \in H_\loc^1(\mathbb{R}_{\leq 0}^2) \cap \Lip_\loc(\mathbb{R}_{\leq 0}^2)$ that is nonpositive and harmonic on $\Omega^-(u_0)$.

We claim, moreover, that $u_0$ is $1$-homogeneous. Fix $0 < \sigma_1 < \sigma_2$, and note that by Theorem~\ref{gc monotonicity formula theorem},
\begin{align*}
	\mathcal{M}_\gc(u; \, r_m \sigma_2) - \mathcal{M}_\gc(u; \, r_m \sigma_1) & = \int_{\sigma_1}^{\sigma_2} \frac{2}{\sigma^2} \int_{\partial B_\sigma^- \setminus T} \left( \nabla u_m \cdot \nu - \frac{u_m}{\sigma} \right)^2 \, d\mathcal{H}^1 \, d\sigma \\
		& \qquad - r_m \int_{\sigma_1}^{\sigma_2} \frac{3}{\sigma^3} \int_{B_\sigma^-} y \left( \rho_+ \chi_m^+ + \rho_- \chi_m^- \right) \, dx \,dy \, d\sigma.
\end{align*}
Sending $r_m \searrow 0$, the $H^1$ and uniform convergence of $u_m \to u_0$ ensure that
\[
	0 = \int_{\sigma_1}^{\sigma_2} \frac{1}{\sigma^2} \int_{\partial B_\sigma^- \setminus T} \left( \nabla u_m \cdot \nu - \frac{u_m}{\sigma} \right)^2 \, d\mathcal{H}^1 \, d\sigma,
\]
which implies the claimed $1$-homogeneity of $u_0$. But, now the fact that $u_0$ is harmonic on $\Omega^-(u_0)$ $1$-homogeneous, and vanishes on $\{ y = 0\}$ forces it to be a nonnegative multiple of $y$.  This completes the proof of part~\ref{vK blowup up part Q not 0}.

Next consider part~\ref{vK blowup domain part Q not 0}. The sequences $\{\chi_m^+\}$ and $\{ \chi_m^-\}$  are uniformly $BV_\loc$ as a consequence of the Lipschitz continuity of $\Gamma(u)$ near $0$. The existence of the blowup limits $\chi_0^+$ and $\chi_0^-$ then follows from the same argument as in Lemmas~ \ref{limit of chi lemma}. Moreover, from the definition of a variational solution \eqref{definition vK critical point}, the scaling in~\eqref{vK blowup sequence}, and the fact that $u_m \to \alpha y$ in $H_\loc^1 \cap C_\loc^\varepsilon$ for each $\varepsilon \in (0,1)$, we find that
\[
	0 = \int_{\{ y < 0 \}} \left(\alpha^2 \nabla \cdot \phi - 2 \alpha^2 \partial_y \phi_2 + \left(  \rho_+ \chi_0^+ + \rho_- \chi_0^- \right) \nabla\cdot (Q \phi) \right) \, dx \,dy
\]
for all vector fields $\phi \in C_c^1(\mathbb{R}_{\leq 0} ^2; \mathbb{R}^2)$ with $\phi \cdot \nu = 0$ on $\{ y = 0 \}$. Note that the first two terms in the integrand vanish, and so this implies that $\rho_+ \chi_0^+ + \rho_- \chi_0^-$ is constant. 
\end{proof}

\begin{remark}
\label{density Q not 0 remark}
In fact, the value of $\mathcal{M}_\gc(u; \, 0+)$ in Lemma~\ref{vK blowup lemma Q not 0} uniquely determines the blowup domain. Fix any sequence $r_m \searrow 0$ so that $u_m \to \alpha y$ locally uniformly and $\chi_m^\pm \to \chi_0^\pm$ in $L^1_\loc$. We compute directly that
\begin{align*}
	\mathcal{M}_\gc(u; \, r_m) = \mathcal{M}_\gc(u_m; \, 1) & = \int_{B_1^-} \left( |\nabla u_m|^2 - Q \left( \rho_+ \chi_m^+ + \rho_- \chi_m^- \right) \right) \, dx \, dy - \int_{\partial B_1^-} u_m^2 \, d\mathcal{H}^1 \\ 
		& \longrightarrow   \int_{B_1^-} \left( \alpha^2 - Q \left( \rho_+ \chi_0^+ + \rho_- \chi_0^- \right) \right)  \, dx \, dy - \int_{\partial B_1^-} \alpha^2 y^2 \, d\mathcal{H}^1 \\
		& = - Q \int_{B_1^-} \left( \rho_+ \chi_0^+ + \rho_- \chi_0^- \right)   \, dx \, dy
\end{align*} 
as $r_m \searrow 0$. In light of Lemma~\ref{vK blowup lemma Q not 0}~\ref{vK blowup domain part Q not 0}, it follows that $\mathcal{M}_\gc(u; \, 0+) = -Q \rho_+ \pi/2$ or $\mathcal{M}_\gc(u; \, 0+) = - Q \rho_- \pi/2$, depending precisely on which phase has full density at the origin. Since these are distinct, the blowup domain is uniquely determined by the value of $\mathcal{M}_\gc(u; \, 0+)$.
\end{remark}

The situation in the stagnation point setting where $Q = 0$ is similar to what we saw in Section~\ref{lower bound section}, as without assuming monotonicity, there are several potential nontrivial blowup limits.

\begin{lemma}[Stagnation blowup limit] \label{vK monotonicity lemma}
Suppose that $u$ is a variational solution to the free boundary elliptic problem~\eqref{vK elliptic PDE} with $Q = 0$ and satisfying \eqref{Lipschitz constant decay}. Assume that, in a neighborhood of $0$, the free boundary $\Gamma(u)$ is a union of finitely many Lipschitz curves.
\begin{enumerate}[label=\rm(\alph*\rm)]
	\item \textup{(Density)} The right-hand side limit $\mathcal{M}_\gc(u; \, 0+)$ exists and is finite.
	\item \textup{(Blowup $u$)} For a sequence $r_m \searrow 0$, consider the blowup sequence $\{u_m\}$ defined by \eqref{vK blowup sequence}. Perhaps passing to a subsequence, 
	\[
		u_m \xrightarrow{H_\loc^1 \cap C_\loc^\varepsilon} u_0 \qquad \textrm{for all } \varepsilon \in (0,1),
	\]
	where $u_0 \in H_\loc^1(\mathbb{R}_{\leq 0}^2) \cap \Lip_\loc(\mathbb{R}_{\leq 0}^2)$ is a nonpositive $3/2$-homogeneous function that is harmonic on $\Omega^-(u_0)$. Either
  \begin{equation}
    \label{vK blowup u}
    u_0 = 0, \quad u_0 = \Stokesu, \quad u_0 = \Stokesu \circ A_{\tfrac{\pi}{6}}, \quad \textrm{or} \quad u_0 = \Stokesu \circ A_{\tfrac{\pi}{6}}^{-1}.
  \end{equation}
	
	\item \textup{(Blowup domain)} \label{vK blowup domain part} 
	For $u_m$ and $u_0$ be given as in the previous part, define $\chi_m^\pm \colonequals \chi^\pm(u_m)$. Then there exists $\chi_0^+, \chi_0^- \in BV_\loc(\mathbb{R}_{\leq 0}^2)$ such that $\chi_m^\pm \to \chi_0^\pm$ in $L_\loc^1(\mathbb{R}_{\leq 0}^2)$. Moreover,
	\[
		\chi_0^- = 1,~\chi_0^+= 0 \quad \textrm{ on } \Omega^-(u_0), \qquad \chi_0^+ + \chi_0^- = 1,~\chi_0^+ \chi_0^- = 0 \quad \textrm{ on } \mathbb{R}_{\leq 0}^2,
	\]
	and
	\[
		\rho_+ \chi_0^+ + \rho_- \chi_0^- \quad \textrm{is constant on each connected component of } \{ u_0 = 0\}.
	\]
\end{enumerate}
\end{lemma}
\begin{proof}
The existence of the right-sided limit $\mathcal{M}_\gc(u; \, 0+)$ follows from the monotonicity of $\mathcal{M}_\gc(u; \, \placeholder)$ established in Theorem~\ref{gc monotonicity formula theorem} and the boundedness of $\mathcal{M}_\gc$ due to~\eqref{vK blowup sequence Lipschitz bound}. By hypothesis and~\eqref{vK grad u bound}, the blowup sequence $\{ u_m \}$ is uniformly bounded in $H_\loc^1$ and $\Lip_\loc$. The corresponding characteristic functions $\chi_m^\pm$ are uniformly $BV_\loc$ as a consequence of the Lipschitz continuity of $\Gamma(u)$ near $0$. The existence of the blowup limit $u_0$ and $\chi_0^\pm$ then follows from the same argument as in Lemmas~\ref{blowup limits lemma} and \ref{limit of chi lemma}.  
\end{proof}

The next lemma refines the characterization of the blowup limit in the stagnation case by considering the value of the density $\mathcal{M}_\gc(u; \, 0+)$. Importantly, it shows that the blowup is uniquely determined by the density.

\begin{figure}
	\centering
	\includegraphics[scale=1]{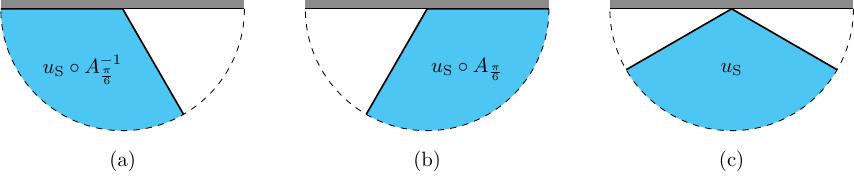}
	\caption{Possible nonzero blowup $u_0$ for the gravity current problem with $Q = 0$. The region $\{ u_0 < 0\}$ is shaded in blue, while the region $\inter\{u_0 = 0\}$ is white. Both (a) and (b) correspond to alternative~\ref{vK corner alternative}, while (c) represents alternative~\ref{vK Stokes corner alternative}.}
\end{figure}
\begin{lemma}[Density with stagnation point]
\label{vK density lemma}
Let $u$ be a variational solution to~\eqref{vK elliptic PDE} satisfying the hypotheses of Lemma~\ref{vK monotonicity lemma}. There are four mutually exclusive alternatives for any blowup limit:
\begin{enumerate}[label=\rm(\rm A\arabic*\rm)]
	 	\item \label{vK all heavier alternative} heavier fluid has full density,
		\[
			\mathcal{M}_\gc(u; \, 0+) = \frac{2 \rho_-}{3}, \quad u_0 = 0, \quad \chi_0^+ \equiv 0, \quad \chi_0^- \equiv 1;
		\]
		\item \label{vK all lighter alternative} lighter fluid has full density,
		\[
			\mathcal{M}_\gc(u; \, 0+) = \frac{2\rho_+}{3}, \quad u_0 = 0 \quad \chi_0^+ \equiv 1, \quad \chi_0^- \equiv 0;
		\]
		\item \label{vK Stokes corner alternative} Stokes corner flow,
		\[
			\mathcal{M}_\gc(u; \, 0+) = \left( \frac{2}{3} - \frac{1}{\sqrt{3}} \right) \rho_+ + \frac{\rho_-}{\sqrt{3}}, \quad u_0 = \Stokesu,\quad \chi_0^\pm = \chi^\pm(\Stokesu); \quad \textrm{or}
		\]
		\item \label{vK corner alternative} Stokes corner flow rotated $30\degree$ clockwise or counterclockwise,
		\[
			\mathcal{M}_\gc(u; \, 0+) =  \frac{\rho_+}{6} +\frac{\rho_-}{2}, \quad u_0 = \Stokesu \circ A_{\pi/6} \textrm{ or }\Stokesu \circ A_{\pi/6}^{-1}, \quad \chi_0^\pm = \chi^\pm(u_0).  \]
	\end{enumerate}
	Moreover, if in a neighborhood of $0$, $\Omega^+(u)$ is connected, then alternative~\ref{vK Stokes corner alternative} cannot occur.
\end{lemma}
\begin{proof}
  Let $u$ be given as above. To evaluate $\mathcal{M}_\gc(u;\, 0+)$, take any sequence $r_m \searrow 0$ along which the corresponding blowup sequence $\{ u_m\}$ defined via~\eqref{blowup sequence} converges, and consider
  \begin{align*}
    \mathcal{M}_\gc(u; \, r_m) & = \frac{1}{r_m^3} \int_{B_{r_m}^-} \left( |\nabla u|^2 -  \left( \rho_+ \chi^+(u) + \rho_- \chi^-(u) \right) y \right)  \, dx \, dy - \frac{3}{2} \frac{1}{r_m^4} \int_{\partial B_{r_m}^-} u^2 \, d\mathcal{H}^1 \\
    & = \frac{1}{r_m} \int_{B_1^-} \left( |\nabla u_m|^2 -  \left( \rho_+ \chi_m^+ + \rho_- \chi_m^- \right) r_m y \right)  \, dx \, dy - \frac{3}{2} \frac{1}{r_m^2} \int_{\partial B_{1}^-} u_m^2 \, d\mathcal{H}^1  \\
    & = \frac{1}{r_m} \int_{\partial B_1^-} \left(  \nabla u_m \cdot \nu - \frac{3}{2}\frac{u_m}{r_m} \right) u_m \, d \mathcal{H}^1 - \int_{B_1^-}  \left( \rho_+ \chi_m^+ + \rho_- \chi_m^- \right) y \, dx \, dy,
  \end{align*}
  where in the last line we used the integration by parts identity~\eqref{integration by parts identity}. Sending $r_m \searrow 0$, and recalling the $H^1_\loc \cap C^\varepsilon$ convergence of $u_m \to u_0$, the $L^1_\loc$ convergence of $\chi_m^\pm \to \chi_0^\pm$, and the $3/2$-homogeneity of $u_0$ ensured by Lemma~\ref{vK monotonicity lemma}, we arrive at
  \begin{equation}
    \label{vK density formula}
    \mathcal{M}_\gc(u; \, 0+) = -  \int_{B_1^-} y \left( \rho_+ \chi_0^+ + \rho_- \chi_0^- \right)  \, dx \, dy.
  \end{equation}

  Consider now the four possibilities for $u_0$ in \eqref{vK blowup u}. If $u_0 = 0$, then from Lemma~\ref{vK monotonicity lemma}\ref{vK blowup domain part}, it follows that precisely one of $\chi_0^+$ and $\chi_0^-$ vanishes identically while the other takes on the constant value $1$. Thanks to the formula~\eqref{vK density formula}, we see that these correspond to cases~\ref{vK all heavier alternative} and \ref{vK all lighter alternative}. On the other hand, if $u_0 \neq 0$, then we have that $\chi_0^- = 1$ and $\chi_0^+ = 0$ on $\supp{\Stokesu}$. As in the proof of Lemma~\ref{limit of chi lemma}, we find that 
  \begin{equation}
    \label{vK u0 variational equation}
    0 = \int_{\mathbb{R}^2} \left(|\nabla u_0|^2 \nabla \cdot \phi - 2 D\phi[ \nabla u_0, \nabla u_0] + \left(  \rho_+ \chi_0^+ + \rho_- \chi_0^- \right)\nabla\cdot (y \phi) \right) \, dx \,dy.
  \end{equation}
  Since $u_0 = \Stokesu$, $\Stokesu \circ A_{\pi/6}$, or $\Stokesu \circ A_{\pi/6}^{-1}$, the set $\Omega^-(u)$ has exactly one connected component, and the outward unit normal along $\partial \Omega^-(u) \setminus \{ 0\}$ is locally constant. Reprising the argument in Lemma~\ref{limit of chi lemma}, we infer that $\rho \chi_0$ has a nonzero jump across the boundary, whence $\chi_0^\pm = \chi^\pm(u_0)$. Together, with~\eqref{vK density formula}, this gives the remaining cases~\ref{vK Stokes corner alternative} and \ref{vK corner alternative}. Note that because each of the values of the density are necessarily distinct, the blowup limits $u_0$ and $\chi_0^\pm$ are uniquely determined by $\mathcal{M}_\gc(u; \, 0+)$.

  Finally, assume that $\Omega^+(u)$ is connected in a neighborhood of $0$ and, seeking a contradiction, suppose that alternative~\ref{vK Stokes corner alternative} occurs. Let $r_m \searrow 0$ be given so that the corresponding blowup sequence $\{ u_m \}$ converges to $u_0 = \Stokesu$ and likewise $\chi_m^\pm$ converges to $\chi_0^\pm = \chi^\pm(\Stokesu)$. Since $u_0$ is strictly negative on the ball $B_r(0,-1/2)$ for $0 < r \ll 1$, the uniform convergence of $u_m \to u$ implies that each $u_m $ is strictly negative on the segment $\{ 0 \} \times (-1/2,0)$. But then we can partition $\Omega^+(u) \setminus \{0\}$ into the subset that lies strictly to the left of the $y$-axis and the subset that lies strictly to the right of the $y$-axis. Both of these are nonempty, as $\chi_0^+ = 1$ outside the support of $\Stokesu$, and open sets. Thus we have a contradiction to the connectedness of $\Omega^+(u) \setminus \{0\}$ near $0$.
\end{proof}

\subsection{The gravity current limit and von Kármán's conjecture}
\label{von Karman section}

We are now prepared to prove the main theorem on the gravity current limit along $\cmd$. It was already established in Theorem~\ref{limiting gravity current theorem} that, were overturning not to occur, then there exists a limiting variational solution to the model problem~\eqref{vK elliptic PDE} that is monotone. What remains is to apply the results of the previous subsection to resolve the contact angle for this liming solution at the point where the free surface meets the upper wall. As this is a question of independent interest, we make the effort to carry out the analysis in a more general setting; in particular, we do not assume monotonicity of $u$. This gives a partial proof of von Kármán's conjecture.

\begin{figure}
	\includegraphics{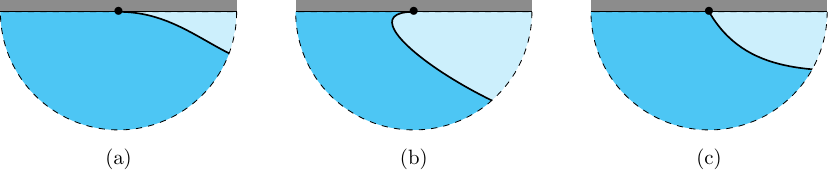}
	\caption{Possible contact angles in Theorem~\ref{contact angle theorem} assuming for definiteness that $\Omega^-(u)$, shaded darker blue, lies to the left of the free boundary and the  $\Omega^+(u)$, shaded in lighter blue, lies to the right. If $Q \neq 0$, then $\Gamma(u)$ is tangent to the upper wall, corresponding to either (a) or (b). If $Q = 0$, then either there is a cusp (b) or the contact angle is $60\degree$ as in (c).}  
\end{figure}

\begin{theorem}[Contact angle]
\label{contact angle theorem}
Let $(u,Q)$ be a variational solution to the gravity current problem~\eqref{vK elliptic PDE} satisfying~\eqref{Lipschitz constant decay}. Suppose that in a neighborhood of $0$, the free boundary $\Gamma(u)$ is a continuous injective curve with the parameterization 
\[
	\Sigma = (\Sigma_1(t), \Sigma_2(t)) \colon [0,1] \to B_1^- \cup T \quad \textrm{and} \quad \Sigma(0) = 0.
\]
Without loss of generality, assume that the region $\Omega^\pm(u)$ lies to the right of $\Gamma(u)$ near $0$.
\begin{enumerate}[label=\rm(\alph*)]
	\item \textup{(No stagnation)} \label{contact angle Q not 0 part} If $Q \neq 0$, then $\Gamma(u)$ is tangent to the wall at $0$ in that
	\[
		\lim_{t\searrow 0} \frac{\Sigma_2(t)}{\Sigma_1(t)} = 0.
	\]
	
	\item \textup{(Stagnation)} \label{contact angle Q is 0 part} If $Q = 0$, then either 
    \begin{equation}
      \label{contact angles Q is 0}
      \lim_{t\searrow 0} \frac{\Sigma_2(t)}{\Sigma_1(t)} = 0 \quad \textrm{and} \quad \pm\Sigma_1(t) < 0 \textrm{ for } 0 < t \ll1  \qquad \textrm{or} \qquad \lim_{t\searrow 0} \frac{\Sigma_2(t)}{\Sigma_1(t)} = \pm \sqrt{3}.
    \end{equation}
\end{enumerate}
\end{theorem}
\begin{remark}
This result falls short of completely resolving von Kármán's conjecture in two respects. First, we must assume the uniform Lipschitz continuity of the blowup~\eqref{vK blowup sequence Lipschitz bound}, which holds for the limiting solution furnished by Theorem~\ref{limiting gravity current theorem}, but may not hold in general. Also, the first alternative in~\ref{contact angle Q is 0 part} is that the free boundary has a cusp at the contact point, and this is not predicted by von Kármán. On the other hand, the hypotheses of Theorem~\ref{contact angle theorem} are considerably weaker than the assumptions underlying the complex analytic approach taken by von Kármán and others. The contact angle in the non-stagnation case~\ref{contact angle Q not 0 part} was not specifically discussed in~\cite{vonkarman1940engineer}, though it was obtained formally by Chandler and Trinh~\cite{chandler2018complex} by more sophisticated techniques but in the same complex variables vein. 
\end{remark}

\begin{proof}[Proof of Theorem~\ref{contact angle theorem}]
Define the set
\[
	\Theta \colonequals \left\{ \theta \in [-\pi,0] :~ \textrm{there exists $t_m \searrow 0$ with $\arg{\Sigma(t_m)} \to \theta$} \right\},
\]
where the argument $\arg{\Sigma}$ is measured with respect to the positive $x$-axis. As $\Sigma$ is continuous, the intermediate value theorem implies that $\Theta$ is convex.  We claim in fact that, if $Q \neq 0$, then $\Theta \subset \{ 0, -\pi\}$ and if $Q =0$, then $\Theta \subset \{ 0, -\pi/3, -2\pi/3, -\pi \}$; in either situation, convexity means that $\Theta$ is a singleton. The argument is based on \cite[Theorem 4]{weiss2021bernoulli} and quite similar to the proof of Lemma~\ref{flatness lemma}. 

Seeking a contradiction, suppose that there exists $\theta_0 \in \Theta \setminus \{ 0, -\pi\}$ if $Q \neq 0$ or $\theta_0 \in \Theta \setminus \{0,-\pi/3,-2\pi/3, -\pi \}$ if $Q = 0$. Let $t_m \searrow 0$ be given with $\arg{\Sigma(t_m)} \to \theta_0$.  We can therefore find a cone $\mathcal{K}_{\theta_0, \alpha}$ centered on $\{ \theta = \theta_0\}$ and with aperture $0 < 2\alpha \ll 1$ so that $\Sigma(t_m) \subset \mathcal{K}_{\theta_0, \alpha}$ for all $m \gg 1$, and 
\[
	\left\{ \begin{aligned}
		\mathcal{K}_{\theta_0, \alpha} \cap \left(  \{ \theta = 0\} \cup \{ \theta = -\pi\} \right) = \emptyset & \qquad \textrm{if } Q \neq 0 \\
		\mathcal{K}_{\theta_0, \alpha} \cap  \left( \{ \theta = 0\} \cup \{ \theta = -\tfrac{\pi}{3} \} \cup \{ \theta =-\tfrac{2\pi}{3}\} \cup \{ \theta =  -\pi\} \right) = \emptyset & \qquad \textrm{if } Q = 0.
	\end{aligned} \right.
\]
Now, set $r_m \colonequals |\Sigma(t_m)|$, and consider the corresponding blowup sequence $\{u_m\}$ given by~\eqref{blowup sequence}. Possibly passing to a subsequence, we have that $u_m \to u_0$ in $H_\loc^1$ and locally uniformly. In light of Lemma~\ref{vK density lemma}, the only possibilities are that 
\[
	u_0 = \left\{ \begin{aligned}
		\alpha y & \qquad \textrm{if } Q \neq 0 \\
		0,~\Stokesu \circ A_{\tfrac{\pi}{6}}, \textrm{ or } \Stokesu \circ A_{\tfrac{\pi}{6}}^{-1} & \qquad \textrm{if } Q = 0.
	\end{aligned} \right.
\]
Note that we have used the fact that $\Omega^+(u)$ is necessarily connected to rule out alternative~\ref{vK Stokes corner alternative}.
In all cases, then, the interior of $\mathcal{K}_{\theta,\alpha}$ is outside the support of the measure $\Delta u_0$, and hence $\Delta u_m(K) \to 0$ for any $K \subset\subset \mathcal{K}_{\theta_0,\alpha}$. On the other hand, note that the  ball $\tilde B$ of radius $\alpha$ centered at $(\cos{\theta_0}, \sin{\theta_0})$ lies inside $\mathcal{K}_{\theta_0,\alpha}$, and $\tilde B \cap \Gamma(u)$ contains a curve of length at least $2\alpha$. But then, because
\[
	\Delta u_m = |\jump{\rho} (y+Q)|^{\frac{1}{2}} d\mathcal{H}^1 \resmes \Gamma(u_m),
\]
it must be that 
\[
	\Delta u_m(\tilde B) \gtrsim 1 \qquad \textrm{as } m \to \infty, 
\]
a contradiction.

We have therefore shown that if $Q \neq 0$, then $\Theta = \{ 0\}$ or $\{ -\pi \}$; this completes the proof of part~\ref{contact angle Q not 0 part}. If $Q = 0$, the above argument shows that $\Theta = \{ -\pi/3\}$, $\{ -2\pi/3\}$, $\{0\}$, or $\{ -\pi\}$. The first two of these to correspond to alternative~\ref{vK corner alternative} for the blowup, with $u_0 = \Stokesu \circ A_{\pi/6}$ and $u_0 = \Stokesu \circ A_{\pi/6}^{-1}$, respectively. The latter two correspond to alternatives~\ref{vK all heavier alternative} and \ref{vK all lighter alternative}: if $\Omega^+(u)$ is locally to the right of $\Gamma(u)$ near $0$, then alternative~\ref{vK all heavier alternative} implies $\Theta = \{ 0 \}$ and alternative~\ref{vK all lighter alternative} implies $\Theta = \{ -\pi\}$; if $\Omega^+(u)$ is to the left of $\Gamma(u)$ locally, then the reverse is true. For simplicity, suppose that $\Omega^+(u)$ is to the right of $\Gamma(u)$. Then, the final step is simply to exclude the case $\Theta = \{0\}$. Once again we argue by contradiction: suppose $\Theta = \{0\}$. Thus $|\arg \Sigma(t)| < \pi/6$ for $0 < t \ll 1$, and so there exists a truncated $\mathcal{D}$ cone with vertex at $0$ and aperture $144\degree$ lying in the interior of $\Omega^-(u)$. Applying Theorem~\ref{oddson theorem}, we find that on $\mathcal{D}$,
\[
	|u_-(x,y) | \gtrsim |(x,y)|^{\tfrac{5}{4}} \cos{\left(\tfrac{5}{4}(\theta - \tfrac{7\pi}{5}) \right)}, \quad \theta = \arg{(x,y)}.
\]
But, as in the proof of Theorem~\ref{overturning theorem}, this gives a lower bound on the gradient decay
\begin{align*}
	\frac{1}{r^2} \int_{B_r^-} |\nabla u|^2 \, dx \, dy &  \gtrsim r^{\frac{1}{2}},
\end{align*}
which contradicts~\eqref{vK grad u bound}. It follows, then that $\Theta = \{-\pi/3\}$, $\{-2\pi/3\}$, or $\{-\pi\}$. The last of these clearly corresponds to the cusp case in~\eqref{contact angles Q is 0}. If this does not occur, then the interface meets the upper wall at an angle of either $-\pi/3$ or $-2\pi/3$. Which of these occurs is clearly determined by whether $\Omega^-(u)$ lies to the right or the left of $\Gamma(u)$ locally near $0$.
\end{proof}

The proof of Theorems~\ref{non-boussinesq gravity current theorem} and~\ref{boussinesq gravity current theorem} is now an easy corollary of Theorem~\ref{limiting gravity current theorem} and the above contact angle theorem.

\begin{proof}[Proof of Theorems~\ref{non-boussinesq gravity current theorem} and~\ref{boussinesq gravity current theorem}]
As mentioned above, we only give the argument for the depression bore case, as this can be straightforwardly adapted to give the statement in Theorem~\ref{boussinesq gravity current theorem}\ref{boussinesq elevation part}
Assume that overturning does not occur along $\cmd$ in that \eqref{depression no overturning} holds. Then, by Theorem~\ref{limiting gravity current theorem}, there exists a monotone variational solution $(u,Q)$ to the gravity current problem satisfying the uniform Lipschitz constant condition~\eqref{Lipschitz constant decay}. 

As a further consequence of Theorem~\ref{limiting gravity current theorem}, near $0$, the free boundary $\Gamma(u)$ admits the Lipschitz continuous parameterization 
\[
	\Sigma \colon t \in [0,1] \mapsto \left(\frac{t}{N}, \eta\left(\frac{t}{N}\right) \right) \in B_1^- \cup T,
\]
for some $N \gg 1$. Thus, Theorem~\ref{contact angle theorem} implies that $\eta$ is differentiable at $0$. In the non-Boussinesq setting, in which case $Q\neq 0$, by Theorem~\ref{contact angle theorem}\ref{contact angle Q not 0 part}, we must have $\eta^\prime(0) = 0$; this proves Theorem~\ref{non-boussinesq gravity current theorem}. On the other hand, in the Boussinesq case where $Q = 0$, Theorem~\ref{contact angle theorem}\ref{contact angle Q is 0 part} and the monotonicity of $\eta$ imply that $\eta^\prime(0) = -\sqrt{3}$, which completes the proof of Theorem~\ref{boussinesq gravity current theorem}\ref{boussinesq depression part}.
\end{proof}

\section*{Acknowledgments}

The research of RMC is supported in part by the NSF through DMS-2205910. The research of SW is supported in part by the NSF through DMS-2306243, and the Simons Foundation through award 960210. The authors would also like to thank Sunao Murashige for helpful conversations during the writing of this paper.
\bibliographystyle{siam}
\bibliography{projectdescription}

\end{document}